\documentclass[a4paper,english,10pt]{amsart}
\usepackage[applemac]{inputenc}
\usepackage{babel}
\usepackage{amsfonts}
\usepackage{amsmath}
\usepackage{amssymb}
\usepackage{amsthm}
\usepackage{graphicx}
\usepackage[all]{xy}
\usepackage{textcomp}
\usepackage{enumerate} 

\newtheorem{thm}{Theorem}[section]  
\newtheorem{cor}[thm]{Corollary}
\newtheorem{lem}[thm]{Lemma}
\newtheorem{defi}[thm]{Definition}
\newtheorem{prop}[thm]{Proposition}
\newtheorem{es}[thm]{Example}
\newtheorem{rem}[thm]{Remark}

\title{A Riemann-Hurwitz Formula for Skeleta in Non-Archimedean Geometry}
\author{John Welliaveetil}

\begin{document}

\maketitle

\begin{abstract}
Let $k$ be an algebraically closed non-Archimedean non trivially real valued field which is complete with respect to its valuation. 
      Let $\phi : C' \to C$ be a finite morphism between smooth projective irreducible $k$-curves.
The morphism $\phi$ induces a morphism $\phi^{\mathrm{an}} : C'^{\mathrm{an}} \to C^{\mathrm{an}}$ 
between the Berkovich analytifications of the curves. We construct
a pair of deformation retractions of $C'^{\mathrm{an}}$ and $C^{\mathrm{an}}$
which are compatible with the morphism $\phi^{\mathrm{an}}$ and
 whose images 
$\Upsilon_{C'^{\mathrm{\mathrm{an}}}}$, $\Upsilon_{C^{\mathrm{\mathrm{an}}}}$
are closed subspaces of $C'^{\mathrm{an}}$,  
$C^{\mathrm{an}}$ that are homeomorphic to finite metric graphs. 
We refer to such closed subspaces as skeleta.  
In addition, the subspaces 
$\Upsilon_{C'^{\mathrm{\mathrm{an}}}}$ and $\Upsilon_{C^{\mathrm{\mathrm{an}}}}$
are such that their complements in their respective analytifications decompose into the disjoint union of 
isomorphic copies of
Berkovich open balls. 
The skeleta can be seen as the union of vertices and edges, thus allowing us to define their genus. 
The genus of a skeleton in a curve $C$ is in fact an invariant of the curve which we call $g^{\mathrm{an}}(C)$.  
The pair of compatible deformation retractions forces the morphism $\phi^{\mathrm{an}}$ to restrict to a
map $\Upsilon_{C'^{\mathrm{\mathrm{an}}}}  \to \Upsilon_{C^{\mathrm{\mathrm{an}}}}$. 
 We study how the genus of $\Upsilon_{C'^{\mathrm{\mathrm{an}}}}$ can be 
calculated using the morphism $\phi^{\mathrm{an}}_{|\Upsilon_{C'^{\mathrm{\mathrm{an}}}}}$ and 
invariants defined on $\Upsilon_{C^{\mathrm{an}}}$. 
\end{abstract}

\tableofcontents

\emph{Acknowledgments:} This research was funded by the ERC Advanced Grant NMNAG. I would like to thank my advisor 
 Professor François Loeser for his support and guidance during this period of work.
 I am grateful as well to J\'er\^ome Poineau and Matt Baker for their suggestions and encouragement. 
  I would also like to thank Marco Macaulan, Antoine Ducros, 
 Giovanni Rosso and Yimu Yin for the discussions and comments which have been integral to the development of this article.  

\section{Introduction}
    Our goal in this paper is to define and study a topological invariant on algebraic curves defined over suitable non-Archimedean real valued fields
 that arise naturally from their analytifications. 
 
  Let $k$ be an algebraically closed, complete non-Archimedean non trivially real valued field. 
 Let $C$ be a $k$-curve.
 By $k$-curve, we mean a one dimensional connected reduced separated scheme of finite type over the field $k$. 
  It is well known that there exists a deformation retraction of $C^{\mathrm{an}}$ onto a closed subspace $\Upsilon$ 
 which is homeomorphic to a finite metric graph [\cite{berk}, Chapter 4], [\cite{HL}, Section 7]. We call such subspaces \emph{skeleta}.
The skeleton $\Upsilon$ can be decomposed into a set of vertices 
$V({\Upsilon})$ and a set of edges 
$E({\Upsilon})$. 
We define the genus of the skeleton 
$\Upsilon$ as follows. 
\begin{align*}
  g(\Upsilon) = 1 - V({\Upsilon}) + E({\Upsilon}).  
\end{align*}
  In Proposition 2.25, we show that $g(\Upsilon)$ is a well defined invariant of the curve and does not depend on the 
  retract $\Upsilon$. Let $g^{\mathrm{an}}(C) := g(\Upsilon)$ for any such $\Upsilon$. 
We study how $g^{\mathrm{an}}$ varies 
 for a finite morphism using a compatible pair of deformation retractions.

 Let $C',C$ be smooth projective irreducible $k$-curves and  
 $\phi : C' \to C$ be a finite morphism. 
The 
morphism $\phi$ induces a morphism between the respective analytifications 
which we denote $\phi^{\mathrm{an}}$. Hence we have 
\begin{align*}
  \phi^{\mathrm{an}} : C'^{\mathrm{an}} \to C^{\mathrm{an}}. 
\end{align*}
   We prove that there exists a
 pair of \emph{compatible} deformation retractions
\begin{align*}
 \psi : [0,1] \times C^{\mathrm{an}} \to C^{\mathrm{an}}
 \end{align*}
 and 
\begin{align*}
  \psi' : [0,1] \times C'^{\mathrm{an}} \to C'^{\mathrm{an}}
\end{align*}
 with the following properties. 
\begin{enumerate} 
\item The sets $\Upsilon_{C'^{\mathrm{an}}} := \psi'(1,C'^{\mathrm{an}})$ and
 $\Upsilon_{C^{\mathrm{an}}} := \psi(1,C^{\mathrm{an}})$ are closed subspaces of $C'^{\mathrm{an}}$ and $C^{\mathrm{an}}$ 
 which are homeomorphic to finite metric graphs. Furthermore, we have that 
  $\Upsilon_{C'^{\mathrm{an}}} = (\phi^{\mathrm{an}})^{-1}(\Upsilon_{C^{\mathrm{an}}})$.
\item  The analytic spaces $C'^{\mathrm{an}} \smallsetminus \Upsilon_{C'^{\mathrm{an}}}$ and 
$C^{\mathrm{an}} \smallsetminus \Upsilon_{C^{\mathrm{an}}}$ decompose into the disjoint union of 
isomorphic copies of
Berkovich open disks i.e. there exist weak semistable vertex sets (cf. Definition 2.19) 
$\mathfrak{A} \subset C^{\mathrm{an}}$ and $\mathfrak{A}' \subset C'^{\mathrm{an}}$ such that 
$\Upsilon_{C^{\mathrm{an}}} = \Sigma(C^{\mathrm{an}},\mathfrak{A})$ 
and $\Upsilon_{C'^{\mathrm{an}}} = \Sigma(C'^{\mathrm{an}},\mathfrak{A}')$. 

\item  The deformation retractions $\psi$ and $\psi'$ are said to be 
\textbf{compatible} in the sense that the following diagram is commutative. 

 \setlength{\unitlength}{1cm}
\begin{picture}(10,5)
\put(3.5,1){$[0,1] \times C^{\mathrm{an}}$}
\put(7,1){$C^{\mathrm{an}}$}
\put(3.5,3.5){$[0,1] \times C'^{\mathrm{an}}$}
\put(7,3.5){$C'^{\mathrm{an}}$}
\put(4.4,3.3){\vector(0,-1){1.75}}
\put(7.1,3.3){\vector(0,-1){1.75}}
\put(5.4,1.1){\vector(1,0){1.55}}
\put(5.4,3.6){\vector(1,0){1.55}}
\put(5.8,0.7){$\psi$}
\put(5.6,3.2){$\psi'$}
\put(4.5,2.3){$id \times \phi^{\mathrm{an}}$}
\put(7.2,2.3){$\phi^{\mathrm{an}}$} .
\end{picture}   
 \end{enumerate}
   
 In Sections 4 and 5, we study how $g^{\mathrm{an}}(C')$ and $g^{\mathrm{an}}(C)$ relate to each other 
    under the added assumption that $\phi : C' \to C$ is a finite morphism between smooth projective irreducible curves. 
 The necessary notation to make sense of the following result - Corollary 4.9 can be found in Section 4.1 and Definitions 4.6 and 4.8. \\ 
    
\noindent $\mathbf{Corollary}$ $\mathbf{4.9}$
      \emph{ Let $\phi : C' \to C$ be a finite separable morphism between 
      smooth projective irreducible curves over the field $k$. Let $g^{\mathrm{an}}(C'), g^{\mathrm{an}}(C)$ be as in Definition 2.26.
  We have the following equation. 
  \begin{align*}
     2g^{\mathrm{an}}(C') - 2 = \mathrm{deg}(\phi)(2g^{\mathrm{an}}(C) - 2) + \Sigma_{p \in C^{\mathrm{an}}} 2 i(p) g_p  + R  - \Sigma_{p \in C^{\mathrm{an}}} R^1_p.
                              \end{align*}  }

    In Section 5, we present another method to calculate the invariant $g^{\mathrm{an}}(C')$ using the existence of a 
    pair of compatible deformation retractions $\psi$ and $\psi'$ on $C^{\mathrm{an}}$ and $C'^{\mathrm{an}}$ whose images are skeleta  
    $\Upsilon_{C^{\mathrm{an}}}$ and $\Upsilon_{C'^{\mathrm{an}}}$. We assume in addition 
    that the morphism $\phi : C' \to C$ is such that the induced extension of function fields $k(C) \hookrightarrow k(C')$ is 
    Galois. By construction of $\psi'$ and $\psi$, $\phi^{\mathrm{an}}$ restricts to a morphism between the two skeleta. 
    We show that 
    the genus
of the skeleton $\Upsilon_{C'^{\mathrm{an}}}$ can be calculated using 
invariants associated to the points of $\Upsilon_{C^{\mathrm{an}}}$.
In order to do so we define a divisor $w$ on $\Upsilon_{C^\mathrm{an}}$ whose degree is 
$2g(\Upsilon_{C'^{\mathrm{an}}}) - 2$. A divisor on a finite metric graph is an element of the free abelian 
group generated by the points on the graph. 

We define $w$ as follows.  
For a point $p \in \Upsilon_{C^\mathrm{an}}$, let $w(p)$ denote the order of the divisor at $p$. 
We set
\begin{align*}
   w(p) :=  (\sum_{e_p \in E_p, p' \in (\phi^{\mathrm{an}})^{-1}(p)} l(e_p,p')) - 2n_p. 
\end{align*}
    The terms in this expression are defined as follows. Let $T_p$ denote the 
    tangent space at the point $p$ (cf. 2.2.3, 2.4.1).
  \begin{enumerate}
  \item Let $E_p \subset T_p$ be those elements for which there exists a representative
 starting from $p$ and contained completely in 
$\Upsilon_{C^{\mathrm{an}}}$.
 \item Let $p' \in C'^{\mathrm{an}}$ such that 
 $\phi^{\mathrm{an}}(p') = p$.  
The morphism $\phi^{\mathrm{an}}$ induces a map $d\phi_{p'}$ between the tangent spaces $T_{p'}$ and $T_{p}$ (cf. 2.2.3, 2.4.1).
Let $e_p \in E_p$. We define $L(e_p,p') \subset T_{p'}$ to be the preimages of $e_p$ for the map $d\phi_{p'}$.  
As $\Upsilon_{C'^{\mathrm{an}}} = (\phi^{\mathrm{an}})^{-1}(\Upsilon_{C^{\mathrm{an}}})$, any
element of $L(e_p,p')$ can be represented by a geodesic segment that is contained completely in 
$\Upsilon_{C'^{\mathrm{an}}}$. Let $l(e_p,p')$ denote the cardinality of the set $L(e_p,p')$. 
\item We define $n_p$ to be the cardinality of the set of preimages of the point $p$ i.e. 
$n_p := \mathrm{card}  \{(\phi^{\mathrm{an}})^{-1}(p)\}$.  
\end{enumerate}

     In Proposition 5.4, we show that $w$ is indeed a well defined divisor whose degree is equal to 
$2g(\Upsilon_{C'^{\mathrm{an}}}) - 2$. We then study the values $n_p$   
and $l(e_p,p')$ described above. These results are sketched below. 
\\

 We study the value $n_p$ for $p \in \Upsilon_{C^{\mathrm{an}}}$ in terms of
two invariants - $\mathrm{ram}(p)$ and $c_1(p)$ which are defined as follows. 

 Let $p \in \Upsilon_{C^{\mathrm{an}}}$. 
\begin{enumerate}
\item If $p$ is a point of type I then we set $\mathrm{ram}(p)$ to be the ramification degree 
$\mathrm{ram}(p'/p)$ for any $p' \in C'^{\mathrm{an}}$ such that 
$\phi^{\mathrm{an}}(p') = p$. As the morphism $\phi$ is Galois, $\mathrm{ram}(p)$ is well defined. If 
$p$ is not of type I then we set $\mathrm{ram}(p) := 1$.     

\item In order to define the invariant $c_1$, we introduce an equivalence relation on $C'(k)$. For $y_1,y_2 \in C'(k)$, we set 
$y_1 \sim_{c(1)} y_2$ if $\phi(y_1) = \phi(y_2)$ and $\psi'(1,y_1) = \psi'(1,y_2)$. Let $c_1(y)$ denote
 the cardinality of the equivalence class that contains 
$y$.
In Lemma 5.8, we show that the function $c_1 : C(k) \to \mathbb{Z}_{\geq 0}$ defined by setting $c_1(x) = c_1(y)$ for
any $y \in \phi^{-1}(x)$ is well defined. 
 We proceed to show that
if $x \in C(k)$ then $c_1(x)$ depends only on the point $\psi(1,x) \in \Upsilon_{C^{\mathrm{an}}}$. 
This defines $c_1 : \Upsilon_{C^{\mathrm{an}}} \to \mathbb{Z}_{\geq 0}$.    
\end{enumerate} 
    The values $c_1(p)$ and $\mathrm{ram}(p)$ can be used to calculate $n_p$ by the following relation (Proposition 5.10). 
\begin{align*}
    n_p = [k(C') : k(C)]/(c_{1}(p)\mathrm{ram}(p)).  
\end{align*}     
  \\
  
   We simplify the term $l(e_p,p')$
   which appears in the expression defining $w$. 
   Let $p \in \Upsilon_{C^{\mathrm{an}}}$ and $e_p \in E_p$.  
    In Lemma $5.12$ we show that $l(e_p,p')$ remains constant as $p'$ varies through 
the set of preimages $p' \in (\phi^{\mathrm{an}})^{-1}(p)$. We set 
$l(e_p) := l(e_p,p')$. 
   We    
   introduce the invariants - $\widetilde{\mathrm{ram}(e_p)}$ and 
   $\widetilde{\mathrm{ram}}(p)$ to study $l(e_p)$. 
   
   \begin{enumerate} 
\item Let $p \in \Upsilon_{C^{\mathrm{an}}}$.
  By definition $e_p$ is an element of the tangent space $T_p$ at $p$ 
(cf. Sections 2.2.3, 2.4.1). As $p$ is of type II, it corresponds to a discrete valuation of the 
$\tilde{k}$-function field $\widetilde{\mathcal{H}(p)}$. 
 For any $p' \in (\phi^{\mathrm{an}})^{-1}(p)$,
the extension of fields $\widetilde{\mathcal{H}(p)} \hookrightarrow \widetilde{\mathcal{H}(p')}$ can be decomposed into the composite of a 
purely inseparable extension and a Galois extension. Hence the ramification 
degree $\mathrm{ram}(e'/e_p)$ is constant as $e'$ varies through the set of preimages of $e_p$ 
at $T_{p'}$ for the map $d\phi_{p'}^{alg} : T_{p'} \to T_p$ (cf. 2.4.1). Let 
$\widetilde{\mathrm{ram}}(e_p)$ be this number. When $p$ is of type $I$, we set 
$\widetilde{\mathrm{ram}}(e_p) = \mathrm{ram}(p)$ and when $p$ is of type III, we set 
$\widetilde{\mathrm{ram}}(e_p) = c_1(p)$.  
\item For $p \in \Upsilon_{C^{\mathrm{an}}}$, we define $\widetilde{\mathrm{ram}}(p) := \Sigma_{e_p \in E_p} 1/ \widetilde{\mathrm{ram}}(e_p)$.
\end{enumerate} 

 In Proposition $5.15$, we show that if $p \in \Upsilon_{C^{\mathrm{an}}}$ and $e_p \in E_p$ then   
   \begin{align*}
      l(e_p) = [k(C'):k(C)]/(n_p \widetilde{\mathrm{ram}}(e_p)). 
   \end{align*}
 \\   
    
   The results of section 5 are compiled so that the value $2g^{\mathrm{an}}(C') - 2$ can be computed in terms of 
   the invariants $c_1, \widetilde{\mathrm{ram}}$ and $\mathrm{ram}$. 
   \\ 

\noindent $\mathbf{Theorem}$ $\mathbf{5.17}$ \emph{    Let $\phi : C' \to C$ be a finite morphism between smooth projective irreducible $k$-curves such that 
the extension of function fields $k(C) \hookrightarrow k(C')$ induced by $\phi$ is Galois. Let $g^{\mathrm{an}}(C')$ be as in Definition 2.26.
    We have that 
    \begin{align*} 
        2g^{\mathrm{an}}(C') - 2 = \mathrm{deg}(\phi) \Sigma_{p \in \Upsilon_{C^{\mathrm{an}}}} [\widetilde{\mathrm{ram}}(p) - 2/(c_1(p)\mathrm{ram}(p))].  
    \end{align*} }  

      Recently, in \cite{TEM2} and \cite{TEM3}, 
      Michael Temkin, Adina Cohen and Dmitri Trushin have obtained results on wild ramification for finite morphisms between  
  quasi-smooth Berkovich curves which bear some resemblance to results considered in this paper.  
    
\section{Preliminaries}

\subsection{ The analytification of a $k$-variety}
     Let $k$ be a non-trivially non-Archimedean real valued, algebraically closed complete field. 
 Let $|.|$ denote the valuation on $k$. By definition, $|.| : k^* \to \mathbb{R}_{> 0}$ which we extend 
 to $|.| : k \to \mathbb{R}_{\geq 0}$ by setting $|0| := 0$. Similarly, we define $\mathrm{val} : k^* \to \mathbb{R}$ by setting 
 $\mathrm{val} = -\mathrm{log}(|.|)$ and extend it to $\mathrm{val} : k \to \mathbb{R} \cup \infty$ by setting $\mathrm{val}(0) = \infty$.   
 We begin with a brief discussion of
 the analytification of a $k$-variety after which we deal with the analytification of a $k$-curve. 
  By curve, we mean a one dimensional connected reduced separated scheme of finite type over
$k$. 
      
     Let $X$ be a reduced, separated scheme of finite type over the field $k$. Associated functorially
to $X$ is a Berkovich analytic space $X^{\mathrm{an}}$. We examine this notion in more detail. 

     Let $\mathrm{k-an}$ denote the category of
$k$-analytic spaces [\cite{berk2}, 1.2.4], $\mathrm{Set}$ denote the category of sets and $\mathrm{Sch_{lft/k}}$ denote the category of schemes which
are locally 
of finite type over $k$. 
   We define a functor 
\begin{align*}        
   F &:  \mathrm{k-an} \to \mathrm{Set} \\
     &    Y \mapsto \mathrm{Hom}(Y,X)
 \end{align*}
where $\mathrm{Hom}(Y,X)$ is the set of morphisms of $k$-ringed spaces. The following theorem defines the space $X^{\mathrm{an}}$.   

\begin{thm} [\cite{berk}, 3.4.1]
   The functor $F$ is representable by a $k$-analytic space $X^{\mathrm{an}}$ and a morphism 
$\pi : X^{\mathrm{an}} \to X$. For any non-Archimedean field $K$ extending $k$, there is a 
bijection $X^{\mathrm{an}}(K) \to X(K)$. Furthermore, the map $\pi$ is surjective. 
\end{thm}

  The associated $k$-analytic space $X^{\mathrm{an}}$ is \emph{good}.
This means
 that for every point $x \in X^{\mathrm{an}}$ there exists a neighbourhood of $x$
isomorphic
to an affinoid space. 
Theorem 2.1 implies the existence of a well defined functor
\begin{align*}
   ()^{\mathrm{an}} :& Sch_{lft/k} \to  \mbox{\emph{good} } k-\mathrm{an} \\
           &     X        \mapsto  X^{\mathrm{an}}.  
\end{align*}

As a set $X^{\mathrm{an}}$ is the collection of pairs $\{(x,\eta)\}$ where $x$ is a scheme 
theoretic point of $X$ and $\eta$ is a valuation on the residue field $k(x)$ which extends the 
valuation on the field $k$. We endow this set with a topology as follows. A pre-basic open set 
is of the form $\{(x,\eta) \in U^{\mathrm{an}}||f(\eta)| \in W\}$ where $U$ is a Zariski open subset of $X$ with $f \in O_X(U)$, 
$W$ an open subspace of $\mathbb{R}_{\geq 0}$ and $|f(\eta)|$ is the evaluation of the
 image of $f$ in the residue field $k(x)$ at 
$\eta$. A basic open set is any set which is equal to the intersection of a finite number of pre-basic open sets.    
Properties of the scheme translate to properties of the associated analytic space. If $X$ is proper then $X^{\mathrm{an}}$ is 
compact and if $X$ is connected then $X^{\mathrm{an}}$ is pathwise connected. 
 
 Let $C$ be a $k$-curve.  
As above, the set $C^{\mathrm{an}}$ is the collection of pairs $\{(x,\eta)\}$ where $x$ is a scheme 
theoretic point of $C$ and $\eta$ is a rank one valuation on the residue field $k(x)$ which extends the 
valuation on the field $k$.
We divide the points of $C^{\mathrm{{an}}}$ into four groups using this description. 
 For a point 
$\mathbf{x} := (x,\mu) \in C^{\mathrm{an}}$, let $\mathcal{H}(x)$ denote the completion of the residue field $k(x)$ for the valuation $\eta$.
 Let $s(\mathbf{x})$ denote the trancendence degree of the residue field 
$\widetilde{\mathcal{H}(\mathbf{x})}$ over $\tilde{k}$ and $t(\mathbf{x})$ the rank of the group $|\mathcal{H}(\mathbf{x})^*|/|k^*|$. 
Abhyankar's inequality implies that $s(\mathbf{x}) + t(\mathbf{x}) \leq 1$. This allows us to classify points. We call $\mathbf{x}$
a type I point if it is a $k$-point of the curve. In which case, both $t(\mathbf{x}) = s(\mathbf{x}) = 0$.
 If $s(\mathbf{x}) = 1$ then $t(\mathbf{x}) = 0$ and such a point is said to be of 
type II. If $t(\mathbf{x}) = 1$ then $s(\mathbf{x}) = 0$ and such a point is considered to be of type III. Lastly,
if $t(\mathbf{x}) = d(\mathbf{x}) = 0$ and $\mathbf{x}$ is not a $k$-point of the curve then we call $\mathbf{x}$ a point 
of type IV.   

   The fact that $C$ is connected and separated implies that the analytification 
$C^{\mathrm{an}}$ is Hausdorff and pathwise connected. 
When $C$ is in addition projective, the analytification $C^{\mathrm{an}}$ compact. 
As an example we describe the analytification of the 
projective line $\mathbb{P}^{1,\mathrm{an}}_k$.
\\  
\\

 \subsection{Semistable vertex sets}

\subsubsection{$\mathbb{P}^{1,\mathrm{an}}_k$- The analytification of the projective line over $k$}

        The points of $\mathbb{P}^{1,\mathrm{an}}_k$ can be classified as follows.  
        The set of type I points are the $k$-points $\mathbb{P}^1_k(k)$ of the projective line.  
The type II, III and IV points are of the form $(\zeta, \mu)$ where $\zeta$ is the generic point of $\mathbb{P}^1_k$ and 
$\mu$ is a multiplicative norm on the function field 
$k(\mathbb{P}^1_k)$ which extends the valuation on the field $k$.
The field $k(\mathbb{P}^1_k)$ can be identified with $k(T)$ by choosing coordinates.
Hence describing the set of points of 
$\mathbb{P}^{1,an}_k \setminus \mathbb{P}^1_k(k)$ is equivalent to describing the set of multiplicative norms on the 
function field $k(T)$ which extend the valuation on $k$.
 Let $a \in \mathbb{P}^1_k(k)$ be a 
$k$-point and $B(a,r) \subset k$ denote the closed disk around 
$a$ of radius $r$ contained in $\mathbb{P}^1_k(k)$.
 We define a multiplicative norm $\eta_{a,r}$ on $k(T)$ as follows.
Let $f \in k(T)$. We set $|f(\eta_{a,r})| := \mathrm{sup}_{y \in B(a,r)} \{|f(y)|\}$. It can be checked that 
this is a multiplicative norm on the function field. If $r$ belongs to
$|k^*|$ then $(\zeta,\eta_{a,r})$ is a type II point. Otherwise $(\zeta,\eta_{a,r})$ defines a 
type III point. It can be shown that every type II and type III point is of this form. A type IV point corresponds to 
a family of nested closed disks with empty intersection. Let $J$ be a directed index set and for every 
$j \in J$, $B(a_j,r_j)$ be a closed disk around $a_j \in k$ of radius $r_j$ such that $\bigcap_{j \in J} B(a_j,r_j) = \emptyset$. 
Let $\mathfrak{E} := \{B(a_j,r_j)|j \in J\}$. We define a multiplicative norm $\eta_{\mathfrak{E}}$ on the function field as follows. For 
$f \in k(T)$, let $|f(\eta_\mathfrak{E})| := \mathrm{inf}_{j \in J} \{\mathrm{sup}_{y \in B(a_j,r_j)} |f(y)|\}$. The set of multiplicative 
norms on $k(T)$ defined in this manner corresponds to the set of type IV points in $\mathbb{P}^{1,an}_k$. 

   It is
standard practice to describe the points of $\mathbb{A}^{1,an}_k$ as the collection 
$\mathcal{M}(k[T])$
of multiplicative seminorms 
on the algebra $k[T]$ which extend the valuation of the field $k$. 
 As a set $\mathbb{P}^{1,an}_k = \mathbb{A}^{1,an}_k \cup \{\infty\}$
where $\infty \in \mathbb{P}^1_k(k)$ is the complement of the affine subspace $\mathrm{Spec}(k[T]) \subset \mathbb{P}^1_k$.
       
\subsubsection{The standard analytic domains in $\mathbb{A}^{1,\mathrm{an}}$.}
We follow the treatment in Section 2 of \cite{BPR}. 
   The topological space $\mathbb{P}^{1,an}_k$ is compact,
simply connected and Hausdorff. We now describe certain subspaces of 
$\mathbb{A}^{1,an}_k \subset \mathbb{P}^{1,an}_k$.
The tropicalization map, $\mathrm{trop} : \mathcal{M}(k[T]) = \mathbb{A}^{1,\mathrm{an}} \to \mathbb{R} \cup \infty$
is defined by $p \mapsto -\mathrm{log}|T(p)|$. Using $\mathrm{trop}$, we define certain analytic domains contained 
in $\mathbb{A}^{1,\mathrm{an}}$.  

\begin{enumerate}
\item For $r \in |k^*|$, the standard closed ball of radius $r$, $\mathbf{B}(r)$ is the set \\
$\mathrm{trop}^{-1}([-\mathrm{log}(r),\infty])$. 
The space $\mathbf{B}(r)$ is the affinoid space $\mathcal{M}(k\{r^{-1}T\})$. 
\item For $r \in |k^*|$, the standard open ball of radius $r$ denoted $\mathbf{O}(r)$ is the set $\mathrm{trop}^{-1}((-\mathrm{log}(r),\infty])$. 
The space $\mathbf{O}(r)$ is an open analytic domain contained in $\mathbb{A}^{1,\mathrm{an}}$. 
\item For $r_1,r_2 \in |k^*|$ with $r_1 \leq r_2$, the standard closed annulus $\mathbf{S}(r_1,r_2)$ of inner radius $r_1$ and outer radius 
$r_2$ is the set $\mathrm{trop}^{-1}([-\mathrm{log}(r_2),-\mathrm{log}(r_1)])$. It is the affinoid space $\mathcal{M}(k\{r_1T^{-1},r_2^{-1}T\})$. 
The (logarithmic) modulus of $\mathbf{S}(r_1,r_2)$ is defined to be the value $\mathrm{log}(r_2) - \mathrm{log}(r_1)$.  
\item For $r_1,r_2 \in |k^*|$ with $r_1 < r_2$, the standard open annulus of inner radius $r_1$ and outer radius $r_2$ denoted 
$\mathbf{S}(r_1,r_2)_+$ is the set $\mathrm{trop}^{-1}((-\mathrm{log}(r_2),-\mathrm{log}(r_1)))$. 
The (logarithmic) modulus of $\mathbf{S}(r_1,r_2)_+$ is defined to be the value $\mathrm{log}(r_2) - \mathrm{log}(r_1)$.
\item Let $r \in |k^*|$. The standard punctured 
Berkovich open disk of radius $r$ is the set $\mathbf{O}(r) \smallsetminus \{0\}$ which we denote $\mathbf{S}(0,r)_+$.
\end{enumerate} 

   We now highlight certain sub spaces of the analytic domains defined above.
   The tropicalization map defined above restricts to a map $\mathrm{trop} : \mathbf{G}_m^{\mathrm{an}} \to \mathbb{R}$.  
 We define a section $\sigma : \mathbb{R} \to \mathbf{G}_m^{\mathrm{an}}$ of the restriction of the tropicalization map to $\mathbf{G}_m^{\mathrm{an}}$
  by mapping 
 $r \in \mathbb{R}$ to the point $\eta_{0,-\mathrm{exp}(r)}$ (cf. 2.2.1). 
 
\begin{defi}
  \emph{Let $A$ be a standard open annulus, a standard closed annulus or a standard punctured open disk. The} skeleton 
  of $A$ \emph{denoted $\Sigma(A)$ is the set $\sigma(\mathbb{R}) \cap A$.}    
\end{defi} 

\begin{es}
\begin{enumerate}
\item \emph{If $r_1, r_2 \in |k^*|$ with $r_1 < r_2$ then the skeleton of the standard open annulus $\mathbf{S}(r_1,r_2)_+$ is the 
 set $\sigma((-\mathrm{log}(r_2),-\mathrm{log}(r_1)))$.} 
 \item \emph{If $r_1, r_2 \in |k^*|$ with $r_1 \leq r_2$ then the skeleton of the standard closed annulus $\mathbf{S}(r_1,r_2)$ is the 
 set $\sigma([-\mathrm{log}(r_2),-\mathrm{log}(r_1)])$.}
 \end{enumerate} 
 \end{es} 

Following \cite{BPR}, we introduce the following definition to 
distinguish those properties of the standard analytic domains above and their skeleta which are 
 invariant under isomorphism. 
 
 \begin{defi} 
  \emph{A} general closed disk (resp. general
   closed annulus, resp. general open annulus, 
   resp. general open disk, resp. general punctured Berkovich open disk) \emph{is an analytic space that is isomorphic to 
  a standard closed disk (resp. standard closed annulus, resp. standard open annulus, resp. standard open disk, resp. standard punctured open disk)}. 
 \end{defi}  

\begin{prop} [\cite{BPR}, 2.8]
  Let $A,A'$ be standard closed annuli or open annuli or punctured open disks. Let $\phi : A \to A'$ be an isomorphism. 
  Then $\Sigma(A) = \phi^{-1}(\Sigma(A'))$. 
\end{prop} 

\begin{defi}
  \emph{Let $A$ be a general open annulus (resp. general closed annulus, resp. general punctured Berkovich open disk). Let 
  $A'$ be a standard open annulus (resp. standard closed annulus, resp. standard punctured Berkovich open disk) such that 
there exists an isomorphism of analytic spaces $\phi : A \to A'$. The} skeleton $\Sigma(A)$ of $A$ \emph{is the set 
$\phi^{-1}(\Sigma(A'))$. The skeleton of $A$ is well defined by Proposition 2.5. }   
\end{defi}
 
      When $A$ is a general open or closed annulus or general punctured open disk, the skeleton $\Sigma(A)$ can be identified with a real interval 
      upto linear transformations of the form $x \mapsto x + \mathrm{val}(\alpha)$ for some 
      $\alpha \in k^*$. 
      The skeleton $\Sigma(A)$ is endowed with the structure of a metric space.   
  
  We introduce the notion of a semistable vertex set of a smooth, projective curve 
 and then generalize this notion to the case of any curve $C$ over $k$. As above, 
 given a semistable vertex set 
 we associate to it a closed subspace 
 called its skeleton.
 We then show that the homotopy type of $C^{\mathrm{an}}$ is determined by such skeleta. 
What follows is inspired by the treatment in [\cite{aminibaker}, 4.4], [\cite{HL}, Section 7] and [\cite{BPR}, Section 5]. 
     
\begin{defi}
     \emph{   Let $C$ be a smooth, projective, irreducible curve defined over the field $k$ and $C^{\mathrm{an}}$ be its analytification.}
 A semistable vertex set $\mathfrak{V}$ for $C^{\mathrm{an}}$ \emph{is a finite collection of type II points such that if $\mathcal{C}$ denotes the 
 set of connected components of $C^{\mathrm{an}} \smallsetminus \mathfrak{V}$ then there exists a finite subset $S \subset \mathcal{C}$ such that 
every $A \in S$ is isomorphic to a standard open annulus whose inner and outer radius belong to $|k^*|$ 
 and every $A \in \mathcal{C} \smallsetminus S$ is isomorphic to 
the standard open disk of unit radius $\mathbf{O}(1)$. 
Such a decomposition of the space $C^{\mathrm{an}} \smallsetminus \mathfrak{V}$ is called a} semistable decomposition. 
\end{defi}
   The existence of semistable vertex sets in $C^{\mathrm{an}}$ follows from Section 4 in \cite{BPR}. 
\\

\begin{defi}
 An abstract finite metric graph \emph{comprises the following data: A finite set of} vertices \emph{$W$, a set of}
 edges \emph{$E \subset W \times W$
which is symmetric and a function $l:E \to \mathbb{R}_{> 0} \cup \infty$ such that if $(x,y) \in E$ then
$l(x,y) = l(y,x)$}. 
\end{defi}
 The function $l$ is called the length function. 
 \begin{defi}
  A finite metric graph $G$ \emph{is the geometric realisation 
 of an abstract finite metric graph $(V,E,l)$ in which every edge $e$ can be identified with a real interval of length $l(e)$.} 
The genus $g(G)$ of the graph $G$ \emph{is defined to be the number $1 - \mathrm{card}(V) + \mathrm{card}(E)$.} 
\end{defi}
     It can be verified that if $G$ is a graph which is the geometric realisation of two abstract finite metric graphs 
     $(V,E,l)$ and $(V',E',l')$ then $1 - \mathrm{card}(V) + \mathrm{card}(E) = 1 - \mathrm{card}(V') + \mathrm{card}(E')$.

\begin{defi} 
\emph{Let $C$ be a smooth projective irreducible curve over $k$.} 
The skeleton associated to a semistable vertex set $\mathfrak{V}$ in $C^{\mathrm{an}}$ \emph{is defined to be  
the union of the skeleta of all open annuli which occur in the semistable
decomposition along with the vertex set $\mathfrak{V}$.
It is denoted $\Sigma(C^{\mathrm{an}},\mathfrak{V})$.}
\end{defi}
Let $C$ be as in the definition above.
The skeleton $\Sigma(C^{\mathrm{an}},\mathfrak{V})$ can be seen as a graph whose edges correspond 
to the closures of the skeleta of the generalized open annuli that occur in the semistable decomposition associated to 
the set $\mathfrak{V}$. 
 The space $C^{\mathrm{an}}$ is pathwise connected. Hence 
for any two points $v_1,v_2 \in \mathfrak{V}$, there exists a
path between them. From the nature of the semistable 
decomposition, each such path can be taken to be  the union of a finite number of edges of  
the skeleton $\Sigma(C^{\mathrm{an}},\mathfrak{V})$. 
It follows that $\Sigma(C^{\mathrm{an}},\mathfrak{V})$ is pathwise connected.
By Corollary 2.6 in \cite{BPR}, the modulus of the skeleton of every open annulus which occurs in the semistable
decomposition defines a length function on the set of edges of $\Sigma(C^{\mathrm{an}},\mathfrak{V})$.
The skeleton 
$\Sigma(C^{\mathrm{an}},\mathfrak{V})$ is thus a finite, metric graph.
Let $x,y \in \Sigma(C^{\mathrm{an}},\mathfrak{V})$ and $P$ be an injective path from $x$ to $y$.
The path $P$ can be seen as the finite union of injective closed paths $\bigcup_i P_i$ such that 
for every $i$, $P_i$ is contained in the closure of the skeleton of a general open annulus
 which occurs in the semistable decomposition associated to 
$\mathfrak{V}$. The skeleton of such an open annulus is a metric space,
 and its metric extends to its closure in $C^{\mathrm{an}}$. It follows that the 
 length $l(P_i)$ of the path $P_i$ is well defined.
 For instance, if $P_i$ is an injective path from $x_i$ to $y_i$ then $l(P_i) := d(x_i,y_i)$ where 
 $d$ is the metric on the skeleton of the closure of the open annulus that contains $x_i$ and $y_i$. 
  We set $l(P) := \Sigma_i l(P_i)$.  
The graph $\Sigma(C^{\mathrm{an}},\mathfrak{V})$ can be given the structure of a metric space 
by defining the distance between two points $x,y$ in $\Sigma(C^{\mathrm{an}},\mathfrak{V})$ to be 
$\mathrm{min}_{P \in \mathcal{P}(x,y)} \{l(P)\}$ where $\mathcal{P}(x,y)$ is the set of injective paths between $x$
 and $y$.  
This defines a metric on 
$\Sigma(C^{\mathrm{an}},\mathfrak{V})$. 
  
   Let $\mathfrak{V}'$ be a semistable vertex set that contains $\mathfrak{V}$. 
By Propositions 3.13 and 5.3 in \cite{BPR},
 $\Sigma(C^{\mathrm{an}},\mathfrak{V}) \subseteq \Sigma(C^{\mathrm{an}},\mathfrak{V}'$ and 
 the inclusion is an isometry. Let $\mathbf{H}_0(C^{\mathrm{an}})$ denote the set of points in $C^{\mathrm{an}}$ of 
 type II or III. By Corollary 5.1 in loc.cit., 
 
 \begin{align*}
    \mathbf{H}_0(C^{\mathrm{an}}) = \varinjlim_{\mathfrak{V}} \Sigma(C^{\mathrm{an}},\mathfrak{V}). 
 \end{align*} 
  
     The limit in the above equation is taken over the family of semistable vertex sets $\mathfrak{V}$ in $C^{\mathrm{an}}$. 
  As each of the $\Sigma(C^{\mathrm{an}},\mathfrak{V})$ are metric spaces and the inclusions in the inductive limit are isometries, 
  we have a metric on the space $\mathbf{H}_0(C^{\mathrm{an}})$ which is called its skeletal metric. By Corollary 5.7 in \cite{BPR}, 
  this metric extends in a unique way to the space $\mathbf{H}(C^{\mathrm{an}}) := C^{\mathrm{an}} \smallsetminus C(k)$.

\subsubsection{ The tangent space at a point on $C^{\mathrm{an}}$ }
      Let $C$ be a smooth projective irreducible $k$-curve.         
We begin with the notion of a geodesic segment in a metric space $T$. 

\begin{defi}
 \emph{A} geodesic segment from 
 $x$ to $y$ in a metric space 
 $T$ \emph{is the image of an isometric embedding $[a,b] \to T$
  with $[a,b] \subset \mathbb{R}$ and $a \mapsto x$, $b \mapsto y$. We often
   identify a geodesic segment with its image in $T$ and denote 
   it $[x,y]$. }
\end{defi} 

Let $p \in \mathbf{H}(C^{\mathrm{an}})$. A \emph{non trivial geodesic segment starting at $p$} is 
a geodesic segment $\alpha : [0,a] \hookrightarrow \mathbf{H}(C^{\mathrm{an}})$ such that $a > 0$ and
$\alpha(0) = p$. We say that two non trivial geodesic segments starting from $p$ are \emph{equivalent at $p$} if 
they agree in a neighborhood of zero. If $\alpha$ is a geodesic segment starting at the point $p$ then 
we refer to the equivalence class defined by $\alpha$ as its \emph{germ}. These notions can be adapted 
to the case $p \in C(k)$ as follows.  A \emph{non trivial geodesic segment starting at $p$} is 
an embedding $\alpha : [\infty,a] \hookrightarrow C^{\mathrm{an}}$ such that $a < \infty$, $\alpha(\infty) = p$,
$\alpha((\infty,a]) \subset \mathbf{H}(C^{\mathrm{an}})$ and the restriction $\alpha_{|(\infty,a]}$ is an isometry. 
As before, we say that two non trivial geodesic segments starting from $p$ are \emph{equivalent at $p$} if 
they agree in a neighborhood of $\infty$ and if $\alpha$ is a geodesic segment starting at the point $p$ then 
we refer to the equivalence class defined by $\alpha$ as its \emph{germ}.
We now define the tangent space at a point $C^{\mathrm{an}}$.  

\begin{defi}
  \emph{Let $x \in C^{\mathrm{an}}$}. 
The tangent space at $x$ \emph{denoted $T_x$ is the set of non trivial geodesic segments starting from $x$ 
upto equivalence at $x$.}   
\end{defi}
  
  Let $x \in C^{\mathrm{an}}$. The tangent space at $x$ depends solely on a neighborhood of $x$. Following Sections 4 and 5 of \cite{BPR}, 
  we introduce the concept of a simple neighborhood $U$ of $x$ and state the result which relates the tangent space $T_x$ to 
  $\pi_0(U \smallsetminus x)$. 
  
\begin{prop} (\cite{BPR}, Corollary 4.27)
Let $C$ be a smooth projective irreducible $k$-curve.
 Let $x \in C^{\mathrm{an}}$. There is a fundamental system
 of open neighborhoods $\{U_\alpha\}$ of $x$ of the following form:
 \begin{enumerate}
 \item If $x$ is a type-I or a type-IV point then the $U_{\alpha}$ are open balls.
 \item  If $x$ is a type-III point then the $U_{\alpha}$ are open annuli with $x \in \Sigma(U_\alpha)$.
 \item  If $x$ is a type-II point then $U_\alpha = \tau^{-1}(W_\alpha)$ where $W_\alpha$ is a simply-connected open neighborhood of 
 $x$ in $\Sigma(C^{\mathrm{an}},\mathfrak{V})$ for some semistable vertex set $\mathfrak{V}$ of $C^{\mathrm{an}}$ that contains $x$
  and 
  $\tau : C^{\mathrm{an}} \to C^{\mathrm{an}}$ is defined by $x \mapsto \lambda_{\Sigma(C^{\mathrm{an}},\mathfrak{V})}(1,x)$ (Proposition 2.21).
  Each $U_{\alpha} \smallsetminus \{x\}$ is a disjoint union of open balls and open annuli
 \end{enumerate} 
\end{prop} 

\begin{defi} [\cite{BPR}, Definition 4.28]
\emph{Let $C$ be a smooth projective irreducible $k$-curve. 
A neighborhood of $x \in C^{\mathrm{an}}$ of the form described in Proposition 2.13 is called a} simple neighborhood of $x$.
  \end{defi}
  
  The following proposition is a minor modification of Lemma 5.12 of \cite{BPR} to include 
  points of type I as well.  
  
\begin{prop} Let $x \in C^{\mathrm{an}}$ and let $U$ be a 
simple neighborhood of 
$x$ in $C^{\mathrm{an}}$. Then $[x,y] \mapsto y$ 
establishes a bijection $T_x \to \pi_0(U \smallsetminus \{x\})$. Moreover,
\begin{enumerate}
\item If $x$ is of type I, IV then there is only one tangent direction at $x$.
\item If $x$ has type III then there are two tangent directions at $x$.
\item If $x$ has type II then $U = \mathrm{red}^{-1}(E)$ for a smooth irreducible component $E$ of the 
special fiber of a semistable formal model $\mathfrak{C}$ of $C$ by (cf. 4.29 loc.cit.) and $T_x \tilde{\to}  \pi_0(U \smallsetminus \{x\}) \tilde{\to} E(\tilde{k})$.
\end{enumerate} 
  \end{prop} 
  
  \begin{rem}
  It should be pointed out that the notation $[x,y]$ was introduced only when $x,y \in \mathbf{H}(C^{\mathrm{an}})$. 
  When $x \in C^{\mathrm{an}}$ is a point of type I and $\alpha : [\infty,a] \hookrightarrow C^{\mathrm{an}}$ is a geodesic segment 
  starting from $x$, by $[x,\alpha(a)]$ we mean the image of the embedding $\alpha([\infty,a])$.  
   \end{rem} 
  
  Let $\rho : C' \to C$ be a finite morphism between smooth projective $k$-curves.
   If $x' \in C'^{\mathrm{an}}$ then the tangent space at $x'$ maps to the tangent space 
at $\rho(x')$ in an obvious fashion.
Suppose $x$ was not of type I. 
 Let $\lambda :[0,1] \to C'^{\mathrm{an}}$ be a representative of a 
 point on the tangent space at $x'$. 
 Let $U$ be a simple neighborhood (cf. Definition 2.14) of the point $\rho(x')$. We can find $a > 0$ such that 
 $\rho \circ \lambda((0,a])$ lies in a connected component of the space 
 $U \smallsetminus \rho(x')$. This connected component which contains $\rho \circ \lambda((0,a])$ depends only on the equivalence class
 of $\lambda$ i.e. on the element of the tangent space that is represented by $\lambda$. A similar argument can be used when $x' \in C'(k)$.
  By 2.15, we have thus defined a map
\begin{align*}
  d\rho_{x'} :& T_{x'} \to T_{\rho(x')} 
\end{align*}

\subsection{Weak semistable vertex sets} 

  Let $C$ be a smooth projective irreducible $k$-curve and let $\mathfrak{V}$ be a semistable vertex set in $C^{\mathrm{an}}$. 
  Recall that we defined the skeleton associated to $\mathfrak{V}$ and denoted it $\Sigma(C^{\mathrm{an}},\mathfrak{V})$. 
Observe that by construction, the connected components of the space $C^{\mathrm{an}} \smallsetminus \Sigma(C^{\mathrm{an}}, \mathfrak{V})$  
are isomorphic to Berkovich open balls. If $C$ was not smooth or not complete then there does not exist a finite set of type II points 
$\mathfrak{V} \subset C^{\mathrm{an}}$
such that $C^{\mathrm{an}}$ decomposes into the disjoint union of general open annuli and general open disks. However, we 
can find a finite set of points $\mathfrak{V}$ in $C^{\mathrm{an}}$ and as before define a finite graph 
$\Sigma(C^{\mathrm{an}},\mathfrak{V})$ such that the space $C^{\mathrm{an}} \smallsetminus \Sigma(C^{\mathrm{an}},\mathfrak{V})$ 
is the disjoint union of general open disks. It is with this goal in mind that we introduce the notion of weak semistable vertex sets, first 
for smooth projective irreducible curves and then for any $k$-curve.

\begin{defi} \emph{Let $C$ be a smooth projective irreducible $k$-curve.} 
  A weak semistable vertex set $\mathfrak{W}$ in $C^{\mathrm{an}}$ \emph{is defined to be 
a finite collection of points of type I or II in $C^{\mathrm{an}}$ such that if $\mathcal{C}$ denotes the 
 set of connected components of $C^{\mathrm{an}} \smallsetminus \mathfrak{W}$ then there exists a finite subset $S \subset \mathcal{C}$ such that 
every $A \in S$ is isomorphic to a standard open 
annulus or a standard punctured Berkovich open unit disk and every $A \in \mathcal{C} \smallsetminus S$ is isomorphic to 
a standard Berkovich open unit disk.} 
\end{defi}   

     As before, we define the skeleton
$\Sigma(C^{\mathrm{an}},\mathfrak{W})$ 
 associated to such a set. Let $\Sigma(C^{\mathrm{an}},\mathfrak{W})$ be the union of 
$\mathfrak{W}$ and the skeleton of
 every open annulus and punctured open disk in the decomposition of $C^{\mathrm{an}} \smallsetminus \mathfrak{W}$. 
The closed subspace $\Sigma(C^{\mathrm{an}},\mathfrak{W})$ 
is homeomorphic to a connected, finite metric graph whose length function is not necessarily finite by which we mean that there 
could be edges of length $\infty$. 

     We generalise this notion of weak semistable vertex sets to the case of curves over $k$. 

\begin{rem}Let $C$ be a $k$-curve. Let $j : C \hookrightarrow \bar{C}$ be a dense open immersion where $\bar{C}$ is projective over $k$.  
   The pair $(j,\bar{C})$ is called a compactification of $C_{/k}$.  
   Let $F := \bar{C} \smallsetminus C$. We know that $F$ is a finite set of points and 
   $C^{\mathrm{an}} = \bar{C}^{\mathrm{an}} \smallsetminus F$. 
Let $\bar{C}_i$ denote the irreducible components of $\bar{C}$ and $\bar{C'}_i$ denote their respective normalisations. 
The canonical morphisms $\bar{C'}_i \to \bar{C}$ define a morphism $\rho_{\bar{C}} : \bigcup_i \bar{C}'_i \to \bar{C}$.  
\end{rem}

        We make use of the notation introduced in Remark 2.18 in the definition that follows.

  \begin{defi} \emph{Let $C$ be a $k$-curve.
Let $\bar{C}$ be a compactification of $C$.} 
   A weak semistable vertex set $\mathfrak{W}$ for $C^{\mathrm{an}}$ \emph{ is a finite collection of points of type I or II,
 in $\bar{C}^{\mathrm{an}}$   
   such that} \emph{
   \begin{enumerate} 
    \item The set $\mathfrak{W}$ contains the set of singular points of $\bar{C}$ and the points $\bar{C} \smallsetminus C$.
   \item 
$\rho_{\bar{C}}^{-1}(\mathfrak{W}) \cap \bar{C}'_i$ is a weak semistable vertex set of the irreducible smooth projective curve $\bar{C}'_i$. 
   \end{enumerate} }
\end{defi} 

  As the above definition requires a compactification $j : C \hookrightarrow \bar{C}$, we should have said a weak 
  semistable set for the pair $(C^{\mathrm{an}}, j : C \hookrightarrow \bar{C})$. However, we abbreviate notation and refer 
  to a set $\mathfrak{W}$ which satisfies the conditions of the above definition, simply as a weak semistable vertex set for $C^{\mathrm{an}}$.  
 
 As before, we define the skeleton associated to a weak semistable vertex set for $C^{\mathrm{an}}$ as follows. 
   \begin{defi} 
     \emph{Let $C$ be a $k$-curve and $\bar{C}$ be a compactification of $C_{/k}$. Let $\mathfrak{W}$ be a 
     weak semistable vertex set for $C^{\mathrm{an}}$.
     Let $\mathfrak{W}'_i := \rho_{\bar{C}}^{-1}(\mathfrak{W}) \cap \bar{C}'^{\mathrm{an}}_i$.     
  We define} the skeleton 
 associated to $\mathfrak{W}$ to be  
  $\Sigma(C^{\mathrm{an}},\mathfrak{W}) := [\bigcup_i \rho_{\bar{C}}(\Sigma(\mathfrak{W}'_i,{\bar{C'}_i}^{\mathrm{an}}))] \cap C^{\mathrm{an}}$. 
\end{defi}
        It can be verified directly from the definition of the skeleta 
      $\Sigma(C^{\mathrm{an}},\mathfrak{W})$ associated to 
      $\mathfrak{W}$ that the space $C^{\mathrm{an}} \smallsetminus \Sigma(C^{\mathrm{an}},\mathfrak{W})$ decomposes into the disjoint union of 
      sets each of which are isomorphic
      as analytic spaces to
      the 
      Berkovich open disk $\mathbf{O}(0,1)$.

\begin{prop} Let $C$ be a $k$-curve. 
Let $\mathfrak{V} \subset \mathfrak{W}$ be weak semistable vertex sets of $C^{\mathrm{an}}$.
 There exists a deformation retraction 
 $\lambda_{\Sigma(C^{\mathrm{an}},\mathfrak{V})} : [0,1] \times C^{\mathrm{an}} \to C^{\mathrm{an}}$ whose image is the skeleton 
 $\Sigma(C^{\mathrm{an}},\mathfrak{V})$ and a deformation retraction $\lambda_{\Sigma(C^{\mathrm{an}},\mathfrak{V})}^{\mathfrak{W}}$ with image 
 $\Sigma(C^{\mathrm{an}},\mathfrak{W})$. \emph{(The image of a deformation retraction $\lambda : [0,1] \times C^{\mathrm{an}} \to C^{\mathrm{an}}$ is the 
 set $\lambda(1,C^{\mathrm{an}}) := \{\lambda(1,p) | p \in C^{\mathrm{an}} \}$).}  
 \end{prop} 
\begin{proof} 
    We begin by constructing a deformation retraction 
    $\lambda_{\Sigma(C^{\mathrm{an}},\mathfrak{V})} : [0,1] \times C^{\mathrm{an}} \to C^{\mathrm{an}}$ with 
    image $\lambda_{\Sigma(C^{\mathrm{an}},\mathfrak{V})}(1,C^{\mathrm{an}}) = \Sigma(C^{\mathrm{an}},\mathfrak{V})$.
    
    Let $\mathcal{D}$ denote the set of connected components of the space $C^{\mathrm{an}} \smallsetminus \Sigma(C^{\mathrm{an}},\mathfrak{V})$.
    By definition, each element $D \in \mathcal{D}$ is isomorphic to the Berkovich open ball $\mathbf{O}(0,1)$. 
     We fix isomorphisms $\rho_D : D \to \mathbf{O}(0,1)$ for every $D \in \mathcal{D}$. 
     
           We define $\lambda_{\Sigma(C^{\mathrm{an}},\mathfrak{V})} : [0,1] \times C^{\mathrm{an}} \to C^{\mathrm{an}}$ as follows. 
           For $p \in \Sigma(C^{\mathrm{an}},\mathfrak{V})$, we set $\lambda_{\Sigma(C^{\mathrm{an}},\mathfrak{V})}(t,p) := p$ for every $t \in [0,1]$. 
           Let $p \in C^{\mathrm{an}} \smallsetminus \Sigma(C^{\mathrm{an}},\mathfrak{V})$. There exists 
           $D \in \mathcal{D}$ such that $p \in D$. 
           Suppose, $p$ was not a type IV point. 
           By 2.1.1, there 
           exists $a \in k$ and $r \in [0,1)$ such that 
          $\rho_D(p) = \eta_{a,r}$.  
       When $r = 0$, we maintain that $\eta_{a,r}$ is the point $a$. For $t \in [0,r]$, we set 
       $\lambda_{\Sigma(C^{\mathrm{an}},\mathfrak{V})}(t,p) := p$ and when $t \in (r,1)$, let $\lambda_{\Sigma(C^{\mathrm{an}},\mathfrak{V})}(t,p) := \rho_D^{-1}(\eta_{a,t})$.  
       Lastly, let $\lambda_{\Sigma(C^{\mathrm{an}},\mathfrak{V})}(1,p) := \bar{D} \cap \Sigma(C^{\mathrm{an}},\mathfrak{V})$ where $\bar{D}$ is the 
       closure of $D$ in $C^{\mathrm{an}}$. 
       When $p$ is of type IV, $\rho_D(p)$ corresponds to the semi-norm associated to a nested sequence of closed disks 
       $\{B(x_i,u_i) \subset k \}_i$ whose intersection is empty. Let $u := \mathrm{lim}_i(u_i)$. For any $t > u$, there exists a unique closed disk 
       $B(x,t)$ such that its analytification $\mathbf{B}(x,t)$ contains the point $\rho_D(p)$. We set $\lambda_{\Sigma(C^{\mathrm{an}},\mathfrak{V})}(t,p) := p$ when $t \in [0,u]$ and 
       $\lambda_{\Sigma(C^{\mathrm{an}},\mathfrak{V})}(t,p) := \rho_D^{-1}(\eta_{x,t})$ for $t \in (u,1)$ where $\mathbf{B}(x,t) \subset \mathbf{O}(0,1)$ is the unique Berkovich closed disk
       of radius $t$
        that contains 
       the point $\rho_D(p)$. As before, let $\lambda_{\Sigma(C^{\mathrm{an}},\mathfrak{V})}(1,p) := \bar{D} \cap \Sigma(C^{\mathrm{an}},\mathfrak{V})$.

            The function $\lambda_{\Sigma(C^{\mathrm{an}},\mathfrak{V})} : [0,1] \times C^{\mathrm{an}} \to C^{\mathrm{an}}$ is well defined and the only 
            points $p$ in $C^{\mathrm{an}}$ such that $\lambda_{\Sigma(C^{\mathrm{an}},\mathfrak{V})}(t,p) = p$ for every $t \in [0,1]$ are those points which belong to 
            $\Sigma(C^{\mathrm{an}},\mathfrak{V})$. Furthermore, 
            $\lambda_{\Sigma(C^{\mathrm{an}},\mathfrak{V})}(1,C^{\mathrm{an}}) = \Sigma(C^{\mathrm{an}},\mathfrak{V})$. 
            We show that $\lambda_{\Sigma(C^{\mathrm{an}},\mathfrak{V})}$ is continuous when $[0,1] \times C^{\mathrm{an}}$ is endowed with the 
            product topology. 
             Let $W \subset C^{\mathrm{an}}$ be a connected open set. If $W$ is disjoint from 
            $\Sigma(C^{\mathrm{an}},\mathfrak{V})$ then it must be contained in some $D \in \mathcal{D}$
            and $\lambda^{-1}(W)$ does not intersect $\{1\} \times C^{\mathrm{an}} \subset [0,1] \times C^{\mathrm{an}}$. 
            We show that $\lambda_{\Sigma(C^{\mathrm{an}},\mathfrak{V})}^{-1}(W)$ is 
            open in $[0,1] \times C^{\mathrm{an}}$. 
             As $\lambda_{\Sigma(C^{\mathrm{an}},\mathfrak{V})}([0,1) \times D) \subseteq D$, the 
             map ${\lambda_{\Sigma(C^{\mathrm{an}},\mathfrak{V})}}_{|[0,1) \times D}$ defines a map 
             $\lambda' : [0,1) \times \mathbf{O}(0,1) \to \mathbf{O}(0,1)$ given by 
             $\lambda'(t,x) := \rho_D(\lambda_{\Sigma(C^{\mathrm{an}},\mathfrak{V})}(t,\rho^{-1}_D(x)))$ for $t \in [0,1)$ and $x \in \mathbf{O}(0,1)$. 
            We need only check that $(\lambda'_D)^{-1}(\rho_D(W))$ is
            open in $\mathbf{O}(0,1)$. This can be verified using the explicit description of connected open sets in $\mathbf{O}(0,1)$
             provided by Lemma 2.32. 
             Let $W$ be a connected open set which intersects $\Sigma(C^{\mathrm{an}},\mathfrak{V})$ in an open set $W'$. 
             Let $\mathcal{D'} := \{D \in \mathcal{D} | \bar{D} \cap \Sigma(C^{\mathrm{an}},\mathfrak{V}) \subset W'\}$. 
             The semistable decomposition of $C^{\mathrm{an}}$ and the connectedness of $W$ imply that it must be contained 
             in $\bigcup_{D \in \mathcal{D}'} D \cup W'$. We can decompose $W$ as the disjoint union $\bigcup_{D \in \mathcal{D}'} (W \cap D) \bigcup W'$.  
             The set $\lambda_{\Sigma(C^{\mathrm{an}},\mathfrak{V})}^{-1}(W \cap D)$ is open in $[0,1] \times C^{\mathrm{an}}$ for every $D \in \mathcal{D}'$. 
             By construction $\lambda_{\Sigma(C^{\mathrm{an}},\mathfrak{V})}^{-1}(W') = ([0,1] \times W') \cup (\{1\} \times (\bigcup_{D \in \mathcal{D}'} D))$. 
             We show that every point in $\lambda_{\Sigma(C^{\mathrm{an}},\mathfrak{V})}^{-1}(W')$ has an open neighbourhood contained in 
             $\lambda_{\Sigma(C^{\mathrm{an}},\mathfrak{V})}^{-1}(W)$. 
             It can be verified that $[0,1] \times W \subset \lambda_{\Sigma(C^{\mathrm{an}},\mathfrak{V})}^{-1}(W)$.  
             Observe that the set $[0,1] \times W \subset \lambda_{\Sigma(C^{\mathrm{an}},\mathfrak{V})}^{-1}(W)$ forms an open neighborhood of every 
             point in 
             $[0,1] \times W'$. It remains to show that every point in 
             $\{1\} \times (\bigcup_{D \in \mathcal{D}'} D)$ has an open neighbourhood contained in $\lambda_{\Sigma(C^{\mathrm{an}},\mathfrak{V})}^{-1}(W)$.
             Let $x \in D$ for some $D \in \mathcal{D}'$. 
             As $W$ is connected and $W$ is an open neighborhood of $\overline{D} \smallsetminus D$, we must have that $W \cap D$ is 
             connected as well. By Remark 2.31, we can reduce to the case when $W \cap D$ is 
             the complement in $D$ of the union of a finite number of Berkovich closed disks and points of types I and IV. 
             It follows that there exists $r_D$ such that $(r_D,1] \times D \subset \lambda_{\Sigma(C^{\mathrm{an}},\mathfrak{V})}^{-1}(W)$. The set 
             $(r_D,1] \times D$ is an open neighborhood of $(1,x)$ contained in 
             $\lambda_{\Sigma(C^{\mathrm{an}},\mathfrak{V})}^{-1}(W)$.
             We have thus proved that $\lambda_{\Sigma(C^{\mathrm{an}},\mathfrak{V})} : [0,1] \times C^{\mathrm{an}} \to C^{\mathrm{an}}$
              is continuous and hence a deformation retraction. 
             
             We now prove the second part of the proposition. 
              We define a deformation retraction 
                $\lambda_{\Sigma(C^{\mathrm{an}},\mathfrak{V})}^\mathfrak{W} : [0,1] \times C^{\mathrm{an}} \to C^{\mathrm{an}}$ with image 
                $\Sigma(C^{\mathrm{an}},\mathfrak{W})$ as follows.
               For $p \in C^{\mathrm{an}}$, let $s_p \in [0,1]$ be the smallest 
               real number such that $\lambda_{\Sigma(C^{\mathrm{an}},\mathfrak{V})}(s_p,p) \in \Sigma(C^{\mathrm{an}},\mathfrak{W})$. 
                We define $\lambda_{\Sigma(C^{\mathrm{an}},\mathfrak{V})}^\mathfrak{W}$ as follows. For $p \in C^{\mathrm{an}}$, 
                $\lambda_{\Sigma(C^{\mathrm{an}},\mathfrak{V})}^\mathfrak{W}(t,p) := \lambda_{\Sigma(C^{\mathrm{an}},\mathfrak{V})}(t,p)$ when $t \in [0,s_p]$ and 
                $\lambda_{\Sigma(C^{\mathrm{an}},\mathfrak{V})}^\mathfrak{W}(t,p) := \lambda_{\Sigma(C^{\mathrm{an}},\mathfrak{V})}(s_p,p)$ when $t \in (s_p,1]$. 
                Using arguments as before, it can be checked that $\lambda_{\Sigma(C^{\mathrm{an}},\mathfrak{V})}^\mathfrak{W}$
                 is indeed a deformation retraction with image $\Sigma(C^{\mathrm{an}},\mathfrak{W})$. 
  \end{proof} 
  
\begin{rem} \emph{Recall that if $C$ is a smooth projective $k$-curve then the space $\mathbf{H}(C) := C^{\mathrm{an}} \smallsetminus C(k)$ is a metric 
     space.
     We can hence define isometries $\alpha : [a,b] \hookrightarrow \mathbf{H}(C)$ where $[a,b] \subset \mathbb{R}$. 
      This fact can be generalized to any $k$-curve. Let $C$ be a $k$-curve. By Remark 2.18, there exists a finite set of 
     smooth projective curves $\{\bar{C}'_i\}$ such that $\bigcup_i \mathbf{H}((\bar{C}'_i)^{\mathrm{an}}) = \mathbf{H}(C^{\mathrm{an}}) := C^{\mathrm{an}} \smallsetminus C(k)$. 
     We say that a continuous function $\alpha : [a,b] \hookrightarrow \mathbf{H}(C^{\mathrm{an}})$ is an \emph{isometry} if 
     $\alpha([a,b]) \subset \mathbf{H}((\bar{C}'_i)^{\mathrm{an}})$ for some $i$ and $\alpha : [a,b] \hookrightarrow \mathbf{H}((\bar{C}'_i)^{\mathrm{an}})$
     is an isometry.}
      
     \emph{Let the notation be as in Proposition 2.21. 
      For $p \in C^{\mathrm{an}}$, let $(\lambda_{\Sigma(C^{\mathrm{an}},\mathfrak{V})}^{\mathfrak{W}})^p : [0,1] \to C^{\mathrm{an}}$ be the path defined 
 by $t \mapsto \lambda_{\Sigma(C^{\mathrm{an}},\mathfrak{V})}^{\mathfrak{W}}(t,p)$. 
  Observe that the deformation retraction
  $\lambda_{\Sigma(C^{\mathrm{an}},\mathfrak{V})}^{\mathfrak{W}}$
   is such that if $a,b \in [0,1]$ and $p \in \mathbf{H}(C^{\mathrm{an}})$ then there exists 
 $a_1 < b_1 \in [a,b]$ such that $(\lambda_{\Sigma(C^{\mathrm{an}},\mathfrak{V})}^{\mathfrak{W}})^p$ is constant on the segments $[a,a_1]$ and 
 $[b_1,b]$ and $(\lambda_{\Sigma(C^{\mathrm{an}},\mathfrak{V})}^{\mathfrak{W}})^p \circ -\mathrm{exp} : [-\mathrm{log}(a_1), -\mathrm{log}(b_1)] \to \mathbf{H}(C^{\mathrm{an}})$
 is an isometry.}
  \end{rem}

\begin{defi}
  \emph{ Let $C$ be a $k$-curve. 
    Let $\mathfrak{V} \subset \mathfrak{W}$ be weak semistable vertex sets. Let 
 $\lambda : [0,1] \times C^{\mathrm{an}} \to C^{\mathrm{an}}$ be a deformation retraction with image 
 $\Sigma(C^{\mathrm{an}},\mathfrak{V})$. 
 For $p \in C^{\mathrm{an}}$, let $s_p \in [0,1]$ be the smallest real number such that 
 $\lambda(s_p,p) \in \Sigma(C^{\mathrm{an}},\mathfrak{W})$. A deformation retraction 
 $\lambda' : [0,1] \times C^{\mathrm{an}} \to C^{\mathrm{an}}$ is said to} extend the deformation retraction 
 $\lambda$ \emph{if for $p \in C^{\mathrm{an}}$, $\lambda'(t,p) = \lambda(t,p)$ when $t \in [0,s_p]$ and 
 $\lambda'(t,p) = \lambda(s_p,p)$ when $t \in (s_p,1]$.}
    \end{defi}  

\begin{rem}
     \emph{In the proof of Proposition 2.21, we constructed deformation retractions 
     $\lambda_{\Sigma(C^{\mathrm{an}},\mathfrak{V})}$ and $\lambda_{\Sigma(C^{\mathrm{an}},\mathfrak{V})}^\mathfrak{W}$. 
     Observe that $\lambda_{\Sigma(C^{\mathrm{an}},\mathfrak{V})}^\mathfrak{W}$ is an extension of 
    $\lambda_{\Sigma(C^{\mathrm{an}},\mathfrak{V})}$ by $\mathfrak{W}$.} 
 \end{rem}

     As outlined in the introduction, we show that given a 
  weak semistable vertex set $\mathfrak{W}$ of a complete curve $C$, the genus of the 
  finite graph $\Sigma(C^{\mathrm{an}},\mathfrak{W})$ is an invariant of the curve. 

\begin{prop} 
   Let $C$ be a complete $k$-curve and $\mathfrak{W}$ be a weak semistable vertex set in $C^{\mathrm{an}}$. 
   Let $\Upsilon \subset C^{\mathrm{an}}$ be a closed subset that does not contain any points of type IV and is a finite graph. Suppose
    that there exists a deformation retraction 
   $\lambda : [0,1] \times C^{\mathrm{an}} \to C^{\mathrm{an}}$ with image 
   $\lambda(1,C^{\mathrm{an}}) = \Upsilon$. We have that 
    \begin{align*}
      g(\Sigma(C^{\mathrm{an}},\mathfrak{W})) = g(\Upsilon).  
   \end{align*}
\end{prop}
  \begin{proof}
       Let $\psi : [0,1] \times C^{\mathrm{an}} \to C^{\mathrm{an}}$ be the deformation retraction associated to 
       the set $\mathfrak{W}$ with image $\Sigma(C^{\mathrm{an}},\mathfrak{W})$ as constructed in Proposition 2.21. 
     As the graph $\Upsilon$ is finite and does not contain any points of type IV, 
     we can find a weak semistable vertex set $\mathfrak{W}'$ such that $\Sigma(C^{\mathrm{an}},\mathfrak{W}')$
     contains $\Upsilon$. We can choose $\mathfrak{W}'$ so that $\mathfrak{W} \subset \mathfrak{W}'$. 
     The restriction 
     $\lambda : [0,1] \times \Sigma(C^{\mathrm{an}},\mathfrak{W}') \to C^{\mathrm{an}}$
   implies that $\Sigma(C^{\mathrm{an}},\mathfrak{W}')$ and $\Upsilon$ are homotopy equivalent. It follows that 
   $\Sigma(C^{\mathrm{an}},\mathfrak{W})$ is homotopic to $\Upsilon$ as 
   $\lambda_{\Sigma(C^{\mathrm{an}},\mathfrak{W})} : [0,1] \times \Sigma(C^{\mathrm{an}},\mathfrak{W}') \to \Sigma(C^{\mathrm{an}},\mathfrak{W}')$
   is a deformation retraction onto $\Sigma(C^{\mathrm{an}},\mathfrak{W})$.   
   Hence $g(\Sigma(C^{\mathrm{an}},\mathfrak{W})) = g(\Upsilon)$.    
  \end{proof}

   \begin{defi}   \emph{Let $C$ be a $k$-curve and $\bar{C}$ be a compactification of $C$. 
   Let $f$ denote the cardinality of the finite set of points $\bar{C}(k) \smallsetminus C(k)$. 
   We define}  
  $g^{\mathrm{an}}(C)$ \emph{to be $g(\Sigma(\bar{C}^{\mathrm{an}},\mathfrak{W})) + f$ where 
  $\Sigma(\bar{C}^{\mathrm{an}},\mathfrak{W})$ is the skeleton
   associated to a weak semistable 
  vertex set $\mathfrak{W}$ for $C^{\mathrm{an}}$.} 
      \end{defi}
      
   It can be checked easily that this definition does not depend on the compactification of $C$ chosen. 
 Proposition 2.25 implies that $g^{\mathrm{an}}(C)$ is a well defined invariant of the $k$-curve $C$.    
 We end this section with the following proposition concerning finite graphs. 
  
 \begin{prop}
  Let $\phi : C' \to C$ be a finite morphism of smooth projective irreducible curves. Let $H \subset C^{\mathrm{an}}$ be
  a finite graph which does not contain any points of type IV. We then have that $(\phi^{\mathrm{an}})^{-1}(H)$ is a finite graph.  
 \end{prop}  
  \begin{proof}
       We may suppose at the outset that the graph $H$ is connected. 
      We show that we may reduce to the case when the extension of function fields $k(C) \hookrightarrow k(C')$ 
      induced by $\phi$ is Galois. 
    The extension of function fields $k(C) \hookrightarrow k(C')$ decomposes into a pair of extensions 
    $k(C) \hookrightarrow L$ which is separable and $L \hookrightarrow k(C')$ which is purely inseparable. 
    Let $C_1$ denote the smooth projective irreducible curve that corresponds to the function field $L$. The morphisms 
    $C' \to C_1$ and $C'^{\mathrm{an}} \to C_1^{\mathrm{an}}$ are homeomorphisms.
     It follows that if the preimage of $H$ for the morphism $C_1^{\mathrm{an}} \to C^{\mathrm{an}}$ is a finite graph 
    then $(\phi^{\mathrm{an}})^{-1}(H)$ is a finite graph as well. We may hence suppose that $k(C) \hookrightarrow k(C')$ is 
    separable. Let $L'$ denote the Galois closure of the extension $k(C) \hookrightarrow k(C')$ and $C''$ be the smooth projective 
    irreducible curve that corresponds to the function field $L'$. We have a sequence of morphisms 
    $C'' \to C' \to C$. Let $\psi : C'' \to C'$. 
    If the preimage $H''$ of $H$ for the morphism $C''^{\mathrm{an}} \to C^{\mathrm{an}}$ is a finite graph then 
    its image $\psi^{\mathrm{an}}(H'') = (\phi^{\mathrm{an}})^{-1}(H)$ 
    for the morphism $\psi^{\mathrm{an}} : C''^{\mathrm{an}} \to C'^{\mathrm{an}}$ is a finite graph. 
    Indeed, the group $G' := \mathrm{Gal}(k(C'')/k(C'))$ acts on $H''$. The graph $H''$ is defined by combinatorial data 
    i.e. a finite set of vertices $W \subset C''^{\mathrm{an}}$ and a set of edges $E \subset W \times W$ which can be realized as
    subspaces of $C''^{\mathrm{an}}$. The group $G'$ must act on the sets $W$ and $E$. It follows that the quotient of $H''$ 
    for the action of the group $G'$ can be described in terms of the $G'$-orbits in $W$ and $E$. Hence $\psi^{\mathrm{an}}(H'')$ 
    is a finite graph. We have thus reduced to the case when the morphism $\phi : C' \to C$ is Galois. 
    
     Let $\mathfrak{W}$ be a semistable vertex set in $C'^{\mathrm{an}}$. 
     Let $\mathfrak{V}$ be a weak semistable vertex set in $C^{\mathrm{an}}$ that contains 
     $\phi^{\mathrm{an}}(\mathfrak{W})$ and is such that $\Sigma(C^{\mathrm{an}},\mathfrak{V})$ contains 
     the finite graph $H$.
     We may suppose in addition that $\mathfrak{V}$ was chosen so that
      the graph $\Sigma(C^{\mathrm{an}}, \mathfrak{V})$ contains no loop edges. 
      It suffices to prove the lemma for $H = \Sigma(C^{\mathrm{an}},\mathfrak{V})$. 
     We show that there exists a finite graph $H' \subset C'^{\mathrm{an}}$ such that $\phi^{\mathrm{an}}(H') = H$. 
     Let $\mathcal{C}$ denote the connected components of the space $C^{\mathrm{an}} \smallsetminus \mathfrak{V}$. 
     As $\mathfrak{V}$ is a weak semistable vertex set, 
     there exists a finite set $S \subset \mathcal{C}$ such that every $A \in S$ is isomorphic to a standard Berkovich punctured 
     open disk of unit radius 
     or a standard open annulus and if $A \in \mathcal{C} \smallsetminus S$ then $A$ is isomorphic to a standard Berkovich open disk. 
     Likewise, let $\mathcal{C}'$ denote the set of connected components of the space $C'^{\mathrm{an}} \smallsetminus \mathfrak{W}$. 
     As $\mathfrak{W}$ is a semistable vertex set, 
     there exists a finite set $S' \subset \mathcal{C}'$ such that every $A \in S'$ is isomorphic to a standard 
     open annulus and if $A \in \mathcal{C}' \smallsetminus S'$ then $A$ is isomorphic to a standard Berkovich open disk.
    
    Let $\mathfrak{V}' := (\phi^{\mathrm{an}})^{-1}(\mathfrak{V})$.
    The morphism $\phi^{\mathrm{an}}$ is surjective, open and closed (cf. 6.1). It follows that the restriction of $\phi^{\mathrm{an}}$ to 
    $C'^{\mathrm{an}} \smallsetminus \mathfrak{V}'$ is a surjective clopen morphism onto $C^{\mathrm{an}} \smallsetminus \mathfrak{V}$. 
    Hence if $D'$ is a connected component of the space $C'^{\mathrm{an}} \smallsetminus \mathfrak{V}'$ then there exists 
    a connected component $D$ of the space $C^{\mathrm{an}} \smallsetminus \mathfrak{V}$ such that $\phi^{\mathrm{an}}$ restricts 
    to a surjective morphism from $D'$ onto $D$. 
    Let $D$ be a connected component of the space $C^{\mathrm{an}} \smallsetminus \mathfrak{V}$ which is not a general 
    Berkovich open ball. We show that if $D'$ is a connected component of the space $C'^{\mathrm{an}} \smallsetminus \mathfrak{V}'$
    such that $\phi^{\mathrm{an}}(D') = D$ then $D'$ cannot be a general Berkovich open ball.  
    Suppose that $D'$ was a general Berkovich open ball. There exists a point $q \in C'^{\mathrm{an}}$ such
     that $D' \cup \{q\}$ is compact. It follows that $D \cup \phi^{\mathrm{an}}(q)$ is 
    compact. The only elements in $\mathcal{C}$ for which this is possible are general Berkovich open disks which contradicts our assumption. 
   
    Let $D \in \mathcal{C}$ be a punctured open disk or open annulus. Let $D'$ be a connected component in 
    $C'^{\mathrm{an}} \smallsetminus \mathfrak{V}'$ such that $\phi^{\mathrm{an}}(D') = D$. There exists a finite set of points $P_{D'}$ such that 
    $D' \smallsetminus P_{D'}$ is the disjoint union of general Berkovich open disks and finitely many general open annuli   
    or punctured Berkovich disks. Let $\mathcal{C}'_{D'}$ denote the connected components of $D' \smallsetminus P_{D'}$. 
    Let $O$ be a Berkovich open disk in $D' \smallsetminus P_{D'}$. The image $\phi^{\mathrm{an}}(O)$ is a connected open 
     subset of $D$ for which there exists $p \in D$ such that $\phi^{\mathrm{an}}(O) \cup \{p\}$ is compact. It follows from Lemma 
     2.32 
     that $\phi^{\mathrm{an}}(O)$ must be a Berkovich open disk $D$ and hence
      lies in the complement of the skeleton $\Sigma(D)$. Let $S_{D'}$ be
     the set of open annuli or punctured open disks in $\mathcal{C}'_{D'}$. 
     The set $S_{D'}$ is finite. If $A \in S_{D'}$ then by Proposition 2.5 in \cite{BPR}, we must have that 
     $\phi^{\mathrm{an}}(\Sigma(A)) \subset \Sigma(D)$. We showed that if $O \in \mathcal{C}'_{D'} \smallsetminus S_{D'}$ then 
     $\phi^{\mathrm{an}}(O) \subset D \smallsetminus \Sigma(D)$. It follows that 
     $\Sigma(D) \smallsetminus \bigcup_{A \in S_{D'}} \phi^{\mathrm{an}}(\Sigma(A))$ is at most a finite 
     set of points. Let $\Sigma(D') := \bigcup_{A \in S_{D'}} \Sigma(A) \cup P_{D'}$. The set $\Sigma(D')$ is a closed connected subset of 
     $D'$. As $\phi^{\mathrm{an}}$ restricted to $D'$ is closed, its image 
     $\phi^{\mathrm{an}}(\Sigma(D')) = \{ \phi^{\mathrm{an}}(p) | p \in P_{D'}  \} \cup \bigcup_{A \in S_{D'}} \phi^{\mathrm{an}}(\Sigma(A))$          
     is a closed connected subset of $D$. Hence $\{ \phi^{\mathrm{an}}(p) | p \in P_{D'} \} \subset \Sigma(D)$ 
     and $\phi^{\mathrm{an}}(\Sigma(D')) = \Sigma(D)$. 
     
     For every $D \in S$, let $D'_D$ be a connected component of $C'^{\mathrm{an}} \smallsetminus \mathfrak{V}'$ such that 
     $\phi^{\mathrm{an}}(D'_D) = D$. We showed that there exists $\Sigma(D'_D) \subset D'_D$ such that 
     $\phi^{\mathrm{an}}(\Sigma(D'_D)) = \Sigma(D)$. If $H'_0 := \bigcup_{D \in S} \Sigma(D'_D)$ then 
     $H' := H'_0 \cup \mathfrak{V}'$ is a finite graph. We must have that $\phi^{\mathrm{an}}(H') = H$. 
     Let $G := \mathrm{Gal}(k(C')/k(C))$. The set $\bigcup_{g \in G} g(H')$ is a finite graph and since the 
     morphism $C' \to C$ is Galois, we must have that $(\phi^{\mathrm{an}})^{-1}(H) = \bigcup_{g \in G} g(H')$.  
     
  \end{proof}

 \subsection{The Non-Archimedean Poincaré-Lelong Theorem}  
The non-Archimedean \\ Poincaré-Lelong theorem is used in Sections 6.3 and 6.4. 
Our treatment follows that of \cite{BPR}. 

Let $C$ be a smooth projective irreducible $k$-curve.
Let $x \in C^{\mathrm{an}}$ be a point of type II. The field $\widetilde{\mathcal{H}(x)}$ (cf. 2.1) is an 
algebraic function field over $\tilde{k}$. Let $\tilde{C}_x$ denote the smooth projective 
$\tilde{k}$-curve that corresponds to the field $\widetilde{\mathcal{H}(x)}$. 
Let $\mathrm{Prin}(C)$ and 
$\mathrm{Prin}(\tilde{C}_x)$
 denote the 
group of principal divisors on the curves $C_{/k}$ and $\tilde{C}_{x /\tilde{k}}$ respectively. We define a map 
$\mathrm{Prin}(C) \to \mathrm{Prin}(\tilde{C}_x)$ as follows. 
Let $f \in k(C)$ be a rational function on $C$ and $c$ be any element in $k$ such that 
$|f(x)| = |c|$. This implies that $(c^{-1}f) \in \mathcal{H}(x)^0$. Let $f_x$ denote 
 the image of 
$c^{-1}f$ in $\widetilde{\mathcal{H}(x)}$. Although $f_x \in \widetilde{\mathcal{H}(x)}$
depends on the choice of $c \in k$, the divisor that $f_x$ defines on $\tilde{C}_x$ is independent 
of $c$. Hence we have a well defined map $\mathrm{Prin}(C) \to \mathrm{Prin}(\tilde{C}_x)$. It can be 
shown that this map is a homomorphism of groups. 
  
    A function $F : C^{\mathrm{an}} \to \mathbb{R}$ is piecewise affine if for any geodesic segment 
$\lambda : [a,b] \to \mathbf{H}(C^{\mathrm{an}})$, the composition $F \circ \lambda : [a,b] \to \mathbb{R}$ is 
piecewise affine. 
The outgoing slope of a piecewise affine function $F$ at a point $x \in \mathbf{H}(C^{\mathrm{an}})$ along 
a tangent direction $v \in T_x$ is defined to be 
\begin{align*} 
   \delta_vF(x) := \mathrm{Lim}_{\epsilon \to 0} (F \circ \lambda)'(\epsilon) .  
\end{align*}
       where $\lambda : [0,a] \hookrightarrow \mathbf{H}(C^{\mathrm{an}})$ is a representative of the element $v$. 
   It is evident from the definition that $\delta_vF(x)$ depends only on the equivalence class of $\lambda$ i.e. it depends only on the element $v \in T_x$. 

\begin{thm}\emph{(Non-Archimedean Poincaré-Lelong Theorem)} Let $f \in k(C)$ be a non-zero rational function on 
 the curve $C$ and $S$ denote the set of zeros and poles of $f$. 
Let $\mathfrak{W}$ be a weak semistable vertex set whose set of $k$-points is the set $S$. Let
 $\Sigma(C^{\mathrm{an}},\mathfrak{W})$ be the skeleton associated to $\mathfrak{W}$
 and 
$\lambda_{\Sigma(C^{\mathrm{an}},\mathfrak{W})} :[0,1] \times C^{\mathrm{an}} \to C^{\mathrm{an}}$ be the deformation retraction
 with image $\Sigma(C^{\mathrm{an}},\mathfrak{W})$.
We will use $\lambda_e$ to denote the morphism $\lambda_{\Sigma(C^{\mathrm{an}},\mathfrak{W})}(1,\_) : C^{\mathrm{an}} \to C^{\mathrm{an}}$ 
(cf. Proposition 2.21). If 
$F := -\mathrm{log}|f| : C^{\mathrm{an}} \smallsetminus S \to \mathbb{R}$. Then we have that 
\begin{enumerate}
 \item  $F = F \circ \lambda_e$.
 \item $F$ is piecewise affine with integer slopes and $F$ is affine on each edge of $\Sigma(C^{\mathrm{an}},\mathfrak{W})$.
 \item If $x$ is a type II point of $C^{\mathrm{an}}$ and $v$ is an element of the tangent space $T_x$, then 
         $\mathrm{ord}_{\tilde{v}}(f_x) := \delta_vF(x)$ defines a discrete valuation 
          $\mathrm{ord}_{\tilde{v}}$
 on the $\tilde{k}$-function field $\tilde{k}(\tilde{C}_x)$. 
  \item If $x \in C^{\mathrm{an}}$ is of type II or III then $\sum_{v \in T_x} \delta_vF(x) = 0$.  
 \item Let $x \in S$, $c$ be the ray in $\Sigma(C^{\mathrm{an}},\mathfrak{W})$ whose closure 
         in $C^{\mathrm{an}}$ contains $x$ and $y \in \mathfrak{W}$ the other end point of $e$. If $v \in T_y$ 
         is that element of the tangent space $T_y$ for which $c$ is a representative then
  $\delta_vF(y) = ord_x(f)$.  
\end{enumerate}
   \end{thm}

\subsubsection{ An alternate description of the tangent space at a point $x$ of type II}
   Let $x \in C^{\mathrm{an}}$ be a point of type II. We define the algebraic tangent space at a point of type II and show how this 
notion reconciles nicely with the definition we introduced above. Recall that the field
$\widetilde{\mathcal{H}(x)}$ is of transcendence degree $1$ over $\tilde{k}$ and uniquely associated to this 
$\tilde{k}$-function field is a smooth, projective $\tilde{k}$-curve which is denoted $\tilde{C}_x$.   

 \begin{defi}
     The algebraic tangent space at $x$ \emph{denoted $T^\mathrm{alg}_x$ is the set of closed points of the 
curve $\tilde{C}_x$}. 
 \end{defi}
 
   We now write out a map $B: T_x \to T^\mathrm{alg}_x$. The closed points 
of the $\tilde{k}$-curve $\tilde{C}_x$ correspond to discrete valuations on the 
field $\widetilde{\mathcal{H}(x)}$. Given a germ $e_x \in T_x$ and $f \in \widetilde{\mathcal{H}(x)}$ there 
exists $g \in \mathcal{H}(x)$ such that $|g(x)| = 1$ and $\tilde{g} = f$. Let 
$B(e_x)(f)$ be the slope of the function $-\mathrm{log}|g|$ along the germ $e_x$ directed outwards.
By the Non-Archimedean Poincaré-Lelong Theorem,
$B(e_x)$ defines a discrete valuation on the function field $\widetilde{\mathcal{H}(x)}$ i.e. a closed point of the 
curve $\tilde{C}_x$. 
The map $B$ is a well defined bijection. 

   Let $C'$ be a smooth, projective, irreducible curve over the field $k$ and $\rho : C' \to C$ a finite morphism.
    If $x'$ is a preimage of the point $x$ then it must be of type II as well. The inclusion of 
non-Archimedean valued complete fields $\mathcal{H}(x) \hookrightarrow \mathcal{H}(x')$ 
induces an extension of $\tilde{k}$-function fields 
$\widetilde{\mathcal{H}(x)} \hookrightarrow \widetilde{\mathcal{H}(x')}$. This defines a morphism 
$d\rho_{x'}^{\mathrm{alg}} : T^\mathrm{alg}_{x'} \to T^\mathrm{alg}_x$
between the algebraic tangent space at $x'$ and the algebraic tangent space at $x$. 
Recall that we have in addition a map $d\rho_{x'} : T_{x'} \to T_x$. These maps are 
compatible in the sense that the following diagram is commutative. 

 \setlength{\unitlength}{1cm}
\begin{picture}(10,5)
\put(4,1){$T^\mathrm{alg}_{x'}$}
\put(6.9,1){$T^\mathrm{alg}_x$}
\put(4.3,3.5){$T_{x'}$}
\put(7,3.5){$T_x$}
\put(4.4,3.3){\vector(0,-1){1.75}}
\put(7.1,3.3){\vector(0,-1){1.75}}
\put(4.8,1.1){\vector(1,0){2}}
\put(4.8,3.6){\vector(1,0){2}}
\put(5.6,0.7){$d\rho_{x'}^\mathrm{alg}$}
\put(5.6,3.2){$d\rho_{x'}$}
\put(4.5,2.3){$B$}
\put(7.2,2.3){$B$} 
\end{picture}

 \subsection{Continuity of lifts}      
       
      Let $\phi : C' \to C$ be a finite morphism between irreducible smooth projective curves.  
      In Section 3, we construct a pair of deformation retractions 
     $\lambda' : [0,1] \times C'^{\mathrm{an}} \to C'^{\mathrm{an}}$ and 
     $\lambda : [0,1] \times C^{\mathrm{an}} \to C^{\mathrm{an}}$ which are compatible 
     for the morphism $\phi^{\mathrm{an}}$. Our method of proof is to first construct a suitable 
     deformation retraction $\lambda$ on $C^{\mathrm{an}}$ and then lift it to a function 
     $\lambda' : [0,1] \times C'^{\mathrm{an}} \to C'^{\mathrm{an}}$ such that for 
     every $q \in C'^{\mathrm{an}}$, the map ${\lambda'}^q : [0,1] \to C'^{\mathrm{an}}$ 
     defined by setting ${\lambda'}^{q}(t) = \lambda'(t,q)$ 
     is continuous. Our goal in this section is to show that 
     given a deformation retraction $\lambda$ and 
     a lift $\lambda'$ as above, 
     the 
     function $\lambda'$ is continuous.      
      
 \begin{lem} 
  Let $\phi : C' \to C$ be a finite morphism between $k$-curves and suppose in addition that $C$ is normal. Let 
  $\mathfrak{V} \subset C^{\mathrm{an}}$ be a weak semistable vertex set and suppose $\mathfrak{V}'$ is a weak 
  semistable vertex set for $C'^{\mathrm{an}}$ such that $\Sigma(C'^{\mathrm{an}},\mathfrak{V}') = (\phi^{\mathrm{an}})^{-1}(\Sigma(C^{\mathrm{an}},\mathfrak{V}))$.
   Let $\mathcal{D}$ denote the set of connected components of the space 
  $C^{\mathrm{an}} \smallsetminus \Sigma(C^{\mathrm{an}},\mathfrak{V})$ and likewise, $\mathcal{D}'$
   denote the set of connected components of the space $C'^{\mathrm{an}} \smallsetminus \Sigma(C'^{\mathrm{an}},\mathfrak{V}')$. 
If $D' \in \mathcal{D}'$ then there exists $D \in \mathcal{D}$ such that $\phi^{\mathrm{an}}(D') = D$. Furthermore, the 
restriction $\phi^{\mathrm{an}}_{|D'} : D' \to D$ is both closed and open. 
 \end{lem}      
     \begin{proof}
        Let $D' \in \mathcal{D}'$. As $D'$ is connected and $\phi^{\mathrm{an}}$ is continuous, we must have that 
        $\phi^{\mathrm{an}}(D')$ is connected. Furthermore, $\phi^{\mathrm{an}}(D') \subset C^{\mathrm{an}} \smallsetminus \Sigma(C^{\mathrm{an}},\mathfrak{V})$
        because
         $\Sigma(C'^{\mathrm{an}},\mathfrak{V}') = (\phi^{\mathrm{an}})^{-1}(\Sigma(C^{\mathrm{an}},\mathfrak{V}))$. 
       The open subspace $C^{\mathrm{an}} \smallsetminus \Sigma(C^{\mathrm{an}},\mathfrak{V})$ decomposes into 
       the disjoint union $\bigcup_{A \in \mathcal{D}} A$. It follows that there exists $D \in \mathcal{D}$ such that 
       $\phi^{\mathrm{an}}(D') \subset D$.   
         
            Let $A'$ be a connected component of the space $(\phi^{\mathrm{an}})^{-1}(D)$ that contains $D'$.
            As $A' \subset C'^{\mathrm{an}} \smallsetminus \Sigma(\mathfrak{W},C'^{\mathrm{an}})$ and 
         $C'^{\mathrm{an}} \smallsetminus \Sigma(\mathfrak{W},C'^{\mathrm{an}})$ decomposes into the disjoint union 
         $\bigcup_{U \in \mathcal{D}'} U$, we must have that $D' = A'$. The morphism $\phi^{\mathrm{an}}$ is a finite morphism and 
         hence closed. By Lemma 6.1, it is open as well. 
         It follows that $\phi^{\mathrm{an}}$ restricts to a morphism $D' \to D$ which is both open and closed. 
         As $D$ is connected, we must have that $\phi^{\mathrm{an}}(D') = D$. 
                          \end{proof}   
       
 \begin{rem}      
     \emph{Recall that for $a \in k$ such that $|a| < 1$ and $r < 1$, we used $\mathbf{O}(a,r)$ to denote the Berkovich 
  open disk around $a$ of radius $r$ and 
  $\mathbf{B}(a,r)$ to denote the Berkovich closed disk around 
  $a$ of radius $r$. 
     By Proposition 1.6 in \cite{BR}, a basis $\mathcal{B}$ for the open sets of $\mathbf{O}(0,1)$ is given by the
sets
  \begin{align*}   
  \mathbf{O}(a,r),\mbox{ } \mathbf{O}(a,r) \smallsetminus \bigcup_{i \in I} X_i, \mbox{ } \mathbf{O}(0,1) \smallsetminus  \bigcup_{i \in I} X_i 
    \end{align*}
where $I$ ranges over finite index sets, $a$ ranges over $O(0,1)$, where $r \in (0,1)$ and $X_i$ is either a Berkovich closed sub disk 
    of the form $\mathbf{B}(a_i,r_i)$ with $a_i \in O(0,1)$ and $r_i \in [0,1)$ or a point of type I or IV. 
     We classify the elements of this basis by referring to Berkovich open sub disks as sets of form 1, 
    Berkovich open disks from 
  which a finite number of closed disks have been removed as sets of form 2 and
   the complement of the union of a finite number of closed sub disks as sets of form 3.}
 \end{rem}

 \begin{lem} 
    Let $U \subset \mathbf{O}(0,1)$ be a connected open set. Then $U$ is of the form 
    $\mathbf{O}(a,r) \smallsetminus \bigcup_{j \in J} X_j$ where $a \in O(0,1)$, $r \in (0,1]$, 
    $J$ is an index set and the $X_j$ are Berkovich closed disks or points of type I or IV. In addition, the 
    $X_j$ can be taken to be disjoint from each other. 
 \end{lem}   
    Note that we do not claim every set of the form $\mathbf{O}(a,r) \smallsetminus \bigcup_{j \in J} X_j$
    is open, as this is false. For instance the set $\mathbf{H}(\mathbf{O}(0,1)) := \mathbf{O}(0,1) \smallsetminus O(0,1)$ is 
   not open in $\mathbf{O}(0,1)$ as $O(0,1)$ is dense in $\mathbf{O}(0,1)$. 
\begin{proof}
    Let $x \in U$ be a point of type I. We define $r_x$ to be the supremum of the set $\{r \in [0,1) | \eta_{x,r} \in U\}$ where the 
    notation $\eta_{x,r}$ was introduced in 2.2.1. We maintain that $\eta_{x,0} := 0$. 
    The point $\eta_{x,r_x}$ either does not belong to $\mathbf{O}(0,1)$ or belongs to the set $\bar{U} \smallsetminus U$ where 
    $\bar{U}$ denotes the closure of $U$ in $\mathbf{O}(0,1)$. Indeed, the interval $[0,1)$ can be identified with the 
    set $\{\eta_{x,r} | r \in [0,1)\}$ by $r \mapsto \eta_{x,r}$ and this map is a homeomorphism when $[0,1)$ is equipped with the linear topology 
    and $\{\eta_{x,r} | r \in [0,1)\} \subset \mathbf{O}(0,1)$ is given the subspace topology.
    We claim that if $x,y \in U$ are points 
    of type I and if $\eta_{x,r_x} \in \mathbf{O}(0,1)$ then $\eta_{x,r_x} = \eta_{y,r_y}$. 
   Since $\mathbf{O}(0,1) \smallsetminus \{\eta_{y,r_y}\}$ decomposes into the disjoint union of open sets, we must have that 
   $U \subset \mathbf{O}(y,r_y)$. Hence $r_x \leq r_y$ which implies that $r_x = r_y$. Furthermore, since 
   $x \in \mathbf{O}(y,r_y)$, we must have that $\eta_{x,r_x} = \eta_{y,r_y}$. 
   Likewise, it can be shown that if $x,y \in U$ are points of type I and $r_x = 1$ then 
   $r_y = 1$. Let $\eta := \eta_{x,r_x}$ where $x \in U$ is a point of type I.  
   
       Let $J := \bar{U} \smallsetminus \{U \cup \eta\}$ where $\bar{U}$ is closure of $U$ in $\mathbf{O}(0,1)$. 
    If $x \in J$ is not a point of type IV then it can be seen as $\eta_{a,r}$ for some $a \in O(0,1)$ and 
    $r \in [0,1)$. 
    Let $X_j$ denote the Berkovich closed disk $\mathbf{B}(a,r)$. 
    Let $y \in \mathbf{B}(a,r)$ be a point of type I. If $y \in U$ then we must have that 
    $U \subset \mathbf{O}(y,r)$ since  
    $\mathbf{O}(0,1) \smallsetminus \eta_{a,r}$ is the disjoint union of open sets
    and $U$ is connected. As $r_y$ is constant as $y$ varies along the set of type I points in $U$, it can be deduced 
    that $r_y = r$ and that $x = \eta_{y,r_y}$ which is not possible. Hence $U \subset \mathbf{O}(0,1) \smallsetminus X_j$.  
    
     Let $x \in U$ be a point of type I. 
    If $j \in J$ then since $j \neq \eta$ lies in the closure of $U$, we must have that $j \in \mathbf{O}(x,r_x)$ and 
    $X_j \subset \mathbf{O}(x,r_x)$. 
     It can be verified using the results above that 
    $U = \mathbf{O}(x,r_x) \smallsetminus \bigcup_{j \in J} X_j$ 
     which concludes the proof.           
    \end{proof}

 \begin{lem}
    Let $f : \mathbf{O}(0,1) \to \mathbf{O}(0,1)$ be a surjective, open and closed continuous function.
  \begin{enumerate}
  \item If $U \in \mathcal{B}$ is an open set of form $i$ where $i$ is 1 or 3 then
          $f(U)$ is of form $i$ as well. If $U \in \mathcal{B}$ is of form 2 then $f(U)$ is of 
          form 1 or 2. If we suppose in addition that $f$ is bijective and $U$ is of form 2  
      then $f(U)$ is also of form 2.
  \item If $Y \subset \mathbf{O}(0,1)$ is a Berkovich closed disk then $f(Y)$ is a Berkovich closed disk.
 \end{enumerate} 
 \end{lem}   
\begin{proof}
  \begin{enumerate} 
  \item  Let $U$ be a Berkovich open sub ball. 
  We show that $f(U)$ is a Berkovich open ball. 
  The closure $\overline{U}$ is a compact 
  subspace of $\mathbf{O}(0,1)$.
  Observe that $\overline{U} \smallsetminus U$ is a single point which we denote $p$. 
   As $f$ is continuous, $f(\overline{U}) = f(U) \cup \{f(p)\}$ must be compact as well. 
    As $U$ is connected, $f(U)$ must be a connected open set as well. By Lemma 2.32, it suffices to verify which 
    connected open sets in $\mathbf{O}(0,1)$
     are such that they can be compactified by adding a single point of $\mathbf{O}(0,1)$. It can be checked by hand that the only possibility 
  for $f(U)$ is a Berkovich open ball contained in $\mathbf{O}(0,1)$. 
 
         Let $\{D_1,\ldots,D_m\}$ be a finite number of Berkovich closed disks or points of types I or IV
          in $\mathbf{O}(0,1)$ and 
         $U := \mathbf{O}(0,1) \smallsetminus (\bigcup_i D_i)$.
         The set $U$ is of form 3 and we show that $f(U)$ is also of form 3. 
          As $U$ is connected, the image $f(U)$ is a connected open set as well. 
         By Lemma 2.32, $f(U)$ must be of the form 
         $\mathbf{O}(a,r) \smallsetminus \bigcup_{j \in J} X_j$ where $a \in O(0,1)$, $r \in (0,1]$, 
    $J$ is an index set and the $X_j$ are Berkovich closed disks or points of type I or IV. In addition, the $X_j$ can be taken to be disjoint.  
    We claim that $r = 1$. 
         Suppose $r < 1$. Then the closure of $f(U)$ in $\mathbf{O}(0,1)$ denoted $\overline{f(U)}$ 
        is compact. The $D_i$ are Berkovich closed disks or points of types I or IV and hence compact. As a result, the
         $f(D_i)$ are compact subsets of $\mathbf{O}(0,1)$. The surjectivity of $f$ implies that 
         $\mathbf{O}(0,1) = \bigcup_i f(D_i) \cup (\overline{f(U)})$ is compact. This is a contradiction and we must hence have 
         that $r =1$. We claim that the index set $J$ is finite. There exists a finite set of points
         $S := \{p_1,\ldots,p_r\}$ in $\mathbf{O}(0,1)$ such that 
         $U \cup S$ is closed in $\mathbf{O}(0,1)$. As the map $f$ is closed, $f(U \cup S)$ is closed in $\mathbf{O}(0,1)$. 
         Uniquely associated to each $j \in J$ is an element $x_j \in \mathbf{O}(0,1)$ that lies in the closure of $f(U)$. 
         Hence we must have that the index set $J$ is finite. This implies that $f(U) \in \mathcal{B}$ and is of form 3.

   Let $U \in \mathcal{B}$ be an open set of form 2. As $U$ is contained in a Berkovich open disk $U'$, we must have that its image 
   is a connected open set which is contained in the Berkovich open disk $f(U')$ strictly contained in $\mathbf{O}(0,1)$. 
   Repeating the arguments above, it can be shown that $f(U) \in \mathcal{B}$ is of form either 1 or 2.  
   Suppose that $f$ is bijective and let $U = U' \smallsetminus (\bigcup_i D_i)$ where 
   the $D_i \subset U'$ are Berkovich closed sub disks or points of type I or IV.
    It follows that $f(U) = f(U') \smallsetminus \bigcup_i f(D_i)$.
   As the only connected open subsets of $\mathbf{O}(0,1)$ which are open balls from which a
    finite number of closed subspaces have been removed are of form 2,
  we conclude that $f(U)$ is of 
   form 2.    
   
 \item  Let $Y$ be a closed disk of radius $r$ in $O(0,1)$. 
     The closed disk $Y$ can be seen as the union of a family of Berkovich open sub disks in $O(0,1)$ of radius $r$  
     and a point. Hence we can write $Y = (\bigcup_{i \in I} V_i) \cup \{q\}$ where $I$ is an index set, the $V_i$ are Berkovich 
     open disks of radius $r$ and $q$ is the unique point such that for every $i$, $V_i \cup \{q\}$ is compact. 
     It follows that $f(V_i \cup \{q\}) = f(V_i) \cup f(q)$ is a compact set. 
    Let $U_i := f(V_i)$ and $p = f(q)$. By part (1), the $U_i$ are Berkovich open balls contained in $\mathbf{O}(0,1)$. 
    The point $p$ of $\mathbf{O}(0,1)$
   such that $U_i \cup \{p\}$ is compact is uniquely determined by $U_i$. Furthermore, this 
    point $p$ determines the radius of the Berkovich open ball $U_i$. It follows that the radii of the Berkovich open balls $U_i$ 
    are the same. Let $t$ be the radius of Berkovich open balls $U_i$.
     Let $X$ denote the Berkovich closed ball corresponding to the point $p$. The radius of $X$ is $t$. 
     The tangent space at $p$ is in bijection with the set of Berkovich open balls of radius $t$ contained in 
     $X$ and the open annulus $O(0,1) \smallsetminus X$. 
     Likewise, the tangent space at $q$ is the set of Berkovich open balls $V_i$ of radius $r$ contained in $Y$ and 
     the open annulus $O(0,1) \smallsetminus Y$. 
     The tangent space at $q$ surjects onto the tangent space at $p$. Furthermore, for every $i \in I$, the image $U_i$ of the
     Berkovich open ball $V_i$ is a Berkovich open ball of radius $t$ contained in $X$.
    Hence 
    if $U \subset X$ is a Berkovich open disk of radius $t$ then there exists $j \in I$ such that $f(V_j) = U$. 
     It follows that $f(Y) = X$.  
  \end{enumerate} 
\end{proof} 

\begin{lem} 
   Let $a \in k$ and $r$ be a positive real number belonging to $|k^*|$. 
 Let $\mathbf{B}(a,r)$ denote the Berkovich closed disk around $a$ of radius $r$ and let 
 $\sigma : \mathbf{B}(a,r) \to \mathbf{B}(a,r)$ be an automorphism of analytic spaces. Suppose $W \subset \mathbf{B}(a,r)$ is a Berkovich open disk 
 of radius $0 < s < r$ then $\sigma(W)$ is a Berkovich open disk of radius $s$. Likewise, if $W$ is a Berkovich closed disk of radius $0 \leq s < r$ then
 $\sigma(W)$ is a Berkovich closed disk of radius $s$.      
\end{lem} 
\begin{proof} 
  As $r \in |k^*|$, we can suppose that $r = 1$, $a = 0$. 
We choose coordinates and write $\mathbf{B}(0,1) = \mathcal{M}(k\{T\})$.  
   The automorphism $\sigma$ induces an automorphism 
  $\sigma' : k\{T\} \to k\{T\}$ of affinoid algebras. By the Weierstrass preparation theorem, we must have that 
  $\sigma'(T) = f(T)u$ where $f(T) = c(T - a_1)(T - a_r)$ is a polynomial in $T$ with $c \in k^*$ and $u$ is an invertible element in 
  $k\{T\}$. As $\sigma'$ is an automorphism, we must have that $|c| = 1$ and  
  that $f(T)$ is of degree $1$. It follows that $f(T) = c(T - b)$ for some $b \in B(0,1)$. 
   We now show that if $W$ is a Berkovich open sub ball around a point $x \in B(0,1)$ of radius $s$ then 
   $\sigma(W)$ is a Berkovich open ball around $\sigma(x)$ of radius $s$. By definition, $W = \{p \in \mathbf{B}(0,1)| |(T-x)(p)| < s \}$. 
   It follows that $\sigma(W) = \{q \in \mathbf{B}(0,1) | |(1/(cu)T) - (x - b)| < s\}$. As $|cu| = 1$, it can be checked that the claim has been verified.
   The proof can be repeated when $W$ is the closed disk $\{p \in \mathbf{B}(0,1)| |(T-x)(p)| \leq s \}$.  
\end{proof} 
 
   We make use of the following notation in the statements that follow. Let $\phi : C' \to C$ be a finite morphism between 
   smooth projective curves. Let $\mathfrak{W}', \mathfrak{W}$ be weak semistable vertex sets for $C'^{\mathrm{an}}$ and 
   $C^{\mathrm{an}}$ respectively such that 
   $\Sigma(C'^{\mathrm{an}},\mathfrak{W}') = (\phi^{\mathrm{an}})^{-1}(\Sigma(C^{\mathrm{an}},\mathfrak{W}))$. 
   Let $\mathcal{D}'$ denote the set of connected components of the
space $C'^{\mathrm{an}} \smallsetminus \Sigma(C'^{\mathrm{an}},\mathfrak{W}')$.
If $D' \in \mathcal{D}'$ then $D'$ is isomorphic to the Berkovich unit ball $\mathbf{O}(0,1)$ and we identify $D'$ via this isomorphism.
Likewise, let $\mathcal{D}$ denote the set of connected components of the space 
$C^{\mathrm{an}} \smallsetminus \Sigma(C^{\mathrm{an}},\mathfrak{W})$.
 
 \begin{lem}
  Let $\phi : C' \to C$ be a finite morphism between irreducible projective smooth $k$-curves.
    Let 
  $\mathfrak{V} \subset \mathfrak{W}$ be weak semistable vertex sets of $C^{\mathrm{an}}$. Let
 $\mathfrak{W}' \subset C'^{\mathrm{an}}$ be a weak semistable 
  vertex set such that $\Sigma(C'^{\mathrm{an}},\mathfrak{W}') = (\phi^{\mathrm{an}})^{-1}(\Sigma(C^{\mathrm{an}},\mathfrak{W}))$. 
 Let $\lambda^\mathfrak{W}_{\Sigma(C^{\mathrm{an}},\mathfrak{V})} : [0,1] \times C^{\mathrm{an}} \to C^{\mathrm{an}}$ be the deformation retraction
 constructed in Proposition 2.21 
  whose image is $\Sigma(C^{\mathrm{an}},\mathfrak{W})$.
   Let $\lambda' : [0,1] \times C'^{\mathrm{an}} \to C'^{\mathrm{an}}$ be a 
  function such that 
   for every $q \in C'^{\mathrm{an}}$, the path ${\lambda'}^q : [0,1] \to C'^{\mathrm{an}}$  
 defined by $t \mapsto {\lambda'}(t,q)$ is continuous and 
 $\lambda'^{q}(1) \in \Sigma(C'^{\mathrm{an}},\mathfrak{W}')$. 
  Furthermore, if 
 $\phi^{\mathrm{an}}(q) = p$ then $\lambda'^q$ is the unique path starting from 
 $q$ such that $\phi^{\mathrm{an}} \circ \lambda'^q = (\lambda^\mathfrak{W}_{\Sigma(C^{\mathrm{an}},\mathfrak{V})})^{p}$.     
Let $D' \in \mathcal{D}'$ and $x_1,x_2 \in D'$. There exists $r \in [0,1]$ such that
$\lambda'^{x_1}(r), \lambda'^{x_2}(r) \in D'$ and 
 $\lambda'^{x_1}_{|[r,1]} = \lambda'^{x_2}_{|[r,1]}$.
 \end{lem} 
  We simplify notation and write $\lambda$ in place of $\lambda^\mathfrak{W}_{\Sigma(C^{\mathrm{an}},\mathfrak{V})}$. 
 \begin{proof}    
  Recall that when constructing the deformation retraction $\lambda$, 
  we identified every $D \in \mathcal{D}$ with the standard Berkovich open unit disk. Let $D'$ be as in the statement of the lemma. 
  By Lemma 2.30, there exists $D \in \mathcal{D}$ such that $\phi^{\mathrm{an}}(D') = D$.
    If $\overline{D}$ is the closure of $D$ in $C^{\mathrm{an}}$ then $\overline{D} \smallsetminus D$ is a single point $\eta$. 
  By construction of the deformation retraction $\lambda$ there exists $s \in [0,1]$ such that 
  for every $y \in D$, $\lambda(s,y) = \eta$ and 
  the restriction $[0,s) \times D \to D$ given by $(t,x) \mapsto \lambda(t,x)$ is well defined. 
  We must hence have that the restriction $\lambda' : [0,s) \times D' \to D'$ is well defined. 
  If $\overline{D'}$ is the closure of $D'$ in $C'^{\mathrm{an}}$ then $\overline{D'} \smallsetminus D'$ is a single point $\eta'$. 
  Furthermore, for every $x \in D'$, $\lambda'(s,x) = \eta'$. 
  If $U'$ is a simple neighborhood of the point $\eta'$ then 
  the tangent space $T_{\eta'}$ is in bijection with the connected components of the space 
  $U' \smallsetminus \eta'$. The set $D'$ corresponds to a single element of the tangent space at $\eta'$. 
  It follows that for some 
  $r \in [0,s)$, $\lambda'^{x_1}(r) \in D' \cap \lambda'^{x_2}([0,s])$. Let $q := \lambda'^{x_1}(r)$ and 
  $p := \phi^{\mathrm{an}}(q)$. By construction of the deformation retraction 
  $\lambda$, $\lambda(t,p) = p$ for every $t \in [0,r]$. Also, the deformation $\lambda$ satisfies the following property. 
  For every $a < b \in [0,1]$ and $y \in C^{\mathrm{an}}$, $\lambda(b,(\lambda(a,y)) = \lambda(b,y)$. It follows that 
  $\lambda'^q_{|[r,1]}$, $\lambda'^{x_1}_{|[r,1]}$ and $\lambda'^{x_2}_{|[r,1]}$ are all lifts of the path $\lambda^p_{|[r,1]}$. As the lift of the path 
  starting from $p$ is unique, we must have that $\lambda'^{q}_{[r,1]} = \lambda'^{x_1}_{[r,1]} = \lambda'^{x_2}_{[r,1]}$. 
 \end{proof}

\begin{prop}
  Let $\phi : C' \to C$ be a finite morphism between irreducible projective smooth $k$-curves.
  Assume in addition that the extension of function fields $k(C) \hookrightarrow k(C')$ is a finite Galois extension and 
  let $G := \mathrm{Gal}(k(C')/k(C))$.   
    Let 
  $\mathfrak{V} \subset \mathfrak{W}$ be weak semistable vertex sets of $C^{\mathrm{an}}$. Let
 $\mathfrak{W}' \subset C'^{\mathrm{an}}$ be a weak semistable 
  vertex set such that $\Sigma(C'^{\mathrm{an}},\mathfrak{W}') = (\phi^{\mathrm{an}})^{-1}(\Sigma(C^{\mathrm{an}},\mathfrak{W}))$. 
 Let $\lambda^\mathfrak{W}_{\Sigma(C^{\mathrm{an}},\mathfrak{V})} : [0,1] \times C^{\mathrm{an}} \to C^{\mathrm{an}}$ be the deformation retraction
 as constructed in Proposition 2.21 
  whose image is $\Sigma(C^{\mathrm{an}},\mathfrak{W})$.
   Let $\lambda' : [0,1] \times C'^{\mathrm{an}} \to C'^{\mathrm{an}}$ be a 
  function such that 
   for every $q \in C'^{\mathrm{an}}$, the path ${\lambda'}^q : [0,1] \to C'^{\mathrm{an}}$  
 defined by $t \mapsto {\lambda'}(t,q)$ is continuous and also that  
  the following diagram commutes.   
  
    \setlength{\unitlength}{1cm}
\begin{picture}(10,5)
\put(3.5,1){$[0,1] \times C^{\mathrm{an}}$}
\put(8,1){$C^{\mathrm{an}}$}
\put(3.5,3.5){$[0,1] \times C'^{\mathrm{an}}$}
\put(8,3.5){$C'^{\mathrm{an}}$}
\put(4.4,3.3){\vector(0,-1){1.75}}
\put(8.1,3.3){\vector(0,-1){1.75}}
\put(5.4,1.1){\vector(1,0){2.55}}
\put(5.4,3.6){\vector(1,0){2.55}}
\put(5.7,0.7){$\lambda^\mathfrak{W}_{\Sigma(\mathfrak{V},C^{\mathrm{an}})}$}
\put(6.2,3.2){$\lambda'$}
\put(4.5,2.3){$id \times \phi^{\mathrm{an}}$}
\put(8.2,2.3){$\phi^{\mathrm{an}}$}.
\end{picture}  

We suppose in addition that for every $q \in C'^{\mathrm{an}}$, the path 
$\lambda'^{q}$ is the unique lift
starting from $q$
 of the path $(\lambda^\mathfrak{W}_{\Sigma(\mathfrak{V},C^{\mathrm{an}})})^{\phi^{\mathrm{an}}(q)}$ and
 also that
 $\lambda'$ is $G$-invariant i.e. for every 
$g \in G$, $t \in [0,1]$ and $x \in C'^{\mathrm{an}}$, 
$g(\lambda'(t,x)) = \lambda'(t,g(x))$. 
The following statements are then true. 
\begin{enumerate} 
\item   Let $\mathcal{D}'$ denote the set of connected components of the space $C'^{\mathrm{an}} \smallsetminus \Sigma(C'^{\mathrm{an}},\mathfrak{W}')$.
 If $D' \in \mathcal{D}'$ then $D'$ is isomorphic to the Berkovich unit ball $\mathbf{O}(0,1)$ and we identify $D'$ via this isomorphism. 
 By Lemma 2.30, the group $G$ has a well defined action on the set $\mathcal{D}'$.
Let $H \subset G$ be the sub group which fixes $D'$. Let $W$ be a Berkovich closed or open ball strictly contained in $D'$. 
 There exists a Berkovich closed sub ball $\mathbf{B}(0,r) \subset D'$ with $r \in |k^*|$ such that $H$ stabilizes $\mathbf{B}(0,r)$ and 
 $W \subset \mathbf{B}(0,r)$.  
\item  The map $\lambda' : [0,1] \times C'^{\mathrm{an}} \to C'^{\mathrm{an}}$ is continuous. 
\end{enumerate}   
  \end{prop} 
   Over the course of the proof, we simplify notation and write $\lambda$ in place of $\lambda^\mathfrak{W}_{\Sigma(C^{\mathrm{an}},\mathfrak{W})}$. 
 The hypothesis that $\lambda'$ is $G$-invariant is redundant as it can be deduced from the uniqueness of the lifts.   
\begin{proof}
\begin{enumerate}
\item     Let $\mathcal{D}$ denote the set of connected components of the space $C^{\mathrm{an}} \smallsetminus \Sigma(C^{\mathrm{an}},\mathfrak{W})$. 
     In the proof of Proposition 2.21, we constructed the deformation retraction $\lambda$ by identifying every $D \in \mathcal{D}$ with 
     Berkovich open unit disks centered at $0$. By Lemma 6.1, the morphism $\phi^{\mathrm{an}}$
is open. As $\phi$ is a finite morphism, the morphism $\phi^{\mathrm{an}}$ is closed as well. Let $D' \in \mathcal{D}'$. By Lemma 2.30, there 
exists $D \in \mathcal{D}$ such that $\phi^{\mathrm{an}}(D') = D$. We also showed in 2.30, that $D'$ is a connected component of 
the space $(\phi^{\mathrm{an}})^{-1}(D)$ and hence the morphism $\phi^{\mathrm{an}}$ restricts to a closed and open morphism from 
$D'$ onto $D$. We use $\phi^{\mathrm{an}}_{D'}$ to 
denote the restriction of $\phi^{\mathrm{an}}$ to $D'$. 
Let $x \in D$ and $R := (\phi_{D'}^{\mathrm{an}})^{-1}(x) = \{ y_1,\ldots,y_m\}$. Recall that $H$ is the sub group of $G$ which 
stabilizes $D'$. It follows that $R = \{h(y_1)\}_{h \in H}$. There exists $s \in [0,1]$ such that for every $z \in D$, $\lambda(t,z) \in D$ for $t \in [0,s)$ and 
$\lambda(s,z) \in \overline{D} \smallsetminus D$.
 For every $i$, let $p_{y_i}$ denote the path $[0,s] \to \overline{D'}$ defined by $t \mapsto \lambda'(t,y_i)$.       
The paths $p_{y_i}$ are each lifts of the path $\lambda^x : [0,s] \to \overline{D}$ defined by $t \mapsto \lambda(t,x)$. 
Observe that $p_{y_i}(s) = \overline{D'} \smallsetminus D'$. By Lemma 2.35, there exists $r \in [0,s)$ such that 
for every $y,y' \in R$, ${p_{y'}}_{|[r,s)} = {p_{y}}_{|[r,s)}$. Let $r' \in [r,s)$ be such that 
$\lambda^x(r') = \eta_{0,u}$ (cf. 2.1.1) for some $u \in |k^*|$ and $\mathbf{B}(0,u)$ contains $\phi_{D'}^{\mathrm{an}}(W)$.  
Since $p_{y_i}(r') = p_{y_j}(r')$ for every $y_i,y_j \in R$ and the paths $p_y$ for $y \in R$ are Galois conjugates of each other, we must have that 
$q := p_{y}(r')$ is fixed by $H$ and hence $q = (\phi_{D'}^{\mathrm{an}})^{-1}(\lambda(r',x))$. We simplify notation and use $X$ to denote the closed disk  
$\mathbf{B}(0,u)$. 
The group $H$ restricts to an action of the space $Y := (\phi^{\mathrm{an}}_{D'})^{-1}(X)$. 
 We claim that $Y$ is a Berkovich closed disk in $D'$. Let $Y'$ denote the complement of $Y$ in $D'$. 
The image $\phi^{\mathrm{an}}(Y')$ is the complement of $X$ in $D$ as $\phi^{\mathrm{an}}_{D'}$ is surjective.
Let $X' := D \smallsetminus X$.  
 We claim
 that the space $Y'$ is a connected open set.
 Suppose that $Y'$ is not connected. The morphism $\phi^{\mathrm{an}}_{D'}$ is clopen and hence maps each connected component of $Y'$
onto the complement of $X$ in $D$. Let $Z$ be a connected component of $Y'$. 
Lemmas 2.32 and 2.33 can be used to show that $Z$ must be of the form $\mathbf{O}(0,1) \smallsetminus \bigcup_{j \in J} X_j$ where 
the $X_j$ are Berkovich closed disks or points of type I or IV. As the morphism $\phi^{\mathrm{an}}$ restricts to a finite morphism from 
$Z$ onto $X'$, it can be deduced that there can be only a finite number of points in $\overline{Y'} \smallsetminus Y'$ where $\overline{Y'}$ denotes the closure
of $Y'$ in $D'$. We must hence have that $Y' \in \mathcal{B}$ and is of form 3. 
As the union of open sets in $\mathcal{B}$ of form 3 is connected, we conclude that 
$Y'$ is a connected open set in $\mathcal{B}$ of form 3.  
 It must be the complement of the
union of a finite number of Berkovich closed disks or points of type I or IV.
The complement of $Y'$ is the space $Y$
 and $Y = (\phi^{\mathrm{an}}_{D'})^{-1}(X)$. 
We showed that there exists a point $q \in Y$ which 
 is $H$-invariant. As $\phi^{\mathrm{an}}_{D'}$ is clopen, every connected component of $Y$
 must contain the point $q$ and hence $Y$ is connected. If the union of a finite number of Berkovich closed disks and points of types I or IV  
in $D'$ is connected then that union must be a Berkovich closed disk as well. Hence, $Y$ is a Berkovich closed disk. The morphism $\phi^{\mathrm{an}}_{D'}$ 
restricts to a finite morphism from $Y$ onto $X$. 
As $k$ is algebraically closed, the radius of $Y$ belongs to the group $|k^*|$. This proves the first part of the proposition
\item 
We make use of the notation introduced in the proof of part 1 of the proposition. 
Let $W \subset C'^{\mathrm{an}}$ be a connected open set.
We must show that $\lambda'^{-1}(W)$ is an open subset of $[0,1] \times C'^{\mathrm{an}}$. 
 We divide the proof into two cases - when $W \cap \Sigma(C'^{\mathrm{an}},\mathfrak{W}')$ is empty and 
 when $W \cap \Sigma(C'^{\mathrm{an}},\mathfrak{W}')$ is non-empty. 
 
 We treat the first case - $W \cap \Sigma(C'^{\mathrm{an}},\mathfrak{W}') = \emptyset$. 
  As $W$ is connected, there exists $D' \in \mathcal{D}'$ such that $W \subset D'$. 
By 2.30, there exists $D \in \mathcal{D}$ such that $\phi^{\mathrm{an}}(D') = D$. 
We may suppose further that $W$ belongs to $\mathcal{B}$. 
It must be of form 1, 2 or 3. Suppose that 
$W$ is a Berkovich open disk contained in $D'$. Let $V := \phi^{\mathrm{an}}(W) \subset D$. 
By Lemma 2.33, $V$ is a Berkovich open disk in $D$. By construction of $\lambda$, we must have 
that there exists $s \in [0,1]$ such that $\lambda^{-1}(V) = [0,s) \times V$. By assumption, $\lambda'$ is 
compatible with $\lambda$ in that the following diagram is commutative.

   \setlength{\unitlength}{1cm}
\begin{picture}(10,5)
\put(3.5,1){$[0,1] \times C^{\mathrm{an}}$}
\put(8,1){$C^{\mathrm{an}}$}
\put(3.5,3.5){$[0,1] \times C'^{\mathrm{an}}$}
\put(8,3.5){$C'^{\mathrm{an}}$}
\put(4.4,3.3){\vector(0,-1){1.75}}
\put(8.1,3.3){\vector(0,-1){1.75}}
\put(5.4,1.1){\vector(1,0){2.55}}
\put(5.4,3.6){\vector(1,0){2.55}}
\put(6.2,0.7){$\lambda$}
\put(6.2,3.2){$\lambda'$}
\put(4.5,2.3){$id \times \phi^{\mathrm{an}}$}
\put(8.2,2.3){$\phi^{\mathrm{an}}$}.
\end{picture}

 It follows that 
$\lambda'^{-1}(W) \subset [0,s) \times (\phi^{\mathrm{an}})^{-1}(V)$. Let $A_1,\ldots,A_m$ denote the connected components 
of $(\phi^{\mathrm{an}})^{-1}(V)$.   
Observe that $(\phi^{\mathrm{an}})^{-1}(V) = \bigcup_{\sigma \in G} \sigma(W)$. 
We suppose without loss of generality that $W \subset A_1$.
As $A_1$ is connected, we must have that $A_1 \subset D'$. 
 We claim that $W = A_1$. 
By 2.30, if $\sigma \in G$ then $\sigma(D') \in \mathcal{D}'$. Let $H := \{ \sigma \in G | \sigma(D') = D' \}$.
We must have that $A_1 \subset \bigcup_{\sigma \in H} \sigma(W)$ and $A_1 \cap \sigma(W) = \emptyset$ if 
$\sigma \notin H$. By part (1) of the proposition and Lemma 2.34, if $\sigma \in H$ then 
$\sigma(W)$ is a Berkovich open ball whose radius is equal to that of $W$. It follows that one of the connected components 
of $\bigcup_{\sigma \in H} \sigma(W)$ is the ball $W$. Hence $W = A_1$. Observe that if $A_i$ is a connected component and $x \in A_i$ then 
for every $t \in [0,s)$, we must have that $\lambda'(t,x) \in A_i$. Indeed, the path $\lambda'^x : [0,s) \to C'^{\mathrm{an}}$ defined 
by $t \mapsto \lambda'(t,x)$ is contained in $(\phi^{\mathrm{an}})^{-1}(V)$. As $\{A_j\}$ is the set of connected components of 
$(\phi^{\mathrm{an}})^{-1}(V)$ we must have that $\lambda'(t,x) \in A_i$ for every $t \in [0,s)$. It follows from this 
observation that $\lambda'^{-1}(W) = [0,s) \times W$.  
        
    Let $W \subseteq D$ be a Berkovich open ball and $Y_1,\ldots,Y_m$ be disjoint Berkovich closed sub disks of $W$ or points of type I or IV.
    Let $Z := \bigcup_{1 \leq i \leq n} Y_i$. 
     We show 
    that $\lambda'^{-1}(W \smallsetminus Z)$ is an open subset of $[0,1] \times C'^{\mathrm{an}}$. We have already shown that there 
    exists $s \in [0,1]$ such that  
    $\lambda'^{-1}(W) = [0,s) \times W$. Hence $\lambda'^{-1}(W \smallsetminus Z) = ([0,s) \times W) \smallsetminus \lambda'^{-1}(Z)$. 
    As $Z$ is the disjoint union of the $Y_i$, we must have that $\lambda'^{-1}(Z) = \bigcup_i \lambda'^{-1}(Y_i)$. 
    It suffices hence to show that if $Y$ is a Berkovich closed disk contained in $D'$ then there exists $t \in [0,1]$ such that 
     $\lambda'^{-1}(Y) = [0,t] \times Y$. By Lemma 2.33, the image of $Y$ for 
     the morphism $\phi^{\mathrm{an}}_{D'}$ is a Berkovich closed disk or a point of type I or IV. We can
     use essentially the same argument above wherein we showed that the preimage of a Berkovich open disk $O$ 
     in $D'$ for the function $\lambda'$ is an open subset of $[0,1] \times C'^{\mathrm{an}}$ of the form $[0,s') \times O$ 
     to show that $\lambda'^{-1}(Y) = [0,t] \times Y$.   
    
    To conclude the  proof, we treat the case when $W \subset C'^{\mathrm{an}}$ is a 
    connected open set that intersects $\Sigma(C'^{\mathrm{an}},\mathfrak{W}')$. 
    Let $S := W \cap \Sigma(C'^{\mathrm{an}},\mathfrak{W}')$ and let $\mathcal{D}'_S := \{D' \in \mathcal{D} | \overline{D'} \cap S \neq \emptyset \}$. 
    As $W$ is connected, for every 
    $D' \in \mathcal{D}'_S$, $W \cap D'$ is a non empty connected open neighborhood which is the union of open sets in $\mathcal{B}$ of 
    form 3. It suffices to prove $\lambda'^{-1}(W)$ is open under the assumption that $W \cap D'$
    for $D' \in \mathcal{D}'_S$ 
     belongs to $\mathcal{B}$ and is of form 3.  
    We have the equality 
    $W = \bigcup_{D' \in \mathcal{D}'_S} (W \cap D') \cup (W \cap \Sigma(C'^{\mathrm{an}},\mathfrak{W}'))$. It follows that 
    $\lambda'^{-1}(W) = \bigcup_{D' \in \mathcal{D}'_S} \lambda'^{-1}(W \cap D') \cup \lambda'^{-1}(\Sigma(C'^{\mathrm{an}},\mathfrak{W}') \cap W)$. 
     We showed that for every $D' \in \mathcal{D}'_S$, $\lambda'^{-1}(W \cap D')$ is an open set of $[0,1] \times D'$. Furthermore, it can be checked that 
     $[0,1] \times W \subset \lambda'^{-1}(W)$. 
     By construction, 
   $\lambda'^{-1}(\Sigma(C'^{\mathrm{an}},\mathfrak{W}') \cap W) = 
   ([0,1] \times (\Sigma(C'^{\mathrm{an}},\mathfrak{W}') \cap W)) \bigcup (\{1\} \times \bigcup_{D' \in \mathcal{D}'_S} D')$. 
     The set $[0,1] \times W$ is an open subset of $[0,1] \times C'^{\mathrm{an}}$
     that is contained in $\lambda'^{-1}(W)$
      and is a neighborhood of every point 
     of $[0,1] \times (\Sigma(C'^{\mathrm{an}},\mathfrak{W}') \cap W)$.
    It remains to show that
    if $x \in D'$ for some $D' \in \mathcal{D}'_S$ then
     there exists an open subset of $[0,1] \times C'^{\mathrm{an}}$ 
     contained in $\lambda'^{-1}(W)$ 
     that is a neighborhood of $(1,x)$. We proceed as follows. As $W \cap D'$ is a connected open subset in $\mathcal{B}$
      of form 3, 
      it can be verified using the arguments above (case 1) that there exists $r_{W \cap D'} \in [0,1)$ such that 
      $(r_{W \cap D'},1] \times D' \subset \lambda'^{-1}(W)$. As $(r_{W \cap D'},1] \times D'$ is an open subset of $[0,1] \times C'^{\mathrm{an}}$ that 
      contains $1 \times D'$ we conclude the claim and the proof.     
   \end{enumerate}  
\end{proof}

\section{ Compatible deformation retractions}

      Our goal in this section is to prove the existence of a pair of compatible deformation retractions in the case of a finite morphism between curves. 
      The precise statement is the following. 
 \begin{thm}      
           Let $C$ and $C'$ be smooth projective irreducible $k$-curves and $\phi : C' \to C$ be a finite morphism. There exists a 
 pair of deformation retractions 
\begin{align*}
 \psi : [0,1] \times C^{\mathrm{an}} \to C^{\mathrm{an}}
 \end{align*}
 and 
\begin{align*}
  \psi' : [0,1] \times C'^{\mathrm{an}} \to C'^{\mathrm{an}}
\end{align*}
     with the following properties. 
 \begin{enumerate} 
 \item  The images 
 $\Upsilon_{C'^{\mathrm{an}}} := \psi'(1,C'^{\mathrm{an}})$
 and $\Upsilon_{C^{\mathrm{an}}} := \psi(1,C^{\mathrm{an}})$ are closed subspaces
of $C'^{\mathrm{an}}$ and $C^{\mathrm{an}}$ respectively, each with the structure of a 
connected, finite metric graph. Furthermore, we have that $\Upsilon_{C'^{\mathrm{an}}} = (\phi^{\mathrm{an}})^{-1}(\Upsilon_{C^{\mathrm{an}}})$.
\item 
There exists weak semistable vertex sets $\mathfrak{A}' \subset C'^{\mathrm{an}}$
and $\mathfrak{A} \subset C^{\mathrm{an}}$ such that 
$\Upsilon_{C'^{\mathrm{an}}} = \Sigma(C'^{\mathrm{an}},\mathfrak{A}')$ 
and $\Upsilon_{C^{\mathrm{an}}} = \Sigma(C^{\mathrm{an}},\mathfrak{A})$.
\item The deformation retractions $\psi$
and $\psi'$ are \textbf{compatible} i.e.
the following diagram is commutative. 
 
\setlength{\unitlength}{1cm}
\begin{picture}(10,5)
\put(2.5,1){$[0,1] \times C^{\mathrm{an}}$}
\put(6,1){$C^{\mathrm{an}}$}
\put(2.5,3.5){$[0,1] \times C'^{\mathrm{an}}$}
\put(6,3.5){$C'^{\mathrm{an}}$}
\put(3.4,3.3){\vector(0,-1){1.75}}
\put(6.1,3.3){\vector(0,-1){1.75}}
\put(4.4,1.1){\vector(1,0){1.55}}
\put(4.4,3.6){\vector(1,0){1.55}}
\put(4.7,0.7){$\psi$}
\put(4.7,3.2){$\psi'$}
\put(3.5,2.3){$id \times \phi^{\mathrm{an}}$}
\put(6.2,2.3){$\phi^{\mathrm{an}}.$}
\end{picture}   
\end{enumerate} 
\end{thm} 
 \begin{rem}   In \cite{HL}, Hrushovski and 
Loeser construct compatible 
 deformation retractions in greater generality. 
Let $\pi : V' \to V$ be a finite morphism between quasi-projective varieties $V'$ and $V$ 
over a non-Archimedean non-trivially real valued field $F$. There exists 
  a generalised real interval [\cite{HL}, Section 3.9] $I$ and 
 a pair of deformation retractions $H : I \times V^{\mathrm{an}} \to V^{\mathrm{an}}$ and 
 $H' : I \times V'^{\mathrm{an}} \to V'^{\mathrm{an}}$ which are compatible in the sense defined above. 
 This result follows from Remark 10.1.2 (2) and Corollary 13.1.6 in loc.cit. 
 \end{rem} 

   To prove Theorem 3.1, we adapt the   
  strategy employed in Section 7 of \cite{HL}. 
 Our method of proof is as follows. We begin by proving the theorem 
 for a finite morphism $\phi : C \to \mathbb{P}^1_k$ where 
 $C$ is a smooth, projective curve and the extension of function fields $k(\mathbb{P}^1_k) \hookrightarrow k(C)$ induced by the morphism
 $\phi$ is Galois. 
 We then use this result to prove the theorem 
 for a finite morphism $\phi : C' \to C$ between smooth projective curves. 
 We begin with the following lemma which provides us compatible weak semistable vertex sets for a finite morphism 
 between smooth projective irreducible curves.  
 
 \begin{lem}
     Let $\phi : C' \to C$ be a finite morphism between smooth projective irreducible curves. Let $S \subset C^{\mathrm{an}}$ be a finite set of points none 
     of which are of type IV.
      There exists a weak semistable vertex set $\mathfrak{W} \subset C^{\mathrm{an}}$ such that
      $\Sigma(C^{\mathrm{an}},\mathfrak{W})$ contains $S$, 
       $\mathfrak{V} := (\phi^{\mathrm{an}})^{-1}(\mathfrak{W})$ is a weak semistable vertex set of $C'^{\mathrm{an}}$ and 
     $\Sigma(C'^{\mathrm{an}}, \mathfrak{V}) = (\phi^{\mathrm{an}})^{-1}(\Sigma(C^{\mathrm{an}},\mathfrak{W}))$. 
     \end{lem}
 \begin{proof}
      Let $\mathfrak{W}_0$ be a weak semistable vertex set for $C^{\mathrm{an}}$ such that $\Sigma(C^{\mathrm{an}},\mathfrak{W}_0)$
      contains $S$. As the morphism $\phi^{\mathrm{an}}$ is finite, the preimage $(\phi^{\mathrm{an}})^{-1}(\Sigma(C^{\mathrm{an}},\mathfrak{W}_0))$ is 
      a finite graph which does not 
      contain a point of type IV (Proposition 2.27).
      Let $\mathfrak{V}_0$ be a weak semistable vertex set such that $\Sigma(C'^{\mathrm{an}},\mathfrak{V}_0)$ contains 
      $(\phi^{\mathrm{an}})^{-1}(\Sigma(C^{\mathrm{an}},\mathfrak{W}_0))$. Let $\mathfrak{W}_1$ be a weak semistable vertex set 
      such that the skeleton $\Sigma(C^{\mathrm{an}},\mathfrak{W}_1)$ contains $\phi^{\mathrm{an}}(\Sigma(C'^{\mathrm{an}},\mathfrak{V}_0))$.    
      We claim that the preimage $A := (\phi^{\mathrm{an}})^{-1}(\Sigma(C^{\mathrm{an}},\mathfrak{W}_1))$ is connected. Let 
      $A_1,\ldots,A_m$ denote the connected components of $A$ such that 
      $A_1$ contains $\Sigma(C'^{\mathrm{an}},\mathfrak{V}_0)$. 
      The morphism $\phi^{\mathrm{an}}$ is an open and closed morphism (cf. Lemma 6.1). It follows that 
      $\phi^{\mathrm{an}}$ restricts to a surjective map from each of the $A_i$ onto $\Sigma(C^{\mathrm{an}},\mathfrak{W}_1)$. 
       However since $A_1$ contains the set $(\phi^{\mathrm{an}})^{-1}(\Sigma(C^{\mathrm{an}},\mathfrak{W}_0))$
        we must have that $A = A_1$. It follows that $A$ is a connected graph 
       that contains the skeleton $\Sigma(C'^{\mathrm{an}},\mathfrak{V}_0)$, using which it can be checked
        that $C'^{\mathrm{an}} \smallsetminus A$ is the disjoint union of 
       sets each of which are isomorphic to Berkovich open balls. We claim that these open balls have radii belonging to $|k^*|$. 
       Let $D'$ be a connected component of $C'^{\mathrm{an}} \smallsetminus A$. 
       As the morphism $\phi^{\mathrm{an}}$ is 
       clopen and $A = (\phi^{\mathrm{an}})^{-1}(\Sigma(C^{\mathrm{an}},\mathfrak{W}_1))$, there exists a connected component $D$
       of $C^{\mathrm{an}} \smallsetminus \Sigma(C^{\mathrm{an}},\mathfrak{W}_1)$ such that the morphism 
      $\phi^{\mathrm{an}}$ restricts to a finite morphism from 
      $D'$ onto $D$. As $D$ is isomorphic to a Berkovich open ball of radius belonging to $|k^*|$, we must have that $D'$ is a Berkovich open 
      ball whose radius belongs to $|k^*|$. It follows that there exists a weak semistable vertex set $\mathfrak{V}_1$ in $C'^{\mathrm{an}}$ such that 
      $\Sigma(C'^{\mathrm{an}},\mathfrak{V}_1) = A$. The set $\mathfrak{W} := \mathfrak{W}_1 \cup \phi^{\mathrm{an}}(\mathfrak{V}_1)$ is a weak 
      semistable vertex set for $C^{\mathrm{an}}$ and $\Sigma(C^{\mathrm{an}},\mathfrak{W}) = \Sigma(C^{\mathrm{an}},\mathfrak{W}_1)$. Let 
      $\mathfrak{V} := (\phi^{\mathrm{an}})^{-1}(\mathfrak{W})$. As $\mathfrak{V}$ contains $\mathfrak{V}_1$ and is contained in 
      $\Sigma(C'^{\mathrm{an}},\mathfrak{V})$, we must have that $\mathfrak{V}$ is a weak semistable vertex set and 
      $\Sigma(C'^{\mathrm{an}},\mathfrak{V}) = \Sigma(C'^{\mathrm{an}},\mathfrak{V}_1)$. The pair 
      $\mathfrak{V},\mathfrak{W}$ satisfy the claims made in the lemma.

 \end{proof}

 \subsection{Lifting paths} 
 
       Let $\phi : C' \to C$ be a finite morphism between $k$-curves. A path in $C^{\mathrm{an}}$ is a continuous function 
       $u : [a,b] \to C^{\mathrm{an}}$ where $[a,b]$ is a real interval. To construct deformation retractions 
       which are compatible for the morphism $\phi$, we must understand to what extent certain paths on $C^{\mathrm{an}}$ can be lifted. 
       By a lift of a path, we mean the following. 
       
\begin{defi}
\emph{Let $a < b$ be real numbers and $u : [a,b] \to C^{\mathrm{an}}$ be a continuous function.} A lift of the path 
$u$ \emph{is a path $u' :[a,b] \to C'^{\mathrm{an}}$ such that $u = \phi^{\mathrm{an}} \circ u'$.}  
\end{defi}

\begin{lem} 
  Let $\phi : C' \to C$ be a finite separable morphism between irreducible projective smooth $k$-curves such that 
  the extension of function fields induced by $\phi$ is separable.    
    Let 
  $\mathfrak{V} \subset \mathfrak{W}$ be weak semistable vertex sets of $C^{\mathrm{an}}$.
  Assume that $\mathfrak{W}$ contains the set of $k$-points of $C$ over which the morphism $\phi$ is ramified. 
  Let $\mathfrak{W}' \subset C'^{\mathrm{an}}$ be a weak semistable 
  vertex set such that $\Sigma(C'^{\mathrm{an}},\mathfrak{W}') = (\phi^{\mathrm{an}})^{-1}(\Sigma(C^{\mathrm{an}},\mathfrak{W}))$. 
 Let $\lambda^\mathfrak{W}_{\Sigma(C^{\mathrm{an}},\mathfrak{V})} : [0,1] \times C^{\mathrm{an}} \to C^{\mathrm{an}}$ be the deformation retraction
 constructed in Proposition 2.21 
  whose image is $\Sigma(C^{\mathrm{an}},\mathfrak{W})$. We simplify notation 
  and  write $\lambda$ in place of $\lambda^\mathfrak{W}_{\Sigma(C^{\mathrm{an}},\mathfrak{V})}$. 
Let $p \in C^{\mathrm{an}}$. Let $\lambda^p : [0,1] \to C^{\mathrm{an}}$ be the path 
defined by $t \mapsto \lambda(t,p)$. Let $q \in (\phi^{\mathrm{an}})^{-1}(p)$. There exists a unique path 
$u : [0,1] \to C'^{\mathrm{an}}$ such that $u(0) = q$ and $\phi^{\mathrm{an}} \circ u = \lambda^p$.  
\end{lem} 
\begin{proof} 
We split the proof into two cases. 
\begin{enumerate}
\item \emph{Let $p \in C^{\mathrm{an}}$ be a point which is not of type IV.} 
 We can suppose that $p \notin \Sigma(C^{\mathrm{an}},\mathfrak{W})$ since when 
 $p \in \Sigma(C^{\mathrm{an}},\mathfrak{W})$, the path $\lambda^p$ is constant and hence can always be lifted. 
We show firstly that for every $t \in [0,1)$, there exists $\epsilon > 0$ such that 
  if $z' \in (\phi^{\mathrm{an}})^{-1}(\lambda^p(t))$ then 
  $\lambda^p_{|[t,t + \epsilon]}$ lifts \emph{uniquely} to a path starting from $z'$.  
  
  Suppose $z := \lambda^p(t)$ was a point of type I. 
  By construction of $\lambda$, we must have that $z = p$ and $t = 0$. 
   Our choice of weak semistable 
  vertex sets implies that $\phi$ is étale over $\lambda^p(t)$. It follows from Hensel's lemma that there 
  exists neighborhoods $V_{z'}$ in $C'^{\mathrm{an}}$ around $z'$ and $V_p$ around $p$ such that 
  $\phi^{\mathrm{an}}$ restricts to a homeomorphism from $V_{z'}$ onto $V_p$. We conclude from this fact that 
  there does indeed exist $\epsilon > 0$ such that $\lambda^p_{|[0,\epsilon]}$ lifts uniquely to a path starting from $z'$.
 
 Let $z := \lambda^p(t)$ be a point of type II or III. 
 By construction, for every $s \in [t,1]$, $\lambda^p(s) = \lambda^z(s)$ where 
 $\lambda^z : [0,1] \to C^{\mathrm{an}}$ is the path defined by $s \mapsto \lambda(s,z)$. 
 Furthermore, for every $s \in [0,t]$, $\lambda^z(s) = z$. 
 It suffices to show that there exists $\epsilon > 0$ such that the 
 path $\lambda^z_{|[t,t + \epsilon]}$ lifts uniquely to a path starting from $z'$. 
 If there exists $\epsilon > 0$ such that $\lambda^z_{|[t,t + \epsilon]}$ is constant
 then our claim is obviously true. Let us hence suppose no such $\epsilon$ exists. 
   
  By assumption, $z \notin \Sigma(C^{\mathrm{an}},\mathfrak{W})$. 
  It follows that $z' \in C'^{\mathrm{an}} \smallsetminus \Sigma(C'^{\mathrm{an}},\mathfrak{W}')$.  
  Recall that we used $\mathcal{D}$ to denote the set of connected components of the space 
  $C^{\mathrm{an}} \smallsetminus \Sigma(C^{\mathrm{an}},\mathfrak{W})$ and
  when constructing the 
  deformation $\lambda$
  we identified each $D \in \mathcal{D}$ with a Berkovich open ball whose radius belongs to the value group $|k^*|$.
   Likewise, let $\mathcal{D}'$ denote the set of connected components of the space 
  $C'^{\mathrm{an}} \smallsetminus \Sigma(C'^{\mathrm{an}},\mathfrak{W}')$. We identify each $D' \in \mathcal{D}'$ 
  with a Berkovich open ball of unit radius.
  Let $D \in \mathcal{D}$ be such that $z \in D$ and $D' \in \mathcal{D}'$ such that 
  $z' \in D'$. By Lemma 2.30, we have that $\phi^{\mathrm{an}}(D') = D$ and in addition
  $\phi^{\mathrm{an}}$ restricts to an open and closed map on $D'$.  
   By construction of the deformation retraction $\lambda$, there exists 
  $\beta \in [0,1]$ such that for every $x \in D$,
  $\lambda(s,x) \in D$ when $s \in [0,\beta)$ i.e. $\lambda : [0,\beta) \times D \to D$ is well defined and 
  $\lambda(s,x) = \lambda(\beta,x)$ when $s \in [\beta,1]$. 
Our assumption that there does not exist $\epsilon > 0$ such that 
$\lambda^z_{|[t,t + \epsilon]}$ is constant and the construction of $\lambda$ imply that
 the path $\lambda^z_{|[t,\beta]}$ is injective and that
  $\lambda^z([t,\beta]) \subset \mathbf{H}(C^{\mathrm{an}})$. Furthermore, the composition 
  $\lambda^z \circ (-\mathrm{exp}) : [-\mathrm{log}(\beta),-\mathrm{log}(t)] \to C^{\mathrm{an}}$ is an isometry (cf. Remark 2.22).
  
      Let $r \in |k^*|$ denote the radius of the ball $D$. 
   The point $z$ must be of the form $\eta_{a,t}$ for some $a \in D$. 
   Likewise, the point $z' \in D'$ must be of the form $\eta_{b,t'}$ for some $b \in D'$ and 
   $t' \in (0,1)$.
     We may choose $b$ so that $\phi^{\mathrm{an}}(b) = a$. 
    We show now that we can reduce to the case when $a = b = 0$. 
   The translation automorphism
   $t_{-a}: O(0,r) \to O(0,r)$ defined by
    $x \mapsto x - a$ induces an automorphism $t_{-a}^{\mathrm{an}} : D \to D$ that
    maps the point $z$ to 
   $\eta_{0,t}$. We have that $\lambda : [0,\beta) \times D \to D$ is well defined and 
   it can be checked that $\lambda$ is $t_{-a}^{\mathrm{an}}$ invariant i.e. 
   for every $s \in [0,\beta)$ and $x \in D$, 
   $t_{-a}^{\mathrm{an}}(\lambda(s,x)) = \lambda(s,t_{-a}^{\mathrm{an}}(x))$.   
  Similarly, let $t_{-b} : O(0,1) \to O(0,1)$ denote the translation morphism 
  $y \mapsto y - b$. The map $t_{-b}$ induces an automorphism $t_{-b}^{\mathrm{an}} : D' \to D'$ that 
  maps $z' = \eta_{b,t'}$ to $\eta_{0,t'}$.
  Let $\phi^{\mathrm{an}}_{D'}$ denote the restriction of $\phi^{\mathrm{an}}$ to $D'$. As $t_{-a}^{\mathrm{an}}$ and 
  $t_{-b}^{\mathrm{an}}$ are automorphisms,
  there 
  exists a morphism $f : D' \to D$ such that the following diagram commutes. 
  
  \setlength{\unitlength}{1cm}
\begin{picture}(10,5)
\put(3.3,1){$D$}
\put(6,1){$D$}
\put(3.3,3.5){$D'$}
\put(6,3.5){$D'$}
\put(3.4,3.3){\vector(0,-1){1.75}}
\put(6.1,3.3){\vector(0,-1){1.75}}
\put(3.9,1.1){\vector(1,0){1.9}}
\put(3.9,3.6){\vector(1,0){1.9}}
\put(4.7,0.7){$t^{\mathrm{an}}_{-a}$}
\put(4.7,3.2){$t^{\mathrm{an}}_{-b}$}
\put(3.5,2.3){$\phi_{D'}^{\mathrm{an}}$}
\put(6.2,2.3){$f$}
\end{picture}   
  
   Suppose there exists $\epsilon > 0$ and a unique path 
   $u' : [t,t+\epsilon]$ starting from $\eta_{0,t'}$ such that 
   $f \circ u' = \lambda^{\eta_{0,t}}_{|[t,t+\epsilon]}$. The commutativity of the above diagram
   and the fact that $t_{-a}^{\mathrm{an}}$ and $t_{-b}^{\mathrm{an}}$ are automorphisms 
   imply that there exists a unique path $u : [t,t+\epsilon]$ starting at $z'$ such that $\phi^{\mathrm{an}} \circ u = \lambda^z_{|[t,t+\epsilon]}$.  
    We may hence assume that $z = \eta_{0,t}$ and $z' = \eta_{0,t'}$ and that $\phi^{\mathrm{an}}(0) = 0$.   
   
   Let $F \subset D'$ denote the set $(\phi^{\mathrm{an}})^{-1}(0)$. Let $\mathfrak{A} \subset D'$ be 
   a finite set of points of type II with the following property. 
   Let $\mathcal{C}$ denote the set of connected components of
    $D' \smallsetminus (\mathfrak{A} \cup F)$. 
    There exists a finite set $\Upsilon \subset \mathcal{C}$ such that if $A \in \Upsilon$ then 
    $A$ is isomorphic to a standard open annulus or a standard punctured Berkovich open disk and 
    if $A \in \mathcal{C} \smallsetminus \Upsilon$ then $A$ is isomorphic to 
    Berkovich open ball with radius in $|k^*|$.
    Let 
   $\Sigma(D')$ be the union of $\mathfrak{A} \cup F$ and the skeleton of every element $A \in \Upsilon$.  
  By assumption, we must have that $z' \in \Sigma(D')$. Suppose, $z'$ is not a vertex of the finite graph $\Sigma(D')$. It follows that there 
  exists a standard open annulus $A' \in \Upsilon$ that contains $z$ and in particular does not intersect $F$.
   The map $\phi^{\mathrm{an}}$ restricts to a morphism 
  $ A' \to D \smallsetminus \{0\}$. By [\cite{BPR}, Proposition 2.5], 
  $\phi^{\mathrm{an}}(A') \subset D \smallsetminus \{0\}$ must be a standard open annulus $A$ as well and
  $\phi^{\mathrm{an}}(\Sigma(A')) = \Sigma(A)$. By assumption, we have that $z \in \Sigma(A)$ and 
  for small enough $\epsilon > 0$, the path $\lambda^z_{|[t,t + \epsilon]}$ is contained in 
  $\Sigma(A)$. Recall that we defined sections $\sigma : \mathrm{trop}(A) \to \Sigma(A)$ and 
  $\sigma : \mathrm{trop}(A') \to \Sigma(A')$ 
   of the tropicalization maps 
   $\mathrm{trop} : \Sigma(A) \to \mathrm{trop}(A)$ and $\mathrm{trop} : \Sigma(A') \to \mathrm{trop}(A')$ (cf. 2.2.2).
   By definition of the tropicalization map, 
   we must have that $[-\mathrm{log}(t),-\mathrm{log}(t + \epsilon)] \subset \mathrm{trop}(A)$.  
  By construction, $\lambda^z \circ -\mathrm{exp} : [-\mathrm{log}(t), -\mathrm{log}(t + \epsilon)] \to \Sigma(A)$ coincides with 
  $\sigma_{|[-\mathrm{log}(t), -\mathrm{log}(t + \epsilon)]}$. As $\sigma$ is a homeomorphism, the morphism 
  $\phi^{\mathrm{an}}_{|\Sigma(A')} : \Sigma(A') \to \Sigma(A)$
   induces a map $\phi^{\mathrm{trop}} : \mathrm{trop}(A') \to \mathrm{trop}(A)$. By loc.cit, 
  there exists a non zero $d \in \mathbb{Z}$ such that $\phi^{\mathrm{trop}}$ is of the form 
  $d(.) + -\mathrm{log}(|\delta|)$ for some $\delta \in k^*$. Let $(\phi^{\mathrm{trop}})^{-1}$ denote the inverse of $\phi^{\mathrm{trop}}$. 
  It can be verified that $u := \sigma \circ (\phi^{\mathrm{trop}})^{-1} \circ -\mathrm{log} : [t,t+\epsilon] \to D'$ is a lift of 
  $\lambda^z_{|[t,t+\epsilon]}$ starting from $z'$.
  In fact, $u$ is the unique lift of $\lambda^z_{|[t,t+\epsilon]}$ starting from $z'$. Indeed, by Proposition 2.5 in loc.cit.,
   it can be deduced that $\phi^{\mathrm{an}}(A' \smallsetminus \Sigma(A')) \subset A \smallsetminus \Sigma(A)$ 
  and furthermore, the map $\phi^{\mathrm{an}}$ restricted to $\Sigma(A')$ is a bijection onto $\Sigma(A)$. As 
  $\lambda^z_{|[t,t+\epsilon]}$ is a path along $\Sigma(A)$, $u$ must be a lift along $\Sigma(A')$ and hence unique.

   Suppose $z' = \eta_{0,t'} \in \Sigma(D')$ is a vertex. We must have that $z'$ is a type II point.
  It follows that there exists a standard open annulus $A' \in \Upsilon$ with inner radius $t'$ and outer radius belonging to $|k^*|$.
  By \cite{BPR},
   the image $\phi^{\mathrm{an}}(A')$ is a standard open annulus $A \subset D$ and we choose $\epsilon > 0$ small enough so that 
   $\lambda^z_{|(t,t+\epsilon]}$ is a path contained in $\Sigma(A)$. We now proceed as above to 
   obtain a lift $u$ of $\lambda^z_{|(t,t+\epsilon]}$ starting from $z'$ and for reasons mentioned above it is the unique lift contained 
   in $\Sigma(A)$. It remains to show that 
    $u$ is the unique lift in $C'^{\mathrm{an}}$ and this requires an additional argument for the following reason. 
   Observe that the annulus $A$ contains exactly one element of the tangent space $T_{z'}$, whereas previously 
   when $z'$ was in the interior of $A$, every tangent direction was contained in $A$.  
   It follows that although $u$ might be the unique lift in $A$, it might not be the only lift in $C'^{\mathrm{an}}$. 
   However, it is the only possible lift since by Lemma 2.33, the Berkovich closed ball $\mathbf{B}(0,t')$ maps 
   onto the closed ball $\mathbf{B}(0,t)$ and the path $\lambda^z_{|(t,t + \epsilon]}$ is contained in the complement 
   of the $\mathbf{B}(0,t)$. Hence any lift of the $\lambda^z_{|(t,t + \epsilon]}$ must be contained entirely in $A$. 
   This proves that the lift $u$ is unique.  
   
    Let $L$ denote the set of  
    $s \in [0,1]$ for which there exists a lift of the path $\lambda^p_{|[0,s]}$ starting from $q$. 
    Let $s \in L$. We claim that there exists a unique lift of the path $\lambda^p_{|[0,s]}$ starting from $q$. 
    Let $u',u$ be two lifts of the path $\lambda^p_{|[0,s]}$ starting from $q$. Let $s_0 \in [0,s]$ be the largest 
    real number such that $u'_{|[0,s_0]} = u_{|[0,s_0]}$. Let $q' := u'(s_0) = u(s_0)$ and $p' := \phi^{\mathrm{an}}$. By 
    the first part of the proof, there exists
     $\epsilon > 0$ and
     a unique lift $v$ of the path $\lambda^p_{|[s_0,s_0 + \epsilon]}$. As $u'_{|[s_0,s_0 + \epsilon]}$ and 
     $u_{|[s_0,s_0 + \epsilon]}$ are both lifts of $\lambda^p_{|[s_0,s_0 + \epsilon]}$, we must have that 
     $v = u_{|[s_0,s_0 + \epsilon]} = u'_{|[s_0,s_0 + \epsilon]}$. This implies that 
     $u_{|[0,s_0 + \epsilon]} = u'_{|[0,s_0 + \epsilon]}$ which contradicts our assumption on $s_0$. We have thus verified our claim. It
     can be deduced from this that the set $L$ contains its supremum.    
    Let $t$ 
    be the largest real number in $[0,1]$
    such that the path $\lambda^p_{|[0,t]}$ lifts to a path $u : [0,t] \to C'^{\mathrm{an}}$ starting from $q$. 
    Suppose $t < 1$.
    By the first half of the
    proof, there exists
     a small real number $\epsilon$ such that $\lambda^p_{|[t,t + \epsilon]}$ lifts to a path starting from $u' : [t,t+\epsilon] \to C'^{\mathrm{an}}$ starting from 
    $u(t)$. Glueing $u$ and $u'$ gives a path $u'' : [0,t+\epsilon] \to C'^{\mathrm{an}}$ which lifts $\lambda^p_{|[0,t+\epsilon]}$. 
    Hence we have a contradiction to our assumption that $t < 1$. Therefore there exists a unique lift of the path $\lambda^p$.

\item  \emph{Let $p$ be a point of type IV.}   
    We must have that
   $p \notin \Sigma(C^{\mathrm{an}},\mathfrak{W})$.
   We make use of the notation introduced in part (1) of the proof. 
   There exists $D \in \mathcal{D}$ such that $p \in D$. 
   The path $\lambda^p$ is injective on $[a,b]$.
   Let $U$ be a connected neighborhood in $D$ 
   of the point $p$ such that $(\phi^{\mathrm{an}})^{-1}(U)$ decomposes into the disjoint union of 
   connected open sets $\{U_1,\ldots,U_m\}$ and each $U_i$ contains exactly one preimage of the point $p$.
   We claim that we can shrink $U$ and choose $a < t'_1 < b$ such that
    $\lambda^p([0,t'_1]) \subset U$ and
     for every $x \in \lambda^p([0,t'_1])$ there exists 
   exactly one preimage of $x$ in each of the $U_i$. 
   This can be accomplished as follows. We show that there exists $t' \in (a,b)$ such that for every 
   element $x \in \lambda^p([a,t'))$ the cardinality of the set $(\phi^{\mathrm{an}})^{-1}(x)$ is constant. We can then shrink
    $U$ so that it does not intersect $\lambda^p([t',b])$ and choose $t'_1 \in (a,t')$ suitably small.
    Such a $U$ must satisfy the claim since the morphism $\phi^{\mathrm{an}}$ being closed and open (cf. Lemma 6.1)
    is surjective from each of the $U_i$ onto $U$. 
    Observe that if $t_1,t_2 \in (a,s]$ are such that $t_1 < t_2$ then 
   the number of preimages of $\lambda^p(t_1)$ is greater than or equal to the number of preimages of 
   $\lambda^p(t_2)$. This follows from the uniqueness of lifts from part (1) and 
   that if $P$ is a lift of the path $\lambda^p_{|[t_1,s]}$ then $P_{|[t_2,s]}$ is a lift of the path $\lambda^p_{|[t_2,s]}$. 
   As the morphism $\phi^{\mathrm{an}}$ is finite, there exists $t' \in (a,b)$ such that the number of preimages of 
   every point $x \in \lambda^p((a,t'))$ is constant. The preimages in $C'^{\mathrm{an}}$ of the point $p$ are of type IV and 
   the tangent space at any such point is a single element. It follows that the number of preimages of every point in      
   $\lambda^p([a,t'))$ is a constant. Let $t'_1 \in (a,t')$ be such that $\lambda^p(t') \in U$. This verifies the claim.       
   We suppose without loss of generality
   that $q \in U_1$.  
   We show firstly that the path $\lambda^p_{[0,t'_1]}$ can be lifted to a path in $C^{\mathrm{an}}$ starting from 
   $q$. It suffices to show that the path $\lambda^p_{[a,t'_1]}$ can be lifted to a path in $C^{\mathrm{an}}$ starting from 
   $q$.
    Let $I$ denote the set of real numbers $r \in [a,t'_1]$ for which there exists a lift $P_r$ of the path 
    $\lambda^p_{|[r,t'_1]}$ contained in $U_1$. As $U_1$ contains exactly 
    one preimage of the point $\lambda^p(r)$, the uniqueness of lifts from part (1) of the proposition implies that the set $I$ is closed. 
    Let $t_0$ denote the smallest element of the set $I$. 
        Suppose $t_0 > a$. 
   Let $a' \in (a,t_0)$ and $p' := \lambda^p(a')$. By construction, $p'$ is not a point of type IV.
   Let $q'$ be the unique preimage of $p'$ in $U_1$. 
    By part (1), there exists a lift $P'$ of the path $\lambda^{p'}_{|[a',t'_1]}$ starting from $q'$. By construction, 
    $\lambda^{p'}_{|[t_0,t'_1]}$ coincides with $\lambda^p_{|[t_0,t'_1]}$. As the lifts are unique, we must have that 
    $P'_{|[t_0,t'_1]} = P$. This is a contradiction to our assumption that $t_0 > a$. It follows that there exists a lift of the path 
    $\lambda^p_{|[a,t'_1]}$ in $U_1$ which in turn implies that there exists a lift of the path $\lambda^p_{|[0,t'_1]}$ in $U_1$ as 
    $\lambda^p_{|[0,a]}$ is constant. 
      We abuse notation and 
    refer to this lift as $P'$ as well. 
    Let $p'' := \lambda^p(t'_1)$. By construction, $\lambda^{p''}_{|[t'_1,1]}$ coincides with the path $\lambda^p_{|[t'_1,1]}$. Let 
    $q'' := P'(t'_1)$. By part (1), there exists a lift $P''$ of the path $\lambda^{p''}_{|[t'_1,1]}$ starting from $q''$. Glueing the paths 
    $P'$ and $P''$ results in a lift of the path $\lambda^p$ starting from 
    $q$.        
   \end{enumerate} 
\end{proof}

 \subsection{Finite morphisms to $\mathbb{P}^1_k$}
   
         Let $C$ be a smooth projective irreducible $k$-curve. Let $\phi : C \to \mathbb{P}^1_k$ be a finite morphism such that 
         the extension of function fields $k(\mathbb{P}^1_k) \hookrightarrow k(C)$ induced by $\phi$ is separable.  
          Let $R$ be the finite set of $k$-points of $\mathbb{P}^1_k$ over which the morphism 
$\phi$ is ramified. Let
$\mathfrak{W}$ be a weak semistable vertex set that contains $R$ such that 
$\mathfrak{V} := (\phi^{\mathrm{an}})^{-1}(\mathfrak{W})$ is a weak semistable vertex set for $C^{\mathrm{an}}$ and 
$\Sigma(C^{\mathrm{an}},\mathfrak{V}) = (\phi^{\mathrm{an}})^{-1}(\Sigma(\mathbb{P}^{1,\mathrm{an}}_k,\mathfrak{W}))$.
By Proposition 2.21, there exists a deformation retraction 
\begin{align*}
\lambda_{\Sigma(\mathbb{P}_k^{1,\mathrm{an}},\mathfrak{W})} : [0,1] \times \mathbb{P}_k^{1,\mathrm{an}} \to \mathbb{P}_k^{1,\mathrm{an}} 
 \end{align*}
 whose image is the skeleton $\Sigma(\mathbb{P}_k^{1,\mathrm{an}},\mathfrak{W})$.   
 Recall that for a point $p \in \mathbb{P}^{1,\mathrm{an}}_k$, the deformation retraction $\lambda_{\Sigma(\mathbb{P}_k^{1,\mathrm{an}},\mathfrak{W})}$
 defines a path $\lambda^p_{\Sigma(\mathbb{P}_k^{1,\mathrm{an}},\mathfrak{W})} : [0,1] \to \mathbb{P}^{1,\mathrm{an}}_k$ by 
 $t \mapsto \lambda_{\Sigma(\mathbb{P}_k^{1,\mathrm{an}},\mathfrak{W})}(t,p)$.   
          We are now in a position to prove Theorem 3.1 for the morphism $\phi : C \to \mathbb{P}^1_k$. We suppose in addition that the 
       extension of function fields $k(\mathbb{P}^1_k) \hookrightarrow k(C)$ induced by $\phi$ is Galois.

 \begin{prop} 
 Let $C$ be a smooth projective irreducible curve and let $\phi : C \to \mathbb{P}^1_k$
  be a finite  morphism such that the extension of function fields $k(\mathbb{P}^{1}_k) \hookrightarrow k(C)$ is Galois.  
  Let $\mathfrak{W}$ be a weak semistable vertex set for $\mathbb{P}^{1,\mathrm{an}}_k$ that contains the 
 closed points over which the morphism is ramified such that 
$\mathfrak{V} := (\phi^{\mathrm{an}})^{-1}(\mathfrak{W})$ is a weak semistable vertex set for $C^{\mathrm{an}}$ and 
$\Sigma(C^{\mathrm{an}},\mathfrak{V}) = (\phi^{\mathrm{an}})^{-1}(\Sigma(\mathbb{P}^{1,\mathrm{an}}_k,\mathfrak{W}))$.
  There exists a pair of compatible deformation retractions 
$\psi' : [0,1] \times C^{\mathrm{an}} \to C^{\mathrm{an}}$ 
and $\psi : [0,1] \times \mathbb{P}_k^{1,\mathrm{an}} \to \mathbb{P}_k^{1,\mathrm{an}}$ 
whose images are the connected finite graphs $\Sigma(C^{\mathrm{an}},\mathfrak{V})$ and 
$\Sigma(\mathbb{P}_k^{1,\mathrm{an}},\mathfrak{W})$ respectively. 
 \end{prop}
\begin{proof}
 Let $\psi := \lambda_{\Sigma(\mathbb{P}^{1,\mathrm{an}}_k,\mathfrak{W})}$. 
 We define the deformation retraction $\psi' : [0,1] \times C^{\mathrm{an}} \to C^{\mathrm{an}}$ as follows. Let 
 $q' \in C^{\mathrm{an}}$ and $q := \phi^{\mathrm{an}}(q')$. By Lemma 3.5, there exists a 
 unique lift $\psi'^{q'}$ of the path $\psi^q$ starting at $q'$. 
For $t \in [0,1]$ and $q' \in C'^{\mathrm{an}}$, we set 
$\psi'(t,q') := \psi'^{q'}(t)$. The uniqueness of the lifts $\psi'^{q'}$ imply that $\psi'$ is well defined. Let 
$G = \mathrm{Gal}(k(C)/k(\mathbb{P}^1_k))$. The uniqueness of the lift implies that 
for every $g \in G$, $g \circ \psi'^{q'} = \psi'^{g(q')}$.
 It follows that for every $t \in [0,1]$, 
$g(\psi'(t,q')) = \psi'(t,g(q'))$.
The compatibility of $\psi'$ and $\psi$ implies that $\psi'(1,C^{\mathrm{an}})$ is equal to 
 $(\phi^{\mathrm{an}})^{-1}(\Sigma(\mathbb{P}_k^{1,\mathrm{an}},\mathfrak{W})) = \Sigma(C^{\mathrm{an}},\mathfrak{V})$.    
 The continuity of $\psi'$ follows from 2.36. 
\end{proof}

   We show that Theorem 3.1 can be deduced from Proposition 3.6. 
   
\begin{proof} 
   Let $\phi : C' \to C$ be a finite morphism between smooth projective irreducible $k$-curves. 
It suffices to prove the theorem when the extension of function fields
$k(C) \hookrightarrow k(C')$
 induced by the morphism 
$\phi$ is separable. Indeed, the extension $k(C) \hookrightarrow k(C')$ can be decomposed into a
separable field extension $k(C)  \hookrightarrow L$ and a purely inseparable extension $L \hookrightarrow k(C')$. 
Let $C''$ denote the smooth projective irreducible $k$-curve that corresponds to the function field $L$. 
The corresponding morphism of curves $C' \to C''$ and its analytification 
$C'^{\mathrm{an}} \to C''^{\mathrm{an}}$ are homeomorphisms. If $\mathfrak{V}''$ is a weak semistable vertex set for 
$C''^{\mathrm{an}}$ then its preimage $\mathfrak{V}'$ in $C'^{\mathrm{an}}$ is a weak semistable vertex set as well and 
a deformation retraction of $C''^{\mathrm{an}}$ with image $\Sigma(C''^{\mathrm{an}},\mathfrak{V}'')$ lifts to a deformation 
retraction on $C'^{\mathrm{an}}$ with image $\Sigma(C'^{\mathrm{an}},\mathfrak{V}')$. 
 
  Let $a : C \to \mathbb{P}^1_k$ be a finite separable morphism and let $K$ be a finite Galois extension 
  of $k(\mathbb{P}^1_k)$ that contains $k(C')$. Let $C''$ denote the smooth projective irreducible curve 
  corresponding to the function field $K$. By construction we have the following sequence of morphisms : 
  $C'' \to C' \to C \to \mathbb{P}^1_k$. Let $c : C'' \to \mathbb{P}^1_k$ denote this composition.
Using Lemma 3.3, it can be checked that there exists a weak semistable vertex set
  $\mathfrak{A}$ for $\mathbb{P}^{1,\mathrm{an}}_k$ that contains the points over which the morphism $c : C'' \to \mathbb{P}^1_k$ 
  is ramified  
  and in addition that $(a^{\mathrm{an}})^{-1}(\mathfrak{A})$ is a weak semistable vertex set 
  of $C^{\mathrm{an}}$, ${(a \circ \phi)^{\mathrm{an}}}^{-1}(\mathfrak{A})$ is a weak semistable vertex set of 
  $C'^{\mathrm{an}}$ and $(c^{\mathrm{an}})^{-1}(\mathfrak{A}) = \mathfrak{A}''$ is a weak semistable vertex set 
  for $C''^{\mathrm{an}}$. Furthermore, 
  $\Sigma(C^{\mathrm{an}},(a^{\mathrm{an}})^{-1}(\mathfrak{A})) = (a^{\mathrm{an}})^{-1}(\Sigma(\mathbb{P}^{1,\mathrm{an}}_k,\mathfrak{A}))$,  
  $\Sigma(C'^{\mathrm{an}},{(a \circ \phi)^{\mathrm{an}}}^{-1}(\mathfrak{A})) = {(a \circ \phi)^{\mathrm{an}}}^{-1}(\Sigma(\mathbb{P}^{1,\mathrm{an}}_k,\mathfrak{A}))$ 
  and $\Sigma(C''^{\mathrm{an}},\mathfrak{A}'') = (c^{\mathrm{an}})^{-1}(\Sigma(\mathbb{P}^{1,\mathrm{an}}_k,\mathfrak{A}))$.  
  
       The deformation retraction $\lambda_{\Sigma(\mathbb{P}^{1,\mathrm{an}}_k,\mathfrak{A})}$
       has image $\Sigma(\mathbb{P}^{1,\mathrm{an}}_k,\mathfrak{A})$ and
        lifts to a deformation 
       retraction $\lambda''^{\mathfrak{A}''}$ with image $\Sigma(C''^{\mathrm{an}},\mathfrak{A}'')$. 
       Let $G := \mathrm{Gal}(k(C'')/k(\mathbb{P}^1_k))$. 
       The deformation retraction $\lambda''^{\mathfrak{A}''}$ is $G$-invariant. There exists sub groups $H_{C'} \subset G$ and 
       $H_{C} \subset G$ such that $C'' \to C'$ and $C'' \to C$ are the quotient morphisms $C'' \to C''/H_{C'}$ and 
       $C'' \to C''/H_{C}$. As $\lambda''^{\mathfrak{A}''}$ is $H_{C}$ and $H_{C'}$ invariant, it must induce 
       deformations $\lambda_{C}$ and $\lambda_{C'}$ whose images are 
      $(a^{\mathrm{an}})^{-1}(\Sigma(\mathbb{P}^{1,\mathrm{an}}_k,\mathfrak{A}))$ and
       $((a \circ \phi)^{\mathrm{an}})^{-1}(\Sigma(\mathbb{P}^{1,\mathrm{an}}_k,\mathfrak{A}))$ 
        respectively. 
      This proves the theorem. 
 \end{proof}

\section{Calculating the genera $g^{\mathrm{an}}(C')$ and $g^{\mathrm{an}}(C)$} 

     In the previous sections we showed that given a $k$-curve, there exists a deformation retraction 
     of the curve onto a closed subspace which is a finite metric graph.
     We called such subspaces skeleta.  
      In Definition 2.26, we introduced the genus of a skeleton and by Proposition 2.25 it is independent of the weak semistable vertex sets that 
      define it, implying that it is in fact an invariant of the curve. 
       In what follows we study how these invariants relate to each other given a finite morphism between the spaces they are associated to. 
     
          The theorem that follows is analogous to the Riemann-Hurwitz formula in algebraic geometry. We 
          introduce the notation involved in the statement of 4.1. 
  Let $\phi : C' \to C$ be a finite separable morphism between smooth projective curves over the field $k$.      
\subsection{Notation}      
We define the genus of a point $p \in C^{\mathrm{an}}$ as follows. If $p \in C^\mathrm{an}$ is of type II then let $g_p$ denote
 the genus of the smooth projective curve $\tilde{C}_p$ which corresponds to the
  $\tilde{k}$-function field $\widetilde{\mathcal{H}(p)}$ and if $p \in C^{\mathrm{an}}$ is not of type II then we set $g_p = 0$. 
  
       Let $p' \in C'^{\mathrm{an}}$ which is of type II and let $p := \phi^{\mathrm{an}}(p')$. The $\tilde{k}$-function fields $\widetilde{\mathcal{H}(p')}$, 
       $\widetilde{\mathcal{H}(p)}$ define smooth projective $\tilde{k}$-curves $\tilde{C}'_{p'}$, $\tilde{C}_p$ respectively. The morphism 
       $\phi^{\mathrm{an}}$ induces an injection $\widetilde{\mathcal{H}(p)} \hookrightarrow \widetilde{\mathcal{H}(p')}$ which implies a morphism 
       $\tilde{C'}_{p'} \to \tilde{C}_p$. This morphism is not necessarily separable. The extension $\widetilde{\mathcal{H}(p)} \hookrightarrow \widetilde{\mathcal{H}(p')}$
       can be decomposed so that there exists an intermediate $\tilde{k}$-function field $\widetilde{I(p',p)}$ and the extension 
       $\widetilde{\mathcal{H}(p)} \hookrightarrow \widetilde{I(p',p)}$ is purely inseparable while 
       $\widetilde{I(p',p)} \hookrightarrow \widetilde{\mathcal{H}(p')}$ is separable of degree $s(p',p)$. Such a decomposition exists by \cite{MOU}. 
       Let $\tilde{C}_{p',p}$ denote the smooth projective $\tilde{k}$-curve which corresponds to the field 
       $\widetilde{I(p',p)}$. By construction, the genus of the curve $\tilde{C}_{p',p}$ is equal to $g_p$. 
       
             The finite separable morphism $\tilde{C'}_{p'} \to \tilde{C}_{p',p}$ can be 
             used to relate the genera of the two curves via the Riemann-Hurwitz formula. 
        As in [\cite{hart}, IV.2], 
         let 
         \begin{align*}
         R_{p',p} :=  \Sigma_{P \in \tilde{C'}_{p'}} \mathrm{length} (\Omega_{\tilde{C'}_{p'}/\tilde{C}_{p',p}})_P .P      
        \end{align*}
             and 
           \begin{align*}
         R :=  \Sigma_{P \in C'} \mathrm{length} (\Omega_{C'/C})_P .P.     
        \end{align*}   
        
            We define invariants on the points of $C^{\mathrm{an}}$ which relate the values 
            $g^{\mathrm{an}}(C'), g^{\mathrm{an}}(C)$ from Definition 2.26. For $p \in C^{\mathrm{an}}$ of type II, let
          \begin{align*} 
          s(p) &:= \Sigma_{p' \in (\phi^{\mathrm{an}})^{-1}(p)} s(p',p), \\ 
          R^1_{p',p} &:= \mathrm{deg}(R_{p',p}) - (2s(p',p) - 2), \\ 
          R^1_p &:=  \Sigma_{p' \in (\phi^{\mathrm{an}})^{-1}(p)} R^1_{p',p}.
          \end{align*}  
          
                  When $p$ is not of type II, let $s(p)$ be the cardinality of the fibre $(\phi^{\mathrm{an}})^{-1}(p)$ and $R^1_p := 0$. 
                  
   \subsection{A Riemann-Hurwitz formula for the analytic genus}                 
            
            \begin{thm} 
      Let $\phi : C' \to C$ be a finite separable morphism between
       smooth projective curves over the field $k$. Let $g^{\mathrm{an}}(C'), g^{\mathrm{an}}(C)$ be as in Definition 2.26.
  We have the following equation. 
  \begin{align*}
     2g^{\mathrm{an}}(C') - 2 = \mathrm{deg}(\phi)(2g^{\mathrm{an}}(C) - 2) +
      \Sigma_{p \in C^{\mathrm{an}}} 2 s(p) g_p  + \mathrm{deg}(R)  - \Sigma_{p \in C^{\mathrm{an}}} R^1_p.
                              \end{align*}
     \end{thm}     
\begin{proof} 
          In order to prove Theorem 4.1, we make use of the fact that there exists a pair of 
          deformation retractions $\psi' : [0,1] \times C'^{\mathrm{an}} \to C'^{\mathrm{an}}$ and
           $\psi : [0,1] \times C^{\mathrm{an}} \to C^{\mathrm{an}}$ 
          which are compatible with the morphism $\phi^{\mathrm{an}}$ (Theorem 3.1). Let 
          $\Upsilon_{C'^{\mathrm{an}}}$ and $\Upsilon_{C^{\mathrm{an}}}$ denote the 
          images of the deformation retractions $\psi'$ and $\psi$ respectively. We can assume that $\Upsilon_{C^{\mathrm{an}}}$ contains the 
          ramification locus of the morphism $\phi$. 
                    Furthermore, there exists weak semistable vertex sets
         $\mathfrak{A} \subset C^{\mathrm{an}}$ and $\mathfrak{A}' \subset C'^{\mathrm{an}}$ such that 
        $\Upsilon_{C^{\mathrm{an}}} = \Sigma(C^{\mathrm{an}},\mathfrak{A})$ and 
          $\Upsilon_{C'^{\mathrm{an}}} = \Sigma(C'^{\mathrm{an}},\mathfrak{A}')$. 
             
             We identify a  set of vertices $V(\Upsilon_{C^{\mathrm{an}}}), V(\Upsilon_{C'^{\mathrm{an}}})$ 
              for the skeleta $\Upsilon_{C^{\mathrm{an}}}$
and $\Upsilon_{C'^{\mathrm{an}}}$ which satisfy the following conditions. 

\begin{enumerate}
 \item  $V(\Upsilon_{C'^{\mathrm{an}}}) = (\phi^{\mathrm{an}})^{-1}(V(\Upsilon_{C^{\mathrm{an}}})).$
  \item  
         $\mathfrak{A} \subset V(\Upsilon_{C^{\mathrm{an}}})$ and 
$\mathfrak{A'} \subset V(\Upsilon_{C'^{\mathrm{an}}})$.  
\item  If $p$ (resp. $p'$) is a point on the skeleton $\Upsilon_{C^{\mathrm{an}}}$ (resp. $\Upsilon_{C'^{\mathrm{an}}}$) for which there exists a sufficiently     
small open neighbourhood $U \subset \Upsilon_{C^{\mathrm{an}}}$
 (resp. $U' \subset \Upsilon_{C'^{\mathrm{an}}}$)
such that $U \smallsetminus \{p\}$ (resp. $U' \smallsetminus \{p'\}$) has atleast three connected components then 
$p \in V(\Upsilon_{C^{\mathrm{an}}})$ (resp. $p' \in V(\Upsilon_{C'^{\mathrm{an}}})$).                
\end{enumerate}

       It can be verified that a pair $(V(\Upsilon_{C^{\mathrm{an}}}),V(\Upsilon_{C'^{\mathrm{an}}}))$
satisfying these properties does indeed exist. 
We define the set of edges $E(\Upsilon_{C^{\mathrm{an}}})$ (resp. $E(\Upsilon_{C'^{\mathrm{an}}})$)
for the skeleton $\Upsilon_{C^{\mathrm{an}}}$ $(\Upsilon_{C'^{\mathrm{an}}})$ to be the collection of all paths contained in 
$\Upsilon_{C^{\mathrm{an}}}$ (resp. $\Upsilon_{C'^{\mathrm{an}}}$) connecting any two vertices. 
Since $\Upsilon_{C^{\mathrm{an}}}$ (resp. $\Upsilon_{C'^{\mathrm{an}}}$) is the skeleton
 associated to a weak semistable vertex set, the edges of the skeleton 
are identified with real intervals. This defines a length function on the set of edges. 

By definition, $g^{\mathrm{an}}(C) = g(\Upsilon_{C^{\mathrm{an}}})$ and 
$g^{\mathrm{an}}(C') = g(\Upsilon_{C'^{\mathrm{an}}})$
The genus formula [\cite{aminibaker}, 4.5] implies 
that 
\begin{align*}
  g(C) = g(\Upsilon_{C^{\mathrm{an}}}) + \Sigma_{p \in V(\Upsilon_{C^{\mathrm{an}}})} g_p 
\end{align*}
     and
\begin{align*}
  g(C') =  g(\Upsilon_{C'^{\mathrm{an}}}) + \Sigma_{p' \in V(\Upsilon_{C'^{\mathrm{an}}})} g_{p'}. 
\end{align*}
       
    By definition, the spaces $C'^{\mathrm{an}} \smallsetminus \mathfrak{A}$ and $C'^{\mathrm{an}} \smallsetminus \mathfrak{A'}$ 
  decompose into the disjoint union of Berkovich open balls and open annuli. It follows that if $p \notin \mathfrak{A}$ or $p' \notin \mathfrak{A'}$ then 
  $g_p = 0$ and $g_{p'} = 0$.  
 As $\mathfrak{A} \subset V(\Upsilon_{C^{\mathrm{an}}})$ and 
$\mathfrak{A'} \subset V(\Upsilon_{C'^{\mathrm{an}}})$, the equations above can be rewritten as 
\begin{align}
  g(C) = g(\Upsilon_{C^{\mathrm{an}}}) + \Sigma_{p \in C^{\mathrm{an}}} g_p 
\end{align}
     and
\begin{align}
  g(C') =  g(\Upsilon_{C'^{\mathrm{an}}}) + \Sigma_{p' \in C'^{\mathrm{an}}} g_{p'}. 
\end{align}
 
   The morphism $\phi : C' \to C$ is a finite separable morphism between smooth, projective curves. 
The Riemann-Hurwitz formula  [\cite{hart}, Corollary IV.2.4]
 enables us to relate the genera of the curves $C'$ and $C$. Precisely, 
\begin{align}
  2g(C') - 2 = \mathrm{deg}(\phi)(2g(C) - 2) + \mathrm{deg}(R)
\end{align}
   where $R$ is a divisor on the curve $C'$ such that if $\phi$ is tamely ramified at $x' \in C'$ then 
$\mathrm{ord}_{x'}(R) = \mathrm{ram}(x',x) -1$. Using the above, we obtain the following equation relating 
$g^{\mathrm{an}}(C')$ and $g^{\mathrm{an}}(C)$. 

\begin{align*}
 2g^{\mathrm{an}}(C') - 2 + 2(\Sigma_{p' \in C'^{\mathrm{an}}} g_{p'} )
= \mathrm{deg}(\phi)(2g^{\mathrm{an}}(C) -2) +  \\ 
   \mathrm{deg}(\phi)(2\Sigma_{p \in C^{\mathrm{an}}} g_p) + \mathrm{deg}(R).
  \end{align*}
  
The only points $p' \in C'^{\mathrm{an}}$ for which $g_{p'} \neq 0$ belong to $\mathfrak{A'}$ and are of type II. 
Let $p'$ be such a point and $p := \phi^{\mathrm{an}}(p')$. 
 Applying the Riemann-Hurwitz formula to the 
extension $\widetilde{I(p',p)} \hookrightarrow \widetilde{\mathcal{H}(p')}$ relates $g_{p'}$ and $g_p$ by the following equation. 
\begin{align*}
  2g_{p'} - 2 =  s(p',p)(2g_p - 2) + \mathrm{deg}(R_{p',p}).
\end{align*}
  
     This equation holds for all points of type II. When $p'$ and $p$ are not of type II, we set $s(p',p) := 1$ and $R_{p',p} := 0$. 
These invariants imply the following equation.      
\begin{align*}
2g^{\mathrm{an}}(C') - 2
= \mathrm{deg}(\phi)(2g^{\mathrm{an}}(C) -2) - \Sigma_{p' \in C'^{\mathrm{an}}} [s(p',p)(2g_p -2) + \mathrm{deg}(R_{p',p}) + 2] + \\
         \mathrm{deg}(\phi)\Sigma_{p \in C^{\mathrm{an}}}(2g_p) + \mathrm{deg}(R). 
\end{align*}
 
    Let $s(p) := \Sigma_{p' \in (\phi^{\mathrm{an}})^{-1}(p)} s(p',p)$, $R^1_{p',p} := \mathrm{deg}(R_{p',p}) - (2s(p',p) - 2)$ 
    and $R^1_p :=  \Sigma_{p' \in (\phi^{\mathrm{an}})^{-1}(p)} R^1_{p',p}$.  
These invariants further simplify the equation above to the following form. 
\begin{align*}
2g^{\mathrm{an}}(C') - 2
= \mathrm{deg}(\phi)(2g^{\mathrm{an}}(C) -2) - \Sigma_{p \in C^{\mathrm{an}}} 2s(p)g_p  - \Sigma_{p \in C^{\mathrm{an}}} R^1_{p} 
           + \mathrm{deg}(R).  
      \end{align*}

\end{proof} 

   The rest of this section is dedicated to studying the invariant $s(p)$
    arising in the equation above. 

\subsubsection{Calculating $i(p')$ and the defect}
  Let $M$ be a non-Archimedean valued field with valuation $v$. Let $|M^*|$ denote the value group and 
$\tilde{M}$ denote the residue field. Let $M'$ be a finite extension of the field $M$ such that the valuation $v$ extends 
uniquely to $M'$. By Ostrowski's lemma, we have the following equality.

 \begin{align*} 
  [M':M] = (|M'^{*}| : |M^*|)[\tilde{M'}:\tilde{M}]c^r.
  \end{align*}

   Here $c$ is the characteristic of the residue field if it is positive and one otherwise. 
The value $d(M',M) := c^r$ is called the \emph{defect} of the extension. 
If $r=0$ then we call the extension $M'/M$ \emph{defectless}.
  
    We now relate this definition to the situation we are dealing with.
 Let $p$ be a point
 of type II belonging to $C^{\mathrm{an}}$ and $p' \in (\phi^{\mathrm{an}})^{-1}(p)$. 
Since the field $k$ is algebraically closed non-Archimedean valued 
and the points $p,p'$ are of type II, the value groups of the fields $\mathcal{H}(p)$ 
and $\mathcal{H}(p')$ remain the same. We have the following equality
\begin{align*}
  [\mathcal{H}(p') : \mathcal{H}(p)] = [\widetilde{\mathcal{H}(p')} : \widetilde{\mathcal{H}(p)}]d(p',p)
\end{align*}
 where $d(p',p)$ is the defect of the extension ${\mathcal{H}(p')} / {\mathcal{H}(p)}$. 

\begin{lem}
  Let $p \in C^{\mathrm{an}}$ and $p' \in (\phi^\mathrm{an})^{-1}(p)$. 
The extension $\mathcal{H}(p) \hookrightarrow \mathcal{H}(p')$ is defectless i.e.
$d(p',p) = 1$. 
\end{lem}
\begin{proof}
 We make use of the Poincaré-Lelong theorem and our construction in Section 3
of the pair of compatible deformation retractions $\psi$ and $\psi'$.  
Let $r \in [0,1]$ be the smallest real number such that
$p \in \psi(r,C(k)) = \{\psi(r,x) | x \in C(k)\}$. Since the deformation retractions are compatible it follows that 
if $p' \in (\phi^{\mathrm{an}})^{-1}(p)$ then $p' \in \psi'(r,C'(k))$. 

  Let $x \in C(k)$ be such that $\psi(r,x) = p$. Observe that our choice of $\Upsilon_{C^{\mathrm{an}}}$ implies 
  that the morphism $\phi$ is unramified over $x$. Let $P_x$ denote the path $\psi(\_,x) : [0,r] \to C^{\mathrm{an}}$. 
Given a simple neighborhood [\cite{BPR}, Definition 4.28], $U$ of $p$, the germ of the path $P_x$ at $p$ lies 
in a connected component of $U \smallsetminus \{p\}$ and hence defines an element of the tangent space which we refer to as $e_x$. 
Equivalently, for some $a > 0$, the path $(P_x)_{|[a,r]} \circ -\mathrm{exp} : [-\mathrm{log}(a), -\mathrm{log}(r)] \to C^{\mathrm{an}}$
 is a geodesic segment and its germ defines the element $e_x$ of the tangent space $T_p$ (cf. Remark 2.22). 
Let $t_x$ be a uniformisant of 
the local ring $O_{C,x}$ such that $|t_x(p)| = 1$ and $\tilde{t}_x$- the image of $t_x$
in the residue field $\widetilde{\mathcal{H}(p)}$ is a uniformisant at the point 
$e_x$ in $\widetilde{\mathcal{H}(p)}$. This can be accomplished by choosing $t_x$ so that it has no zeros or poles at any 
$k$-point $y$ for which the path $\psi(\_,y) : [0,r] \to C^{\mathrm{an}}$ coincides with $e_x$
in the tangent space and using the Poincar\'e-Lelong theorem. Such a choice is possible by 
the semistable decomposition associated to 
the skeleton $\Upsilon_{C^{\mathrm{an}}}$. Let $t'_x$ denote the image of $t_x$ in the function field $k(C')$. 
Our choice of $t_x$ implies that for every $p' \in (\phi^{\mathrm{an}})^{-1}(p)$, $|t'_x(p')| = 1$.   
  
  The inclusion $\mathcal{H}(p) \hookrightarrow \mathcal{H}(p')$ induces an inclusion 
of $\tilde{k}$-function fields $\widetilde{\mathcal{H}(p)} \hookrightarrow \widetilde{\mathcal{H}(p')}$. 
As before, let $\tilde{C}_p$ and $\tilde{C}'_{p'}$ denote the smooth projective curves associated to these function fields. 
As explained above, the path $P_x$ defines a $\tilde{k}$-point $e_x$ of the curve $\tilde{C}_p$. 
Let $E(e_x,p')$ denote the set of preimages of this point on the curve $\tilde{C}'_{p'}$. 

  Let $S:= \{x'_1,...x'_k\}$ denote the preimages of the point $x$ and 
$\mathrm{ram}(x'_i,x)$ denote the ramification index of the morphism $\phi$ at the point $x'_i$.
Since the skeleton $\Upsilon_{C^{\mathrm{an}}}$ contains the set of $k$-points over which 
the morphism is ramified, we have that $\mathrm{ram}(x'_i,x) = 1$ for all $i$. 
If $x' \in S$ then the path $P_{x'} := \psi'(\_,x') : [0,r] \to C^{\mathrm{an}}$ defines an element of the tangent space at $\psi'(r,x')$. 
Indeed, if $U'$ is a simple neighborhood of the point $\psi'(r,x')$ then there exists $a \in [0,r)$ such that 
${P_{x'}}_{|[a,r)}$ is contained in exactly one connected component of the space $U' \smallsetminus \psi'(r,x')$. 

The set of elements of the tangent spaces $T_{p'}$ for $p' \in (\phi^{\mathrm{an}})^{-1}(p)$
that are defined by the paths
 $\{\psi'(\_,x'_i) : [0,r] \to C^{\mathrm{an}}\}$ coincides with the set  
$E(e_x) := \cup_{p' \in (\phi^{\mathrm{an}})^{-1}(p)} E(e_x,p')$. 
Our choice of $t_x$ implies that if $y' \in C'(k) \smallsetminus  \phi^{-1}(x)$ and 
$\psi'(\_,y') : [0,r] \to C'^{\mathrm{an}} \in E(e_x)$ then  
$t'_x$ cannot have a zero or pole at $y'$.  

  For $p' \in (\phi^{\mathrm{an}})^{-1}(p)$ and $e' \in E(e_x,p')$, let $S_{e',p'}$ be the collection of those  
$x' \in S$ such that $\psi'(r,x') = p'$ and $\psi'(\_,x') : [0,r] \to C^{\mathrm{an}} = e'$. 
The non-Archimedean Poincaré-Lelong theorem implies that 
\begin{align*}
\delta_{e'}(-|\mathrm{log}(t'_x)|)(p') = \Sigma_{x' \in S_{e',p'}} \mathrm{ram}(x',x) = \mathrm{card}(S_{e',p'}).
\end{align*}
   The second equality follows from the fact that $\mathrm{ram}(x',x) = 1$. Furthermore,
\begin{align*}
 \delta_{e'}(-|\mathrm{log}(t'_x)|)(p') = \mathrm{ord}_{e'}(\tilde{t}'_x). 
\end{align*}
   Since $\Sigma_{e' \in E(e_x,p')} \mathrm{ord}_{e'}(\tilde{t'}_x) = [\widetilde{\mathcal{H}(p')} : \widetilde{\mathcal{H}(p)}]$, we must have that 
   \begin{align*}
 \Sigma_{p' \in (\phi^{\mathrm{an}})^{-1}(p)} [\widetilde{\mathcal{H}(p')} : \widetilde{\mathcal{H}(p)}] = \Sigma_{x' \in S} \mathrm{ram}(x',x)
= \Sigma_{e',p'} \mathrm{card}(S_{e',p'}).
\end{align*}
  Hence 
\begin{align*}
  \Sigma_{p' \in (\phi^{\mathrm{an}})^{-1}(p)} [\widetilde{\mathcal{H}(p')} : \widetilde{\mathcal{H}(p)}] = \mathrm{card}(S). 
\end{align*}

   As the field $k$ is algebraically closed, the expression on the right is equal to the degree of the morphism $\phi$ and 
 we have that 
\begin{align*}
  \Sigma_{p' \in (\phi^{\mathrm{an}})^{-1}(p)} 
[\widetilde{\mathcal{H}(p')} : \widetilde{\mathcal{H}(p)}] = \Sigma_{p' \in (\phi^{\mathrm{an}})^{-1}(p)} [\mathcal{H}(p') : \mathcal{H}(p)].  
\end{align*}
 This implies that for every $p' \in (\phi^{\mathrm{an}})^{-1}(p)$ the extension 
$\mathcal{H}(p) \hookrightarrow \mathcal{H}(p')$ is defectless.      
\end{proof}

    The result above follows from the more general fact that the residue field $\mathcal{H}(p)$ associated to a point $p$ of type II on the analytification $C^{\mathrm{an}}$ of 
     a $k$-curve 
    $C$ is stable [\cite{TEM}, Corollary 6.3.6], \cite{DUC}. 
    Lemma 4.2 can in fact be used to prove this result. 
        Propositions 2 and 4 of Section 3.6 in \cite{BGR} allow us to give the following definition of a stable field which is complete. 
 \begin{defi} 
  \emph{A complete field $K$ is} stable \emph{if and only if for every finite separable field extension $L/K$ 
the following equality holds 
\begin{align*}
    [L:K] = [|L^*| : |K^*|][\widetilde{L} : \widetilde{K}]. 
\end{align*}} 
 \end{defi}    
      
 \begin{prop} 
 Let $S$ be a $k$-curve. Let $p \in S^{\mathrm{an}}$ be a point of type II. The complete 
 field $\mathcal{H}(p)$ is stable. 
 \end{prop}
\begin{proof} 
Let $S_i$ be an irreducible component of $S$ such that $p \in S_i^{\mathrm{an}}$. 
     Let $S'_i$ denote the normalisation of $S_i$. There exists a finite set of $k$-points $F'$ and $F$ in 
     $S'_i$ and $S_i$ respectively such that ${S'}_i^{\mathrm{an}} \smallsetminus F' = S_i^{\mathrm{an}} \smallsetminus F$. 
     It follows that we may reduce to the case when $S$ is smooth, projective and integral.  
     
      Let $L$ be a finite separable extension of $\mathcal{H}(p)$. By definition, $\mathcal{H}(p)$ is the completion of the function 
      field of the curve $S$ with respect to the valuation associated to $p$. Let $L_0$ denote the integral closure of 
      $k(S)$ in $L$. By construction, $L_0$ is a finite separable field extension of $k(S)$. Hence there exists a smooth projective $k$-curve $S'$ such that 
      $k(S') = L_0$. 
      Let $\widehat{L_0}$ denote the completion of $L_0$ induced by the restriction of the valuation 
      of $L$. We claim that $\widehat{L_0} = L$. Let $\alpha \in L$. By construction $\alpha$ is algebraic over $\widehat{L_0}$. 
     Let $g \in \widehat{L_0}[X]$ be the minimal polynomial of $\alpha$. Suppose that $g$ is not a monomial.  
     As $\widehat{L_0}$ is the completion of 
     $L_0$,
      there exists a sequence $(f_i)_i$ of polynomials of $\mathrm{deg}(g)$ in $L_0[X]$ which converge to $g$ with respect to the Gauss norm i.e.
      the coefficients of the $(f_i)_i$ converge to the coefficients of $g$. 
      Let $\alpha_i$ denote a root of $f_i$ for each $i$. 
      By Corollary 2 of [\cite{BGR}, Section 3.4], there exists a sub sequence of $(\alpha_i)_i$ which converges to a root $\alpha'$ of $g$. 
      As the $\alpha_i$ are algebraic over $L_0$ they must be algebraic over $k(S)$. 
      Furthemore, by Proposition 3 in [\cite{BGR}, Section 3.4], for large enough $i$ we must have that $\alpha_i \in L$. 
       Hence by definition of $L_0$,
      $\alpha_i \in L_0$. Hence $\alpha' \in \widehat{L_0}$. This implies a contradiction to our assumption that $g$ is irreducible 
      and of degree greater than or 
      equal to $2$. 
      Consequently, $\alpha \in \widehat{L_0}$ and $\widehat{L_0} = L$. 
       Hence there exists a point $p'$ of type II on $S'^{\mathrm{an}}$ such that 
       $\mathcal{H}(p') = L$. The proposition follows from Lemma 4.2.   
       
\end{proof}

We study the invariants $s(p)$ and $s(p',p)$ of Theorem 4.1 using the deformation retractions $\psi$ and $\psi'$. 

\begin{defi} (The equivalence relation $\sim_{i(r)}$)
\emph{Let $r \in [0,1]$. We define an equivalence relation $\sim_{i(r)}$ on 
$C'(k)$ as follows. We set $x'_1 \sim_{i(r)} x'_2$ 
if and only if 
$\phi(x'_1) = \phi(x'_2)$, $\psi'(r,x'_1) = \psi'(r,x'_2)$ and 
the elements of the tangent space $T_{\psi'(r,x'_1)} = T_{\psi'(r,x'_2)}$  
defined by the 
 paths $\psi'(\_,x'_1) : [0,r] \to C'^{\mathrm{an}}$, $\psi'(\_,x'_2) : [0,r] \to C'^{\mathrm{an}}$ coincide. 
 For $x' \in C'(k)$, let $\mathrm{card} [x']_{i(r)}$ be the cardinality of the equivalence class 
which contains $x'$. }
\end{defi}

\begin{defi} (The real number $r_p$, the set $Q_{p',p}$ and the invariant $i(p',p)$)
\emph{Let $p \in C^{\mathrm{an}}$ be a point which is not
 of type IV and $p' \in (\phi^{\mathrm{an}})^{-1}(p)$. 
\begin{enumerate}
\item  We define $r_p \in [0,1]$ to be the smallest real number for which 
$p \in \psi(r_p,C(k))$ where $\psi(r_p,C(k)) := \{\psi(r_p,x) | x \in C(k)\}$. 
\item We 
define
 $Q_{p',p} := \{x' \in C'(k)|\psi'(r_p,x') = p'\}$. 
\item Let $i(p',p) := \mathrm{min}_{x' \in Q_{p',p}} \{\mathrm{card} [x']_{i(r_p)}\}$. 
\end{enumerate}  
}
 \end{defi} 
  
 Recall that if $p$ is a point of type II then we used $s(p',p)$ to denote the separable degree of 
the field extension $\widetilde{\mathcal{H}(p)} \hookrightarrow \widetilde{\mathcal{H}(p')}$ and we set 
$s(p',p) = 1$ otherwise. 

\begin{prop} 
   Let $p' \in C'^{\mathrm{an}}, p := \phi^{\mathrm{an}}(p')$.
  \begin{enumerate}
  \item  When $p$ is of type II,   
  the number $i(p',p)$ (cf. Definition 4.6) is the degree of 
inseparability of the extension $\widetilde{\mathcal{H}(p)} \hookrightarrow \widetilde{\mathcal{H}(p')}$. Hence 
$s(p',p) = [\mathcal{H}(p'):\mathcal{H}(p)]/i(p',p)$. 
  \item  When $p$ is not of type II or IV, $s(p) := \Sigma_{p' \in (\phi^{\mathrm{an}})^{-1}(p)} s(p',p)$ is the number of $\sim_{i(r(p))}$
equivalence classes in $\phi^{-1}(x)$ for any $x \in C(k)$ such that $\psi(r_p,x) = p$. 
  \end{enumerate}   
 \end{prop}
 \begin{proof}  The second assertion can be verified directly and we restrict to proving the proposition for points of type II. 
   Let $\tilde{C}_p$ and $\tilde{C}'_{p'}$ denote the smooth projective curves corresponding to 
the function fields $\widetilde{\mathcal{H}(p)}$ and $\widetilde{\mathcal{H}(p')}$ respectively.   
For a point $e \in \tilde{C}_p$, let $s_e$ denote the uniformisant of the local ring $O_{\tilde{C}_p,e}$. 

   Let $x \in C(k)$ be such that $\psi(r_p,x) = p$. Since $\Upsilon_{C^{\mathrm{an}}}$ contains 
every $k$-point over which the morphism $\phi$ is ramified, $\phi$ is unramified over $x$. 
Let $e_x \in \tilde{C}_p$ correspond to the path 
$\psi(\_,x) : [0,r_p] \to C^{\mathrm{an}}$. 
Let $t_x$ be a uniformisant 
of $x$ such that $|t_x(p)| = 1$
and it does not have any zeros or poles at any $y$ for which the element of the tangent space $T_p$ defined by 
$\psi(\_,y) : [0,r_p] \to C^{\mathrm{an}}$ coincides with $e_x$.

 It follows that the image of $t_x$ in the field $\widetilde{{\mathcal{H}(p)}}$
is a uniformisant at the point $e_x$. We can hence assume $\tilde{t}_x = s_e$. 
 Let $e' \in \tilde{C}'_{p'}$ map to $e_x$ and 
$y' \in C'(k)$ be such that the element of $T_{p'}$ defined by the path $\psi'(\_,y') : [0,r_p] \to C'^{\mathrm{an}}$ coincides with $e'$. 

   By the Non-Archimedean Poincaré-Lelong Theorem,
the order of vanishing of the uniformisant $\tilde{t}_x$ at $e'$ is equal
to the cardinality of the equivalence class $[y']_{i(r_p)}$. 
The inseparable degree of $\widetilde{\mathcal{H}(p')}/\widetilde{\mathcal{H}(p)}$
is equal to 
$\mathrm{min}_{\{(e',e)|e' \in \tilde{C}'_{p'}, e' \mapsto e\}}  \{\mathrm{ord}_{e'}(s_e)\}$ i.e. 
$\mathrm{min}_{\{(e',e)|e' \in \tilde{C}'_{p'}, e' \mapsto e\}}  \{\mathrm{ord}_{e'}(\tilde{t}_x)\}$.
Hence $i(p',p) = \mathrm{min}_{x' \in Q_{p',p}} \{\mathrm{card} [x']_{i(r(p))}\}$ is the 
degree of inseparability of the extension $\widetilde{\mathcal{H}(p')}/\widetilde{\mathcal{H}(p)}$. 
The equality
$s(p',p) = [\mathcal{H}(p'):\mathcal{H}(p)]/i(p',p)$ follows from Lemma 4.2.  
 \end{proof}
 
 \begin{defi} \emph{
  Let $p \in C^{\mathrm{an}}$. 
  We define $i(p) := \Sigma_{p' \in  (\phi^{\mathrm{an}})^{-1}(p)} [\mathcal{H}(p'):\mathcal{H}(p)]/i(p',p)$ 
  where $i(p',p)$ is as in Definition 4.6.}  
 \end{defi}    
      Proposition 4.7,
     Theorem 4.1 and the fact that $g_p = 0$ when $p$ is not of type II imply the following corollary. 
      
        \begin{cor} 
      Let $\phi : C' \to C$ be a finite separable 
      morphism between smooth projective curves over the field $k$. Let $g^{\mathrm{an}}(C'), g^{\mathrm{an}}(C)$ be as in Definition 2.26.
  We have the following equation. 
  \begin{align*}
     2g^{\mathrm{an}}(C') - 2 = \mathrm{deg}(\phi)(2g^{\mathrm{an}}(C) - 2) + \Sigma_{p \in C^{\mathrm{an}}} 2 i(p) g_p  + \mathrm{deg}R  - \Sigma_{p \in C^{\mathrm{an}}} R^1_p.
                              \end{align*}
     \end{cor}  
 
\section{A second calculation of $g^{\mathrm{an}}(C')$}

       Let $\phi : C' \to C$ be a finite morphism between smooth projective curves over the field $k$. 
       Our results in Section 3 imply the existence of a pair of deformation retractions $\psi'$, $\psi$ on 
       $C'^{\mathrm{an}}$ and $C^{\mathrm{an}}$ which are compatible with the morphism 
       $\phi^{\mathrm{an}}$. We choose $\psi$ and $\psi'$ as in the proof of Theorem 4.1. 
       Let $\Upsilon_{C'^{\mathrm{an}}}$ and $\Upsilon_{C^{\mathrm{an}}}$ denote the images of the retractions $\psi'$ and $\psi$ respectively. 
       The deformation retractions $\psi, \psi'$ can be constructed so that $\Upsilon_{C^{\mathrm{an}}}$ contains 
       those points of $C(k)$ over which the morphism $\phi$ is ramified and does not contain any point of type IV. 
       We have that $g^{\mathrm{an}}(C') = g(\Upsilon_{C'^{\mathrm{an}}})$ and $g^{\mathrm{an}}(C) = g(\Upsilon_{C^{\mathrm{an}}})$. 
        
  \begin{defi}
    \emph{A} divisor \emph{on a finite metric graph is an element of the free abelian group generated by the points of the graph.} 
  \end{defi}  
    
       As outlined in the introduction, in this section we introduce a divisor $w$ on the skeleton
$\Upsilon_{C^{\mathrm{an}}}$ and relate the degree of this divisor to the genus of the skeleton $\Upsilon_{C'^{\mathrm{an}}}$. 
The point of doing so is to study how $g(\Upsilon_{C'^{\mathrm{an}}})$ can be calculated in terms of $g(\Upsilon_{C^{\mathrm{an}}})$ 
and the
behaviour of the morphism 
between the sets of vertices. 
   
   We preserve our choices of vertex sets and edge sets for the two skeleta from the proof of Theorem 4.1. 

 \begin{defi} (The invariant $n_p$ and the sets of tangent directions $E_p$, $L(e_p,p')$.) 
\emph{Let $p \in \Upsilon_{C^{\mathrm{an}}}$.
 \begin{enumerate}
 \item Let $n_p$ denote the number of preimages of $p$ for the morphism $\phi^{\mathrm{an}}$. 
 \item Let $T_p$ denote the 
    tangent space at the point $p$ (cf. 2.2.3, 2.4.1).
     We define $E_{p,\Upsilon_{C^{\mathrm{an}}}} \subset T_p$ to be those elements for which there exists a representative
 starting from $p$ and contained completely in 
$\Upsilon_{C^{\mathrm{an}}}$. When there is no ambiguity concerning the graph $\Upsilon_{C^{\mathrm{an}}}$, we simplify 
notation and write $E_p$. 
 \item For any $p' \in C'^{\mathrm{an}}$ such that 
 $\phi^{\mathrm{an}}(p') = p$, 
the morphism $\phi^{\mathrm{an}}$ induces a map $d\phi_{p'}$ between the tangent spaces $T_{p'}$ and $T_{p}$ (cf. 2.2.3, 2.4.1).
For $p' \in (\phi^{\mathrm{an}})^{-1}(p)$
and $e_p \in E_p$, we define $L(e_p,p') \subset T_{p'}$ to be the set of preimages of $e_p$ for the map $d\phi_{p'}$
and $l(e_p,p')$ to be the cardinality of the set $L(e_p,p')$. 
   \end{enumerate}} 
 \end{defi}

Observe that   
as $\Upsilon_{C'^{\mathrm{an}}} = (\phi^{\mathrm{an}})^{-1}(\Upsilon_{C^{\mathrm{an}}})$, any
element of $L(e_p,p')$ can be represented by a geodesic segment that is contained completely in 
$\Upsilon_{C'^{\mathrm{an}}}$. 

\begin{defi} (The divisor $w$ of $\Upsilon_{C^{\mathrm{an}}}$)
\emph{Let the notation be as in Definition 5.2. For a point $p \in \Upsilon_{C^{\mathrm{an}}}$, 
let $w(p) := (\sum_{e_p \in E_p, p' \in (\phi^{\mathrm{an}})^{-1}(p)} l(e_p,p')) - 2n_p$.
We define $w$ to be the divisor $\Sigma_{p \in \Upsilon_{C^{\mathrm{an}}}} w(p)p$.} 
\end{defi}

\begin{prop}
   The degree of the divisor $w$ is equal to $2g(\Upsilon_{C'^{\mathrm{an}}}) - 2$. 
\end{prop}
\begin{proof}
    We begin by stating the following fact concerning connected, finite metric graphs. Let $\Sigma$ be
 a connected, finite metric graph. Let $p \in \Sigma$. Let $U$ be a simply connected neighborhood 
 of $p$ in $\Sigma$. We define $t_p$ to be the cardinality of the set of connected components of the space 
 $U \smallsetminus \{p\}$ and 
$D_\Sigma := \sum_{p \in \Sigma} (t_p - 2)p$. It can be verified that
 $D_\Sigma$ is a divisor on the finite graph $\Sigma$ whose degree is equal to 
$2g(\Sigma) - 2$. 
 
  The connected, finite graphs $\Upsilon_{C'^{\mathrm{an}}}$ and $\Upsilon_{C^{\mathrm{an}}}$ are the images of 
a pair of compatible deformation retractions. Hence the morphism $\phi^{\mathrm{an}}$ restricts 
to a continuous map $\Upsilon_{C'^{\mathrm{an}}} \to \Upsilon_{C^{\mathrm{an}}}$. This map induces a homomorphism 
$\phi_* : \mathrm{Div}(\Upsilon_{C'^{\mathrm{an}}}) \to \mathrm{\mathrm{Div}}(\Upsilon_{C^{\mathrm{an}}})$ defined as follows. We define $\phi_*$ only
on the generators of the group $\mathrm{Div}(\Upsilon_{C'^{\mathrm{an}}})$. If $1.p' \in \mathrm{Div}(\Upsilon_{C'^{\mathrm{an}}})$ then 
we set $\phi_*(1.p') = 1.\phi(p')$. Note that for any divisor $D' \in \mathrm{Div}(\Upsilon_{C'^{\mathrm{an}}})$,
$\mathrm{deg}(\phi_*(D')) = \mathrm{deg}(D')$.      

    We will show that $w = \phi_*(D_{\Upsilon_{C'^{\mathrm{an}}}})$. By definition, 
\begin{align*}
 \phi_*(D_{\Upsilon_{C'^{\mathrm{an}}}})(p) = (\sum_{p' \in (\phi^{\mathrm{an}})^{-1}(p)} t_{p'}) - 2n_p .
\end{align*}

  Let $p' \in \Upsilon_{C'^{\mathrm{an}}}$ and $p = \phi^{\mathrm{an}}(p')$.  
 We must have that the number of distinct germs of geodesic segments starting from 
 $p'$ and contained in $\Upsilon_{C'^{\mathrm{an}}}$ is $t_{p'}$. 
  We have a map $d\phi_{p'} : T_{p'} \to T_p$ which maps 
 germs of geodesic segments starting at $p'$ to germs of geodesic segments starting at $p$. 
 As $\Upsilon_{C'^{\mathrm{an}}} = (\phi^{\mathrm{an}})^{-1}(\Upsilon_{C^{\mathrm{an}}})$, we must have that 
the image via $d\phi_{p'}$ of a germ 
for which there exists a representative contained in $\Upsilon_{C'^{\mathrm{an}}}$
and starting from $p'$ must be a germ starting at $p$ for which there exists a representative 
contained in $\Upsilon_{C^{\mathrm{an}}}$. Likewise, if $e_p$ is a germ starting at $p$
which has a representative contained in $\Upsilon_{C^{\mathrm{an}}}$
then its preimage for the map $d\phi_{p'}$ is a germ starting at $p'$ for which there exists
a representative contained in $\Upsilon_{C'^{\mathrm{an}}}$.
It follows that 
$\sum_{p' \in (\phi^{\mathrm{an}})^{-1}(p)}t_{p'} =  \sum_{e_p \in E_p, p' \in \phi^{-1}(p)} l(e_p,p')$. 
Hence $\phi_*(D_{\Upsilon_{C^{\mathrm{an}}}}) = w$.  
\end{proof}

\subsection{Calculating $n_p$}

We extend the invariant $n_p$ of Definition 5.2 to all points of $C^{\mathrm{an}}$. 
\begin{defi}(The invariant $n_p$)
\emph{Let $p \in C^{\mathrm{an}}$. 
 Let $n_p$ denote the number of preimages of $p$ for the morphism $\phi^{\mathrm{an}}$.}
\end{defi} 

  In this section we study $n_p$ for $p \in C^{\mathrm{an}}$ with the added restriction 
  that the extension of function fields $k(C) \hookrightarrow k(C')$
associated to the morphism $\phi$ is Galois.

\begin{defi}(The invariant $\mathrm{ram}(p)$ for $p \in C^{\mathrm{an}}$)
\emph{ 
Let $p \in C^{\mathrm{an}}$.
\begin{enumerate}
\item Let $p$ be a point of type I i.e. $p \in C(k)$. 
 Let $p' \in C'(k)$ such that $\phi(p') = p$. Let 
$\mathrm{ram}(p',p)$ denote the ramification degree associated to the extension of the discrete 
valuation rings $O_{C,p} \hookrightarrow O_{C',p'}$. 
Since the morphism $\phi$ is Galois, 
for $p \in C(k)$, the ramification degree $\mathrm{ram}(p',p)$ is a constant as $p'$ varies along the set of preimages 
of the point $p$. The ramification degree depends only on the point $p \in C(k)$ and we denote it $\mathrm{ram}(p)$. 
As $k$ is algebraically closed we have that
\begin{align*}
[k(C') : k(C)] =  n_p \mathrm{ram}(p).
 \end{align*}
 \item When $p$ is not of type I, we define $\mathrm{ram}(p) := 1$.  
\end{enumerate}}
\end{defi}

    Let $p$ be a point of $C^{\mathrm{an}}$ which is not of type IV. Recall that $r_p$ is the smallest real number in 
the real interval $[0,1]$ such that $p$ belongs to $\psi(r_p,C(k)) = \{\psi(r_p,x) | x \in C(k)\}$. Since the pair of deformation retractions 
$\psi'$ and $\psi$ are compatible with the morphism $\phi^{\mathrm{an}}$,
 we must have that $(\phi^{\mathrm{an}})^{-1}(p) \subset \psi'(r_p,C'(k))$.    

\begin{defi}(The equivalence relation $\sim_{c(r)}$ on $C(k)$) 
  \emph{ For $r \in [0,1]$, we define an equivalence relation $\sim_{c(r)}$ on the set of $k$-points of the curve $C'$. Let 
$x'_1,x'_2 \in C'(k)$. We set $x'_1 \sim_{c(r)} x'_2$ if $\phi(x'_1) = \phi(x'_2)$ and   
$\psi'(r,x'_1) = \psi'(r,x'_2)$. Observe that each equivalence class is finite.   
For $x' \in C'(k)$, let $[x']_{c(r)}$ denote that equivalence class containing the point $x'$.} 
\end{defi} 

\begin{lem}
  If $x'_1, x'_2 \in C'(k)$ such that $\phi(x'_1) = \phi(x'_2)$ then 
\begin{align*} 
 \mathrm{card} [x'_1]_{c(r)} =  \mathrm{card} [x'_2]_{c(r)} 
\end{align*}  
    for all $r \in [0,1]$. 
\end{lem}
\begin{proof}
   The lemma is tautological when $x'_1 \sim_{c(r)} x'_2$. Let us hence assume that
$\psi'(r,x'_1) = p'_1$ and $\psi'(r,x'_2) = p'_2$ where $p'_1$ and $p'_2$ are 
two points on $C'^{\mathrm{an}}$. 
Observe that since $\Upsilon_{C'^{\mathrm{an}}}$ and 
$\Upsilon_{C^{\mathrm{an}}}$ are the images of a pair of compatible deformation retractions, 
$\phi^{\mathrm{an}}(p'_1) = \phi^{\mathrm{an}}(p'_2)$. Let $p := \phi^{\mathrm{an}}(p'_1)$.  
 The Galois group $G := \mathrm{Gal}(k(C')/k(C))$ 
acts trasitively on the set of preimages $\phi^{-1}(p)$. Let $\sigma \in G$ 
be an element of the Galois group such that $\sigma(p'_1) = p'_2$. By construction, the deformation 
retraction $\psi'$ is Galois invariant i.e. if 
$t \in [0,1], q \in C'^{\mathrm{an}}$ and $g \in \mathrm{Gal}(k(C')/k(C))$ then 
$\psi'(t,g(q)) = g(\psi'(t,q))$. 
It follows that
 if $a \sim_{c(r)} x'_1$ then
$\sigma(a) \sim_{c(r)} x'_2$. As $\sigma$ is bijective, $\mathrm{card} [x'_1]_{c(r)} \leq  \mathrm{card} [x'_2]_{c(r)}$.  
By symmetry we conclude that the lemma is true.  
  \end{proof} 
 
 \begin{defi} (The invariant $c_r(x)$ for $x \in C(k)$) 
 \emph{ Let $x \in C(k)$ and $x' \in C'(k)$ such that $\phi(x') = x$. We define
\begin{align*}
   c_r(x) := \mathrm{card} [x']_{c(r)}.  
\end{align*}
  Lemma 5.8 implies that $c_r(x)$ is well defined.}  
\end{defi}

\begin{prop}
  Let $p \in C^{\mathrm{an}}$ be a point which is not of type IV. 
We have the following equality.  
\begin{align*}
    n_p = [k(C') : k(C)]/ (c_{r_p}(x)\mathrm{ram}(x))
\end{align*}
     for any $x \in C(k)$ such that $\psi(r_p,x) = p$. 
\end{prop}
\begin{proof}
    When $p \in C(k)$, we must have that if $x \in C(k)$ is such that $\psi(r_p,x) = p$ then 
    $x = p$ and $r_p = 0$. Hence $c_{r_p}(p) = 1$ and the proposition amounts to showing that 
    $[k(C') : k(C)] = n_p\mathrm{ram}(p)$ which is a well known calculation. 
     
     Suppose $p \in C^{\mathrm{an}} \smallsetminus C(k)$.  
       Let $x \in C(k)$ be such that $\psi(r_p,x) = p$. As the deformation retractions $\psi$ and $\psi'$ are 
      compatible we must have that $\psi'(r_p,y) \in (\phi^{\mathrm{an}})^{-1}(p)$ for every       
       $y \in \phi^{-1}(x)$. Furthermore, given $q \in C'^{\mathrm{an}}$ which maps to $p$ via 
       $\phi^{\mathrm{an}}$, there exists a $y \in \phi^{-1}(x)$ such that $\psi'(r_p,y) = q$.
       This can be deduced from the Galois invariance of the deformation retraction $\psi'$.
        As $\psi$ fixes the points of 
       $C(k)$ which are ramified, we must have that if $x \in C(k)$ such
        that $\psi(r_p,x) = p$ then $n_x = [k(C') : k(C)]$ and $\mathrm{ram}(x) = 1$.  
       The proposition can be deduced from these observations. 
       \end{proof}
       
 Observe that if $x \in C(k)$ is such that $\psi(r_p,x) = p$ then 
 $c_{r_p}(x) = c_s(x)$ for any $s \in [r_p,1]$.  
This observation and Proposition 5.10 allow 
us to define the following invariant - $c_{1}(p)$ for $p \in \Upsilon_{C^{\mathrm{an}}}$.     
   
\begin{defi}(The invariant $c_1(p)$ for $p \in \Upsilon_{C^{\mathrm{an}}}$)
\emph{Let $p \in \Upsilon_{C^{\mathrm{an}}}$. 
The function $c_1 : C(k) \to \mathbb{Z}_{\geq 0}$ factors through $\Upsilon_{C^{\mathrm{an}}}$ via the retraction $\psi(1,\_)$.
Hence we have $c_1 : \Upsilon_{C^{\mathrm{an}}} \to  \mathbb{Z}_{\geq 0}$. By Proposition 5.10,} 
\begin{align*}
    n_p = [k(C') : k(C)]/ (c_1(p)\mathrm{ram}(p))
\end{align*}
    \emph{where $\mathrm{ram}(p)$ is as in Definition 5.6.} 
  \end{defi} 

\subsection{Calculating $l(e_p,p')$}

\begin{lem}
   Let $p \in \Upsilon_{C^{\mathrm{an}}}$ and $e_p \in E_p$. Then $l(e_p,p')$ is a constant as $p'$ varies through the set of preimages of
   $p$ for 
the morphism $\phi^{\mathrm{an}}$. 
\end{lem}
\begin{proof}
   Let $p'_1,p'_2 \in (\phi^{\mathrm{an}})^{-1}(p)$. The Galois group $\mathrm{Gal}(k(C')/k(C))$ acts transitively on the set
of preimages of the point $p$. 
As $\Upsilon_{C'^{\mathrm{an}}} = (\phi^{\mathrm{an}})^{-1}(\Upsilon_{C^{\mathrm{an}}})$,  
the elements of the Galois group are homeomorphisms on $C'^{\mathrm{an}}$ which restrict 
to homeomorphisms on $\Upsilon_{C'^{\mathrm{an}}}$.   
It follows that if $\sigma \in \mathrm{Gal}(k(C')/k(C))$ is such that $\sigma(p'_1) = p'_2$ then $\sigma$ maps the set
of germs $L(e_p,p'_1)$ injectively to the set $L(e_p,p'_2)$. 
By symmetry, we conclude that our proof is complete. 
\end{proof}

\begin{defi} (The invariants $l(e_p)$ and $\widetilde{\mathrm{ram}}(e_p)$ for $p \in \Upsilon_{C^{\mathrm{an}}}$ and $e_p \in E_p$)
\emph{Let $p \in \Upsilon_{C^{\mathrm{an}}}$ and $e_p \in E_p$ (Definition 5.2).
\begin{enumerate}
\item We define $l(e_p) := l(e_p,p')$ for any $p' \in (\phi^{\mathrm{an}})^{-1}(p)$. Lemma 5.12 implies that 
$l(e_p)$ is well defined.    
\item 
\begin{enumerate} 
\item By Section 2.4.3, when $p$ is a point of type II,
$e_p$ corresponds to a discrete valuation of the $\tilde{k}$-function field $\widetilde{\mathcal{H}(p)}$.
For any $p' \in (\phi^{\mathrm{an}})^{-1}(p)$,
the extension of fields $\widetilde{\mathcal{H}(p)} \hookrightarrow \widetilde{\mathcal{H}(p')}$ can be decomposed into the composite of a 
purely inseparable extension and a Galois extension. Hence the ramification 
degree ${\mathrm{ram}}(e'/e_p)$ is constant as $e'$ varies through the set of preimages of $e_p$
for the morphism $d\phi^{alg}_{p'} : T_{p'} \to T_p$ (cf. 2.4.1). Let 
$\widetilde{\mathrm{ram}}(e_p)$ be this number.
\item
 When $p$ is of type I,
the set $E_p$ contains only one element and
 we set $\widetilde{\mathrm{ram}}(e_p) := \mathrm{ram}(p)$.
\item When $p$ is of type III, 
let $\widetilde{\mathrm{ram}}(e_p) := c_1(p)$. 
\end{enumerate}
\end{enumerate}}
\end{defi} 

  Applying Propositions 5.4 and 5.12, the value $2g^{\mathrm{an}}(C') - 2$ can be calculated in terms of 
    $l(e_p)$ as follows. 
       
 \begin{prop}
    Let the notation be as in Definition 5.13. We have that 
   \begin{align*}
      2g^{\mathrm{an}}(C') - 2 = \Sigma_{p \in \Upsilon_{C^{\mathrm{an}}}} n_p ((\Sigma_{e_p \in E_p} l(e_p)) - 2).   
   \end{align*} 
\end{prop}

\begin{prop} Let $p \in \Upsilon_{C^{\mathrm{an}}}$ and $e_p \in E_p$. The following equality holds.
   \begin{align*}
      l(e_p) = [k(C'):k(C)]/(n_p\widetilde{\mathrm{ram}}(e_p)). 
   \end{align*}
\end{prop}
\begin{proof}
     
     When $p$ is a point of type I or III, we must have that $l(e_p)$ is $1$
     and hence the proposition can be easily verified by applying Proposition 5.10. 
      Let us suppose that $p$ is a point of type II. 
    The morphism $\phi : C' \to C$ corresponds to an extension of function fields $k(C) \hookrightarrow k(C')$ which is Galois. 
As $p \in C^{\mathrm{an}}$ is of type II, it corresponds to 
a multiplicative norm on the function field $k(C)$. The set of preimages $\phi^{-1}(p)$ corresponds to 
those multiplicative norms on $k(C')$ which extend the multiplicative norm $p$ on $k(C)$. For every $p' \in (\phi^{\mathrm{an}})^{-1}(p)$, 
$\mathcal{H}(p')$ is the completion of $k(C')$ for $p'$ and is a finite extension of the non-Archimedean valued complete field 
$\mathcal{H}(p)$. The Galois group 
$\mathrm{Gal}(k(C')/k(C))$ acts transitively on the set $(\phi^{\mathrm{an}})^{-1}(p)$.
It follows that degree of the extension $[\mathcal{H}(p'):\mathcal{H}(p)]$ is a constant as $p'$ varies through the set
$(\phi^{\mathrm{an}})^{-1}(p)$. We denote this number $f(p)$. Hence we have that 
 \begin{align*} 
  [k(C'):k(C)] = n_pf(p).   
     \end{align*} 
  By Lemma 4.2, $f(p) = [\widetilde{\mathcal{H}(p')}:\widetilde{\mathcal{H}(p)}]$.
Uniquely associated to the $\tilde{k}$-function fields $\mathcal{H}(p)$ and $\mathcal{H}(p')$
 are smooth, projective    
$\tilde{k}$-curves denoted $\tilde{C}_p$ and $\tilde{C}'_{p'}$.
 The germ $e_p$ corresponds to a closed point on the former of these curves. The number $l(e_p)$   
is the cardinality of the set 
of preimages of the closed point $e_p$ for the morphism $\tilde{C}_{p'} \to \tilde{C}_p$ 
induced by $\phi^{\mathrm{an}}$. The result now follows from 
[\cite{Liu}, Theorem 7.2.18] applied to the $\tilde{k}$-function fields $\widetilde{\mathcal{H}(p)}, \widetilde{\mathcal{H}(p')}$ and 
the divisor $e_p$.  
\end{proof}

The results of this section can be compiled so that the value $2g^{\mathrm{an}}(C') - 2$ can be computed in terms of the
invariant $\widetilde{\mathrm{ram}}$ introduced below and the invariants $\mathrm{ram}$ and $c_1$ from Definition 5.11.

\begin{defi}(The invariant $\widetilde{\mathrm{ram}}(p)$ for $p \in \Upsilon_{C^{\mathrm{an}}}$)
\emph{Let $p \in \Upsilon_{C^{\mathrm{an}}}$. We define 
 $\widetilde{\mathrm{ram}}(p) := \Sigma_{e_p \in E_p} (1/ \widetilde{\mathrm{ram}}(e_p))$.}
\end{defi} 

    The following theorem can be verified using 5.14 and 5.10. 
 
\begin{thm}
    Let $\phi : C' \to C$
     be a finite morphism between smooth projective irreducible $k$-curves such that 
the extension of function fields $k(C) \hookrightarrow k(C')$ induced by $\phi$ is Galois. Let $g^{\mathrm{an}}(C')$ be as in Definition 2.26. 
For $p \in \Upsilon_{C^{\mathrm{an}}}$, let $\widetilde{\mathrm{ram}}(p), c_1(p)$ and $\mathrm{ram}(p)$ be the invariants introduced in Definition 5.6, 5.11
and 5.16.   
    We have that 
    \begin{align*} 
        2g^{\mathrm{an}}(C') - 2 = \mathrm{deg}(\phi) \Sigma_{p \in \Upsilon_{C^{\mathrm{an}}}} [\widetilde{\mathrm{ram}}(p) - 2/(c_1(p)\mathrm{ram}(p))].  
    \end{align*} 
\end{thm}

\section{Appendix} 
   
      At several instances over the course of this paper, we used the following fact.
      Let $\phi : C' \to C$ be a 
       finite surjective morphism 
       between $k$-curves where $C$ is in addition normal. 
       The induced morphism $\phi^{\mathrm{an}} :C'^{\mathrm{an}} \to C^{\mathrm{an}}$ is then open. 
        The following lemma 
      justifies this statement. 
 
  \begin{lem} 
Let $F$ be a non-Archimedean complete non trivially real valued field. 
  Let  $\phi : V \to W$ be a finite surjective morphism between irreducible $F$-varieties with $W$ normal. 
  The induced morphism $\phi^{\mathrm{an}} : V^{\mathrm{an}} \to W^{\mathrm{an}}$ is an open morphism. 
\end{lem} 
\begin{proof}
    We apply Lemma 3.2.4 in \cite{berk} to prove the lemma. 
   Clearly, we need only show that if $W$ is normal 
    then $W^{\mathrm{an}}$ is locally irreducible. By 3.4.3 in loc.cit, $W^{\mathrm{an}}$ is a normal $k$-analytic space. 
    Let $x \in W^{\mathrm{an}}$ and $U \subset W^{\mathrm{an}}$ be a $k$-analytic neighborhood of $x$. Let $U' \subset U$ be the connected 
    component that contains $x$. The space $U'$ is a normal $k$-analytic space. By 3.1.8 in loc.cit, it must be irreducible. This completes the proof.  
\end{proof}

               \end{document}